\documentclass[12pt,reqno]{amsart}

\usepackage{indentfirst}

\usepackage{amsmath,amssymb,amsthm}

\usepackage[dvipsnames]{xcolor}

\usepackage{newtxtext,newtxmath}

\usepackage[final]{microtype}

\UseMicrotypeSet[protrusion]{basicmath}

\usepackage[a4paper,margin=25mm]{geometry}

\usepackage{mathtools}

\usepackage{mathrsfs}

\usepackage{tikz-cd}

\usetikzlibrary{calc}

\usepackage[all]{xy}

\usepackage{float}

\usepackage{placeins}

\usepackage{graphicx}

\graphicspath{{figures/}{./}{../figures/}}

\DeclareGraphicsExtensions{.pdf,.png,.jpg,.jpeg}

\usepackage{caption}

\usepackage{subcaption}

\usepackage{array}

\usepackage{booktabs}

\usepackage{multicol,multirow}

\usepackage[toc,page]{appendix}

\usepackage{enumitem}

\newcolumntype{L}[1]{>{\raggedright\arraybackslash}p{#1}}

\newcolumntype{C}[1]{>{\centering\arraybackslash}m{#1}}

\newlength{\tblimgH}

\definecolor{tocblue}{RGB}{0,0,204}

\usepackage{hyperref}

\hypersetup{
colorlinks=true,
linktoc=all,
linkcolor=tocblue,
citecolor=tocblue,
urlcolor=tocblue
}

\usepackage{bookmark}

\numberwithin{equation}{section}

\makeatletter

\renewcommand{\section}{%
\@startsection{section}{1}{\z@}%
{-3.5ex \@plus -1ex \@minus -.2ex}%
{2.3ex \@plus .2ex}%
{\normalfont\Large\bfseries}%
}

\renewcommand{\subsection}{%
\@startsection{subsection}{2}{\z@}%
{1.2\baselineskip}%
{-0.6em}%
{\normalfont\bfseries}%
}

\makeatother

\newtheoremstyle{plainrunin}
{6pt}
{6pt}
{\itshape}
{}
{\bfseries}
{.}
{0.5em}
{}

\newtheoremstyle{uprunin}
{6pt}
{6pt}
{\normalfont}
{}
{\bfseries}
{.}
{0.5em}
{}

\theoremstyle{plainrunin}

\newtheorem{thm}{Theorem}[section]

\newtheorem{lemma}[thm]{Lemma}

\newtheorem{prop}[thm]{Proposition}

\newtheorem{corollary}[thm]{Corollary}

\newtheorem{claim}[thm]{Claim}

\theoremstyle{uprunin}

\newtheorem{defi}[thm]{Definition}

\newtheorem{remark}[thm]{Remark}

\newtheorem{notation}[thm]{Notation}

\newtheorem{maintheorem}{Theorem}

\setlist[enumerate]{label=\textup{(\roman*)},topsep=0pt,partopsep=0pt}

\newlist{steplist}{enumerate}{1}

\setlist[steplist]{
label=\textbf{Step (\roman*).},
ref=\roman*,
leftmargin=*,
itemsep=0.4ex,
topsep=0.6ex
}

\renewcommand{\abstractname}{Abstract}

\makeatletter

\def\l@section{%
\@tocline{1}
{0.45em plus 0.2em minus 0.1em}
{0pt}
{1.5em}
{\bfseries}%
}

\makeatother

\makeatletter

\renewenvironment{abstract}{%
\if@twocolumn
\section*{\abstractname}%
\else
\small
\vspace{0.4\baselineskip}
\begin{center}\bfseries\abstractname\end{center}%
\quotation
\fi
}{%
\if@twocolumn\else\endquotation\fi
}

\makeatother

\newcommand{\pr}{\mathrm{pr}}

\DeclareMathOperator{\Pic}{Pic}

\DeclareMathOperator{\Aff}{Aff}

\DeclareMathOperator{\cyc}{cyc}

\DeclareMathOperator{\Div}{Div}

\DeclareMathOperator{\Jac}{Jac}

\DeclareMathOperator{\rank}{rank}

\DeclareMathOperator{\id}{id}

\DeclareMathOperator{\Hom}{Hom_{\Z}}

\DeclareMathOperator{\im}{im}

\newcommand{\Z}{\mathbb{Z}}

\newcommand{\C}{\mathbb{C}}

\newcommand{\MO}{\mathcal{O}}

\newcommand{\Q}{\mathbb{Q}}

\newcommand{\R}{\mathbb{R}}

\newcommand{\relint}{\mathrm{relint}}

\makeatletter

\def\thm@space@setup{%
\thm@preskip=7pt plus 2pt minus 2pt%
\thm@postskip=7pt plus 2pt minus 2pt%
}

\makeatother

\title{Tropical Matsusaka and Matsusaka--Ran Criteria}

\author{Zhijie Ji}

\address{Department of Mathematics, Graduate School of Science, Kyoto University, Kyoto 606-8502, Japan}

\email{scfssb8@gmail.com, ji.zhijie.78n@st.kyoto-u.ac.jp}

\date{}

\begin{document}
\maketitle
\vspace{-2.2em}

\begin{abstract}
The tropical Schottky problem asks which principally polarized tropical abelian varieties are isomorphic to tropical Jacobian varieties equipped with their canonical principal polarizations. We prove the tropical Matsusaka criterion and the tropical Matsusaka--Ran criterion, which are the main results of this paper, and show that their hypotheses are equivalent. Together with the tropical Poincar\'e formula proved by Gross--Shokrieh, these results give a geometric solution to the tropical Schottky problem. In particular, we give an affirmative answer to a question posed by Brannetti--Melo--Viviani.

Several tools are developed for the proofs of the two criteria. In particular, we establish Hodge-type inequalities for ample tropical Cartier divisors on tropical abelian varieties and prove a splitting criterion for principally polarized tropical abelian varieties.
\end{abstract}

\tableofcontents

\section{Introduction}\label{sec:introduction}

\subsection{Background}
\begingroup
\setlength{\parindent}{1.5em}
\setlength{\parskip}{0.5\baselineskip}

\indent The classical Schottky problem asks for a characterization of principally polarized abelian varieties that are Jacobians of smooth projective curves equipped with their canonical principal polarizations; see \cite{Gru12} for a survey. We begin by recalling the following theorem. Throughout the theorem below and the following discussion, all varieties are defined over \(\C\).

\smallskip
\noindent\textbf{Theorem.}
Let \((A,\Theta_A)\) be a \(g\)-dimensional polarized abelian variety with \(g\ge 2\), and let \(i\colon Z\hookrightarrow A\) be an embedding of an irreducible projective curve \(Z\). Then the following are equivalent.

\begin{enumerate}[label=\textup{(\roman*)}, itemsep=2pt]
\item The curve \(Z\) generates \(A\), and \((\Theta_A\cdot Z)=g\).

\item The polarization \(\Theta_A\) is principal, and \((g-1)![Z]\equiv_{\hom}([\Theta_A])^{g-1}\).

\item The curve \(Z\) is smooth, and the homomorphism \(f\colon\Jac(Z)\to A\) induced by \(i\) is an isomorphism of principally polarized abelian varieties
\[
(\Jac(Z),\Theta_Z)\cong (A,\Theta_A),
\]
where \(\Theta_Z\) denotes the canonical principal polarization on \(\Jac(Z)\).
\end{enumerate}
\smallskip

For \(0\le p\le g\), let \(\operatorname{CH}^{p}(A)\) denote the Chow group of codimension-\(p\) algebraic cycles on \(A\) modulo rational equivalence. Here and below, for a divisor \(D\) and a curve \(Y\) on \(A\), the symbols \([D]\) and \([Y]\) denote their rational equivalence classes in \(\operatorname{CH}^{1}(A)\) and \(\operatorname{CH}^{g-1}(A)\), respectively. If \(Y\) is smooth, \(\Theta_Y\) denotes the canonical principal polarization on \(\Jac(Y)\). The symbol \(\equiv_{\hom}\) denotes homological equivalence on these Chow groups. The condition that \(Z\) generates \(A\) means that \(Z\) is not contained in any translate of a proper abelian subvariety of \(A\).

The implication \textup{(ii)}\(\Longrightarrow\)\textup{(iii)} is the Matsusaka criterion, proved by Matsusaka in \cite{Mat59}. The implication \textup{(i)}\(\Longrightarrow\)\textup{(iii)} is a special case of the Matsusaka--Ran criterion, proved by Ran in \cite{Ran81}; we use the formulation in \cite[Thm.~11.8.1]{BL04}. Conversely, the Poincar\'e formula \cite[Thm.~11.2.1]{BL04} gives both \textup{(iii)}\(\Longrightarrow\)\textup{(i)} and \textup{(iii)}\(\Longrightarrow\)\textup{(ii)}. Thus the three conditions are equivalent.

The Matsusaka criterion is commonly regarded as a geometric solution to the classical Schottky problem.

In tropical geometry, one can similarly consider principally polarized tropical abelian varieties \((X,\Theta)\). On \(X\), the tropical intersection theory developed in \cite{AR10} and the tropical homology theory, introduced in \cite{IKMZ19} and developed for integral tori in \cite{GS19}, allow one to study tropical Cartier divisors, tropical cycles, intersection numbers, and homological equivalence classes.

The study of tropical abelian varieties and tropical Jacobian varieties goes back to \cite{MZ08}; see also \cite{BF11}. The tropical Schottky problem asks for a characterization of principally polarized tropical abelian varieties that are tropical Jacobian varieties equipped with their canonical principal polarizations. Brannetti--Melo--Viviani formulated in \cite[Sec.~7]{BMV11} two expected tropical analogues of the classical results recalled above. The first asks for a tropical Poincar\'e formula. This was proved by Gross--Shokrieh in \cite[Thm.~A]{GS23}; see Theorem~\ref{Thm:Poincare} below. {{The second asks for the tropical Matsusaka criterion: let \((X,\Theta)\) be a \(g\)-dimensional principally polarized tropical abelian variety, and assume that there exists an effective tropical \(1\)-cycle \(C\) on \(X\) with connected support such that \((g-1)!C\equiv_{\hom}([\Theta])^{g-1}\).} Is there a smooth tropical curve \(\Gamma\) such that \((\Jac(\Gamma),\Theta_\Gamma)\cong(X,\Theta)\), where \(\Theta_\Gamma\) denotes the canonical principal polarization on \(\Jac(\Gamma)\)? Thus, as in the classical case, an affirmative answer to this question, together with the tropical Poincar\'e formula, will give a geometric solution to the tropical Schottky problem.}

\endgroup

\subsection{Our contribution}

\begingroup
\setlength{\parindent}{1.5em}
\setlength{\parskip}{0.5\baselineskip}

The two main results of this paper are the tropical Matsusaka criterion and the tropical Matsusaka--Ran criterion. Combining these results with the tropical Poincar\'e formula proved by Gross--Shokrieh and the equivalence between the hypotheses of the two criteria established below, we obtain the following result.

\begin{maintheorem}[=Theorem~\ref{thm:main-equivalence}]\label{thm:main-equivalent-criteria}
Let \((X,\Theta)\) be a \(g\)-dimensional polarized tropical abelian variety with \(g\ge 2\), and let \(C\) be an effective tropical \(1\)-cycle on \(X\) such that \(|C|\) is connected. Then the following are equivalent.
\begin{enumerate}[label=\textup{(\roman*)}]
\item The tropical \(1\)-cycle \(C\) is a spanning curve on \(X\), and \((\Theta\cdot C)_X=g\).
\item The polarization \(\Theta\) is principal, and \((g-1)!C\equiv_{\hom}([\Theta])^{g-1}\).
\item There exist a smooth tropical curve \(\Gamma\) of genus \(g\) and a morphism of rational polyhedral spaces \(\chi\colon\Gamma\to X\) such that \(\chi_*[\Gamma]=C\), and the homomorphism
\(
f\colon\Jac(\Gamma)\to X
\)
induced by \(\chi\) is an isomorphism of principally polarized tropical abelian varieties
\[
\bigl(\Jac(\Gamma),\Theta_\Gamma\bigr)\cong(X,\Theta),
\]
where \(\Theta_\Gamma\) denotes the canonical principal polarization on \(\Jac(\Gamma)\).
\end{enumerate}
\end{maintheorem}

Here \(|C|\) denotes the support of \(C\), as defined in Definition~\ref{def:tropical-cycle}, and a spanning curve is defined in Definition~\ref{def:spanning-curve}. The implication \textup{(ii)}\(\Longrightarrow\)\textup{(iii)} is the tropical Matsusaka criterion, Theorem~\ref{thm:tropical-matsusaka-criterion}, while the implication \textup{(i)}\(\Longrightarrow\)\textup{(iii)} is the tropical Matsusaka--Ran criterion, Theorem~\ref{thm:tropical-matsusaka-ran-criterion}. In particular, the tropical Matsusaka criterion gives an affirmative answer to the question posed in \cite[Sec.~7]{BMV11}.

We briefly describe the proof of Theorem~\ref{thm:main-equivalent-criteria} and record the main results used in it.

\smallskip
\noindent\textbf{The implication \textup{(ii)}\(\Longrightarrow\)\textup{(iii)}.}
This is the tropical Matsusaka criterion, Theorem~\ref{thm:tropical-matsusaka-criterion}. Its proof uses the following splitting criterion.

\begin{maintheorem}[=Proposition~\ref{prop:split-off-subppav}]\label{thm:main-splitting-criterion}
Let \((X,\Theta)\) be a principally polarized tropical abelian variety, and let \(i\colon Y\to X\) be an injective homomorphism of integral tori. Assume that \(\Theta_Y\coloneqq i^*\Theta\) is a principal polarization on \(Y\). Then there exist a principally polarized tropical abelian variety \((Z,\Theta_Z)\) and an isomorphism
\[
\beta\colon(Y,\Theta_Y)\times(Z,\Theta_Z)\longrightarrow(X,\Theta)
\]
of principally polarized tropical abelian varieties such that \(\beta\circ\iota_1=i\), where \(\iota_1\colon Y\to Y\times Z\) is the canonical inclusion of the first factor, defined in Subsection~\ref{subsec:products-integral-tori}.
\end{maintheorem}

\smallskip
\noindent\textbf{The equivalence \textup{(i)}\(\Longleftrightarrow\)\textup{(ii)}.}
This is Proposition~\ref{prop:equivalence-conditions-two-criteria}. Its proof uses the following Hodge-type inequalities. The term ``Hodge-type'' follows the classical terminology in \cite[Sec.~1.6]{Laz04}; this class of inequalities includes, for example, the Khovanskii--Teissier inequalities \cite[Ex.~1.6.4]{Laz04}.

\begin{maintheorem}[=Proposition~\ref{prop:tropical-KT} and Corollary~\ref{cor:TK-inequality}]\label{thm:main-hodge-type-inequalities}
Let \(X\) be a \(g\)-dimensional tropical abelian variety with \(g\ge 2\), and let \(D_1,\dots,D_g\) be ample tropical Cartier divisors on \(X\). Then the following hold.
\begin{enumerate}[label=\textup{(\roman*)}]
\item One has
\[
(D_1\cdot D_2\cdots D_g)_X\ge\prod_{i=1}^g (D_i^{g})_X^{1/g}.
\]
Equality holds if and only if there exist \(m_1,\dots,m_g\in\Z_{>0}\) such that
\[
m_1[D_1]\equiv_{\hom}m_2[D_2]\equiv_{\hom}\cdots\equiv_{\hom}m_g[D_g].
\]

\item For any two distinct integers \(i,j\in\{1,\ldots,g\}\),
\[
(D_1\cdots D_g)_X^2\ge(D_1\cdots D_{j-1}\cdot D_i\cdot D_{j+1}\cdots D_g)_X\cdot(D_1\cdots D_{i-1}\cdot D_j\cdot D_{i+1}\cdots D_g)_X.
\]
Equality holds if and only if there exist \(m_i,m_j\in\Z_{>0}\) such that \(m_i[D_i]\equiv_{\hom}m_j[D_j]\).
\end{enumerate}
\end{maintheorem}

Together with the tropical Matsusaka criterion, Theorem~\ref{thm:tropical-matsusaka-criterion}, the implication \textup{(i)}\(\Longrightarrow\)\textup{(ii)} gives the tropical Matsusaka--Ran criterion, Theorem~\ref{thm:tropical-matsusaka-ran-criterion}.

\smallskip
\noindent\textbf{The implication \textup{(iii)}\(\Longrightarrow\)\textup{(ii)}.}
This follows from the special case of the tropical Poincar\'e formula stated in Theorem~\ref{Thm:Poincare}, which was proved by Gross--Shokrieh. Together with the equivalence of \textup{(i)} and \textup{(ii)} and the implication \textup{(ii)}\(\Longrightarrow\)\textup{(iii)} above, this proves the equivalence of the three conditions in Theorem~\ref{thm:main-equivalent-criteria}.

\endgroup

\subsection{Structure of this paper}
After the introduction in Section~\ref{sec:introduction}, Section~\ref{sec:preliminaries} reviews background material on rational polyhedral spaces and integral tori.
Section~\ref{sec:hodge-type-inequalities} establishes the Hodge-type inequalities.
Section~\ref{sec:fin-dec} collects auxiliary results on homomorphisms of integral tori and their induced maps.
Section~\ref{sec:tropical-curves-morphisms} establishes Lemma~\ref{lem:cycle-from-smooth-curve}.
Section~\ref{sec:orthogonal-decompositions} establishes Proposition~\ref{prop:split-off-subppav} and treats orthogonal decompositions of tropical Jacobian varieties.
Section~\ref{sec:tropical-matsusaka} proves the tropical Matsusaka criterion, Theorem~\ref{thm:tropical-matsusaka-criterion}.
Finally, Section~\ref{sec:relation-two-criteria} proves the tropical Matsusaka--Ran criterion, Theorem~\ref{thm:tropical-matsusaka-ran-criterion}, and the main theorem, Theorem~\ref{thm:main-equivalence}.

\medskip
\subsection{Notations and conventions}\label{subsec:notations-conventions}
For any lattice \(\Lambda\), we write \((\Lambda)^{*}\coloneqq \Hom(\Lambda,\Z)\), \((\Lambda)_{\R}\coloneqq \Lambda\otimes_{\Z}\R\), and \((\Lambda)_{\R}^{*}\coloneqq \Hom(\Lambda,\R)\). If \(f\colon\Lambda_1\to\Lambda_2\) is a homomorphism of lattices, we also write \(f\colon(\Lambda_1)_{\R}\to(\Lambda_2)_{\R}\) for its \(\R\)-linear extension. For \(s\in\Z_{>0}\), we write \(M_s(\Z)\) and \(M_s(\R)\) for the sets of \(s\times s\) matrices with entries in \(\Z\) and \(\R\), respectively. We use these conventions throughout this paper.

\section{Preliminaries}\label{sec:preliminaries}

Following \cite{GS19,GS23,RZ25}, this section fixes notation and recalls the basic notions used throughout the paper. We collect standard definitions of rational polyhedral spaces, tropical cycles, tropical Cartier divisors, and integral tori.

\subsection{Rational polyhedral spaces and tropical curves}\label{subsec:rational-polyhedral-spaces-curves}
Write \(\langle\cdot,\cdot\rangle\) for the standard Euclidean inner product on any Euclidean space.
Let {\(n\in \Z_{\ge0}\)}. A \emph{rational polyhedron} in \(\R^n\) is a finite intersection of half-spaces \(\{\,\boldsymbol{x}\in\R^n\mid \langle \boldsymbol{a},\boldsymbol{x}\rangle\le c\,\}\) with \(\boldsymbol{a}\in\Z^n\) and \(c\in\R\). A \emph{rational polyhedral set} is a finite union of rational polyhedra. 
For a rational polyhedral set \(S\subseteq\R^n\) and an open subset \(V\subseteq S\), let \(\Aff_S(V)\) be the abelian group of continuous functions \(h\colon V\to\R\) with the following property. For every \(p\in V\), there exist an open neighbourhood \(W\subseteq V\) of \(p\), \(\boldsymbol{a}_W\in\Z^n\), and \(b_W\in\R\) such that \(h(\boldsymbol{x})=\langle\boldsymbol{a}_W,\boldsymbol{x}\rangle+b_W\) for every \(\boldsymbol{x}\in W\). These groups form a sheaf \(\Aff_S\) on \(S\).
\begin{defi}[{\cite[Def.~2.1]{GS23}}]\label{def:rational-polyhedral-space}
A \emph{rational polyhedral space} is a pair \((P,\Aff_P)\), where \(P\) is a second-countable Hausdorff space and \(\Aff_P\) is a sheaf of abelian groups of continuous functions on \(P\) with values in \(\R\), satisfying the following condition. Every \(p\in P\) has an open neighbourhood \(U\) such that, for some \(n\in\Z_{\ge0}\), there exist a rational polyhedral set \(S\subseteq\R^n\), an open subset \(V\subseteq S\), and a homeomorphism \(\varphi\colon U\to V\) whose pullback \(\varphi^*\colon\varphi^{-1}(\Aff_S|_V)\to\Aff_P|_U\) is an isomorphism of sheaves. Such a triple \((U,V,\varphi)\) is called a \emph{chart} of \(P\). An \emph{atlas} of \(P\) is a family of charts \(\{(U_\alpha,V_\alpha,\varphi_\alpha)\}_{\alpha\in I}\) such that \(P=\bigcup_{\alpha\in I}U_\alpha\). In this paper, we assume that there exist \(M\in\Z_{>0}\) and an atlas with \(V_\alpha \subseteq\R^{n_\alpha}\) and \(n_\alpha\in\{0,\ldots,M\}\) for all \(\alpha\in I\).
\end{defi}

\begin{defi}\label{def:affine-rational-functions}
Let \(P\) be a rational polyhedral space.

\begin{enumerate}
\item For an open set \(U\subseteq P\), the sections in \(\Aff_P(U)\) are called \emph{integral affine functions} on \(U\).

\item For an open set \(U\subseteq P\), let \(\mathcal M_P(U)\) be the abelian group of continuous functions \(\phi\colon U\to\R\) satisfying the following condition. For every \(p\in U\), there {exist} a chart \((U_\gamma,V_\gamma,\varphi_\gamma)\) with \(p\in U_\gamma\) and an open neighbourhood \(W\subseteq U\cap U_\gamma\) of \(p\) {satisfying the following conditions. We have \(V_\gamma\subseteq\R^{n_\gamma}\) for some {\(n_\gamma\in\Z_{\ge0}\)}. There exist} an open set \(O\subseteq\R^{n_\gamma}\) and rational polyhedra \(\Delta_1,\ldots,\Delta_k\subseteq\R^{n_\gamma}\) such that
\[
\varphi_\gamma(W)=O\cap\bigcup_{i=1}^k\Delta_i.
\]
{For} each \(i\), there exist \(\boldsymbol{a}_{\gamma,i}\in\Z^{n_\gamma}\) and \(b_{\gamma,i}\in\R\) such that
\[
(\phi\circ\varphi_\gamma^{-1})(\boldsymbol{x})=\langle\boldsymbol{a}_{\gamma,i},\boldsymbol{x}\rangle+b_{\gamma,i},
\]
for every \(\boldsymbol{x}\in\varphi_\gamma(W)\cap\Delta_i\). Elements of \(\mathcal M_P(U)\) are called \emph{tropical rational functions} on \(U\). As \(U\) varies, these abelian groups form a sheaf of abelian groups, denoted by \(\mathcal M_P\).
\end{enumerate}
\end{defi}

\begin{defi}
Let \(P\) be a rational polyhedral space. We say that \(P\) is a \emph{tropical manifold} if \(P\) admits an atlas \(\{(U_\alpha,V_\alpha,\varphi_\alpha)\}_{\alpha\in I}\) such that each \(V_\alpha\) is an open subset of a tropical linear space \(L_\alpha\) (cf.~\cite[Sec.~2D]{GS23}).
\end{defi}

\begin{defi}\label{def:rational-polyhedral-morphism}
Let \(P\) and \(Q\) be rational polyhedral spaces. A \emph{morphism of rational polyhedral spaces} \(f\colon P\to Q\) is a continuous map whose pullback induces a morphism of sheaves \(f^{-1}\Aff_Q\to \Aff_P\). We call \(f\) an \emph{isomorphism} if there exists a morphism of rational polyhedral spaces \(g\colon Q\to P\) such that \(g\circ f=\id_P\) and \(f\circ g=\id_Q\). The morphism \(f\) is \emph{proper} if \(f^{-1}(K)\subseteq P\) is compact for every compact subset \(K\subseteq Q\).
\end{defi}

\begin{defi}
Let \(P\) be a rational polyhedral space.
A point \(x\in P\) is called \emph{regular} if there exist an integer \(r\in\Z_{\ge0}\), an open neighbourhood \(U\subseteq P\) of \(x\), and an isomorphism of rational polyhedral spaces from \(U\) onto an open subset of \(\R^r\). In this case, the integer \(r\) is called the \emph{local dimension} of \(P\) at \(x\). We denote by \(P_{\mathrm{reg}}\) the set of regular points and call \(P_{\mathrm{reg}}\) the \emph{regular locus} of \(P\).
For \(d\in\Z_{\ge0}\), the rational polyhedral space \(P\) is called \emph{\(d\)-dimensional} if the maximum of the local dimensions of points of \(P_{\mathrm{reg}}\) is \(d\). The \(d\)-dimensional rational polyhedral space \(P\) is called \emph{purely \(d\)-dimensional} if every point of \(P_{\mathrm{reg}}\) has local dimension \(d\).
\end{defi}

\begin{defi}\label{def:tropical-curve}
A \emph{tropical curve} \(\Gamma\) is a connected compact purely \(1\)-dimensional rational polyhedral space. A \emph{global face structure} on \(\Gamma\) is given by choosing a finite subset \(S\subseteq \Gamma\), containing \(\Gamma\setminus\Gamma_{\mathrm{reg}}\), such that the closure of every connected component of \(\Gamma\setminus S\) is contained in a chart and is mapped onto a closed line segment. The points of \(S\) are called \emph{vertices}, and these closures are called \emph{edges}. Assume that \(\Gamma\) has a fixed global face structure.
\begin{enumerate}[label=\textup{(\roman*)}]
\item We write \(E(\Gamma)\) for the set of edges of \(\Gamma\) and \(V(\Gamma)\) for the set of vertices of \(\Gamma\). For an edge \(\tau\in E(\Gamma)\) with endpoints \(v\) and \(w\), define its \emph{relative interior} by
\(
\relint(\tau)\coloneqq \tau\setminus\{v,w\}.
\)
\item Throughout this paper, {we assume that} every global face structure on a tropical curve {has} been {refined} such that no edge has the same endpoint twice and any two vertices are joined by at most one edge. We then regard \(\Gamma\) as a simplicial complex, with vertices as \(0\)-simplices and edges as \(1\)-simplices. Hence the simplicial chain groups \(C_1(\Gamma,\Z)\), \(C_0(\Gamma,\Z)\), the boundary map \(\partial: C_1(\Gamma,\Z) \to C_0(\Gamma,\Z)\), and the homology group \(H_1(\Gamma,\Z)\) are all well-defined.
\item Now take an edge \(\tau \in E(\Gamma)\) with endpoints \(v\) and \(w\). We define an oriented \(1\)-simplex \(c_{\tau,v}\), whose underlying edge is \(\tau\), by requiring that \(\partial c_{\tau,v} = w - v\). This condition means that \(c_{\tau,v}\) is oriented from \(v\) to \(w\) {(or directed out of \(v\))}. We call \(\tau\) the \emph{support} of \(c_{\tau,v}\). We then define \(c_{\tau,w} \coloneqq -c_{\tau,v}\), and \(c_{\tau,w}\) is the oriented \(1\)-simplex pointing from \(w\) to \(v\) {(or directed out of \(w\))}, i.e., it has the \emph{opposite orientation} to \(c_{\tau,v}\). Indeed, \(\partial c_{\tau,w}= -\partial c_{\tau,v} = -(w-v) = v - w\).
\item We define \(g(\Gamma)\coloneqq \rank_{\Z}H_1(\Gamma,\Z)\), the genus of \(\Gamma\). This integer is independent of the global face structure on \(\Gamma\).
\end{enumerate}
\end{defi}

\begin{defi}\label{def:smooth-tropical-curve}
A tropical curve \(\Gamma\) with a fixed global face structure is \emph{smooth} if, at every vertex \(x\) of \(\Gamma\), there exist an integer \(d\in\Z_{>0}\) and a chart \((U,V,\varphi)\) with \(x\in U\) and \(\varphi\colon U\to V\subseteq \R^{d+1}/\R(1,\dots,1)\) such that \(\varphi(x)\) is the class of \(\boldsymbol{0}\), and \(V\) is an open subset of
\[
\bigcup_{i=1}^{d+1}\R_{\ge0}\,\overline{\boldsymbol{e}}_i
\subseteq
\R^{d+1}/\R(1,\dots,1)\cong \R^d.
\]
Here \(\boldsymbol{e}_1,\dots,\boldsymbol{e}_{d+1}\) are the standard basis of \(\R^{d+1}\). For each \(i\), the symbol \(\overline{\boldsymbol{e}}_i\) denotes the image of \(\boldsymbol{e}_i\) in \(\R^{d+1}/\R(1,\dots,1)\).
\end{defi}

\begin{remark}\label{rem:global-face-structures-on-supports}
Let \(\Gamma\) be a tropical curve. Since \(\Gamma\) is compact, choose finitely many charts \((U_i,V_i,\psi_i)\), \(1\le i\le N\), such that \(\Gamma=\bigcup_{i=1}^{N}U_i\). Choose a finite subset \(S\subseteq \Gamma\), containing \(\Gamma\setminus \Gamma_{\mathrm{reg}}\), such that the closure \(e\) of every connected component of \(\Gamma\setminus S\) is contained in some \(U_i\). Enlarge \(S\) by finitely many points of \(\Gamma_{\mathrm{reg}}\) so that, whenever \(e\subseteq U_i\), the image \(\psi_i(e)\subseteq V_i\) is a closed line segment. Taking all points of \(S\) as vertices and the closures of all connected components of \(\Gamma\setminus S\) as edges gives a global face structure on \(\Gamma\).
\end{remark}

\subsection{Integral tangent spaces}\label{subsec:integral-tangent-spaces}
\begin{defi}\label{def:integral-tangent-space}
Let \(P\) be a rational polyhedral space, and let \(\underline{\R}\) denote the constant sheaf with value \(\R\) on \(P\). Let \(\Aff_P\) be the sheaf of integral affine functions on \(P\). Via the natural inclusion of \(\underline{\R}\) into \(\Aff_P\), define \(\Omega_P\coloneqq \Aff_P/\underline{\R}\), the sheaf of tropical \(1\)-forms on \(P\). Fix \(x\in P\).
\begin{enumerate}[label=\textup{(\roman*)}]
\item {Let \(d_P\colon \Aff_P\to\Omega_P\) be the natural quotient morphism of sheaves.} For every \(\alpha\in\Omega_{P,x}\), there exist an open neighbourhood \(W\subseteq P\) of \(x\) and \(h\in\Aff_P(W)\) such that \((d_Ph)_x=\alpha\), where \((d_Ph)_x\) denotes the image of the germ \(h_x\in(\Aff_P)_x\) under the induced map \((d_P)_x\colon(\Aff_P)_x\to(\Omega_P)_x\). We call such an \(h\) a local representative of \(\alpha\) near \(x\).
\item Define the \emph{integral tangent space} at \(x\) by
\[
T_x^{\Z}P\coloneqq \Hom(\Omega_{P,x},\Z).
\]
The tangent space at \(x\) is its \(\R\)-extension
\(
T_xP\coloneqq T_x^{\Z}P\otimes_{\Z}\R.
\)
\item Let \(Q\) be another rational polyhedral space, and let \(f\colon P\to Q\) be a morphism of rational polyhedral spaces. The morphism \(f\) induces a morphism of sheaves \(f^{-1}\Omega_Q\to\Omega_P\). Taking the stalk at \(x\), we obtain the pullback on stalks
\(
f_x^*\colon \Omega_{Q,f(x)}\longrightarrow \Omega_{P,x}.
\)
We define 
\[d_x f\colon T_x^{\Z}P\to T_{f(x)}^{\Z}Q\]
 to be the dual map of \(f_x^*\). Hence, for every \(\xi\in T_x^{\Z}P\) and every \(\alpha\in\Omega_{Q,f(x)}\), one has
\begin{equation}\label{eq:differential-local-representative}
d_xf(\xi)(\alpha)
=
(\xi\circ f_x^*)(\alpha)
=
\xi\bigl(f_x^*(\alpha)\bigr).
\end{equation}

\end{enumerate}
\end{defi}

\begin{defi}\label{def:edge-tangent-length}
Let \(\Gamma\) be a tropical curve with a fixed global face structure. Let \(\Aff_{\Gamma}\) be the sheaf of integral affine functions on \(\Gamma\), and let \(\Omega_{\Gamma}\) be the sheaf of tropical \(1\)-forms on \(\Gamma\). Let \(d_{\Gamma}\colon \Aff_{\Gamma}\to\Omega_{\Gamma}\) be the natural quotient map. Let \(\tau\) be an edge of \(\Gamma\), and let \(v\) and \(w\) be the two endpoints of \(\tau\). Choose a chart \((U,V,\psi)\) with \(\tau\subseteq U\) and \(V\subseteq\R^n\) for some \(n\in\Z_{>0}\), and write \(x_v\coloneqq\psi(v)\) and \(x_w\coloneqq\psi(w)\). Write \(\langle\cdot,\cdot\rangle\) for the standard Euclidean inner product on \(\R^n\).
\begin{enumerate}[label=\textup{(\roman*)}]
\item There exists a unique primitive integral vector \(\boldsymbol{u}_{\tau,v}\in\Z^n\) along \(\psi(\tau)\) out of \(x_v\). Once the chart \((U,V,\psi)\) is fixed, we call \(\boldsymbol{u}_{\tau,v}\) the tangent direction of \(\tau\) out of \(v\) (in this chart).

For \(\alpha\in\Omega_{\Gamma,v}\), choose an open neighbourhood \(W\subseteq U\) of \(v\) and \(h_{\alpha}\in\Aff_{\Gamma}(W)\) such that \((d_{\Gamma}h_{\alpha})_v=\alpha\). After shrinking \(W\) if necessary, there exist \(\boldsymbol{m}_{\alpha}\in\Z^n\) and \(b_{\alpha}\in\R\) such that
\[
(h_{\alpha}\circ\psi^{-1})(\boldsymbol{y})
=
\langle\boldsymbol{m}_{\alpha},\boldsymbol{y}\rangle+b_{\alpha},
\qquad
\boldsymbol{y}\in\psi(W).
\]
Define \(\widetilde{\boldsymbol{u}}_{\tau,v}\in T_v^{\Z}\Gamma\) by
\[
\widetilde{\boldsymbol{u}}_{\tau,v}(\alpha)
\coloneqq
\langle\boldsymbol{m}_{\alpha},\boldsymbol{u}_{\tau,v}\rangle.
\]
This definition is independent of the chosen chart above and the chosen \(h_{\alpha}\). We call \(\widetilde{\boldsymbol{u}}_{\tau,v}\) the tangent functional of \(\tau\) at \(v\).

\item With \(\boldsymbol{u}_{\tau,v}\) as in \textup{(i)}, there exists a unique \(\ell(\tau)\in\R_{>0}\) such that
\(
x_w-x_v=\ell(\tau)\boldsymbol{u}_{\tau,v}.
\)
The number \(\ell(\tau)\) is independent of the chosen chart above and is called the length of \(\tau\).

\end{enumerate}
\end{defi}

\subsection{Tropical cycles}
Following \cite[Sec.~3.1]{GS19}, we recall the definitions of tropical cycles used in this paper. Let \(P\) be a nonempty rational polyhedral space and fix \(k\in\Z_{\ge0}\).
\begin{defi}\label{def:tropical-cycle}
A \emph{tropical \(k\)-cycle} on \(P\) is a \(\Z\)-valued function \(A\colon P\to \Z\) whose support
\(
|A|\coloneqq \overline{\{x\in P\mid A(x)\neq 0\}}
\)
is either empty or a purely \(k\)-dimensional rational polyhedral subspace of \(P\). If \(|A|\neq\varnothing\), then \(A\) is nonzero precisely on \(|A|_{\mathrm{reg}}\), is locally constant on \(|A|_{\mathrm{reg}}\), and satisfies the balancing condition: for every \(x\in P\), the local cone \(\mathrm{LC}_x(A)\), defined as in \cite[Sec.~3.1]{GS19}, is a balanced tropical fan \(k\)-cycle on \(T_xP\) in the sense of \cite[Def.~3.1]{GS19}. A tropical \(k\)-cycle \(A\) is \emph{effective} if \(A(x)\ge0\) for every \(x\in P\).

We write \(Z_k(P)\) for the abelian group of tropical \(k\)-cycles on \(P\). By convention, the zero function \(0\colon P\to\Z\) is a tropical \(k\)-cycle and is the zero element of \(Z_k(P)\). Let \(B,C\in Z_k(P)\), and define \(s\colon P\to\Z\) by \(s(x)=B(x)+C(x)\). In general, \(s\) need not be a tropical \(k\)-cycle. However, there exists a unique tropical \(k\)-cycle on \(P\) which agrees with \(s\) away from a locally polyhedral subset of \(P\) of dimension at most \(k-1\), where locally polyhedral subsets are understood in the sense of \cite[Def.~2.4\textup{(d)}]{GS19}. This tropical \(k\)-cycle is denoted by \(B+C\), and defines the sum of \(B\) and \(C\) in \(Z_k(P)\).
\end{defi}

\begin{defi}\label{def:tropical-1-cycle-weight}
Let \(C\) be a nonzero tropical \(1\)-cycle on a rational polyhedral space \(P\) such that \(|C|\) is connected and compact. Fix a global face structure on \(|C|\), refined if necessary, such that \(C\) is constant on \(\relint(\tau)\) for every edge \(\tau\) of \(|C|\). For each edge \(\tau\) of \(|C|\), define the \emph{weight} of \(\tau\) by
\[
\omega_C(\tau)\coloneqq C(x),
\qquad x\in\relint(\tau).
\]
In this paper, every nonzero tropical \(1\)-cycle with connected support is equipped with a fixed global face structure on its support satisfying the condition above.
\end{defi}

\begin{defi}\label{def:fundamental}
If \(P\) is purely \(n\)-dimensional, we say that \(P\) has a \emph{fundamental cycle} \([P]\) if the function which is \(1\) on \(P_{\mathrm{reg}}\) and \(0\) on \(P\setminus P_{\mathrm{reg}}\) defines a tropical \(n\)-cycle.
\end{defi}

\begin{defi}\label{def:proper-pushforward}
Let \(Q\) be a rational polyhedral space, and let \(f\colon P\to Q\) be a proper morphism of rational polyhedral spaces. Let \(A\in Z_k(P)\). If \(A=0\), we define \(f_*A\coloneqq 0\). Assume that \(A\neq0\).

Since \(|A|\) is a purely \(k\)-dimensional rational polyhedral subspace of \(P\), regard \(A\) as a tropical \(k\)-cycle on \(|A|\). Since \(f(|A|)\) is a rational polyhedral subspace of \(Q\), let
\[
\widetilde f\colon |A|\to f(|A|),
\]
be the restriction of \(f\). Then \(\widetilde f\) is proper and surjective. If \(\dim f(|A|)<k\), define \(f_*A\coloneqq0\). Assume that \(\dim f(|A|)=k\).
Write \(|A|_{\mathrm{reg}}\) and \(f(|A|)_{\mathrm{reg}}\) for the regular loci of \(|A|\) and \(f(|A|)\), respectively. For \(x\in |A|_{\mathrm{reg}}\), set
\begin{equation}\label{eq:local-index-mf}
m_{\widetilde f}(x)\coloneqq
\bigl[\,T^{\Z}_{\widetilde f(x)}f(|A|):d_x\widetilde f(T^{\Z}_x|A|)\bigr],
\end{equation}
where the index is understood to be \(0\) if infinite. For
\[
y\in f(|A|)\setminus
\bigl(\widetilde f(|A|\setminus |A|_{\mathrm{reg}})
\cup (f(|A|)\setminus f(|A|)_{\mathrm{reg}})\bigr)
\]
{such that the local dimension of \(f(|A|)\) at \(y\) is \(k\), define}
\begin{equation}\label{eq:proper-pushforward-value}
(\widetilde f_*A)(y)\coloneqq
\sum_{x\in \widetilde f^{-1}(\{y\})\cap |A|_{\mathrm{reg}}}
m_{\widetilde f}(x)\,A(x).
\end{equation}
This sum is finite. Indeed, nonzero contributions come only from isolated points of \(\widetilde f^{-1}(\{y\})\), and properness of \(\widetilde f\) gives only finitely many such points.
Extending this function by \(0\) to \(Q\) gives a \(\Z\)-valued function on \(Q\). We define \(f_*A\) to be the unique tropical \(k\)-cycle on \(Q\) that agrees with this function outside a locally polyhedral subset of \(Q\) of dimension at most \(k-1\), where locally polyhedral subsets are understood in the sense of \cite[Def.~2.4\textup{(d)}]{GS19}.
\end{defi}

\subsection{Tropical Cartier divisors, tropical line bundles and the first Chern class maps}\label{subsec:tropical-cartier-line-bundles}
Following \cite[Sec.~3.4]{GS19}, we recall tropical Cartier divisors and tropical line bundles. Let \(P\) and \(Q\) be rational polyhedral spaces, and let \(f\colon P\to Q\) be a morphism of rational polyhedral spaces.
For \(T\in\{P,Q\}\), let \(\Aff_T\) be the sheaf of integral affine functions on \(T\), and let \(\mathcal M_T\) be the sheaf of tropical rational functions on \(T\). Set \(\Div_T\coloneqq \mathcal M_T/\Aff_T\). Then there is a short exact sequence
\begin{equation}\label{eq:aff-M-div-short-exact}
0\longrightarrow \Aff_T\longrightarrow \mathcal M_T\longrightarrow \Div_T\longrightarrow 0.
\end{equation}
Set \(\Div(T)\coloneqq H^0(T,\Div_T)\). Elements of \(\Div(T)\) are called \emph{tropical Cartier divisors} on \(T\). Let \(q_T\colon \mathcal M_T\to \Div_T\) be the quotient map. For every open subset \(U\subseteq T\), let
\(
q_{T,\scriptscriptstyle U}\colon H^0(U,\mathcal M_T)\to H^0(U,\Div_T)
\)
be the homomorphism induced by \(q_T\).

The pullback of tropical rational functions and integral affine functions along \(f\) induces a homomorphism \(f^*\colon \Div(Q)\to \Div(P)\). We follow \cite[Sec.~4.A]{JRS18} and call elements of \(H^1(T,\Aff_T)\) \emph{tropical line bundles} on \(T\). We define \(\Pic(T)\coloneqq H^1(T,\Aff_T)\). The natural map \(d_T\colon \Aff_T\to \Omega_T\) induces the first Chern class map
\[
c_1\colon H^1(T,\Aff_T)\to H^1(T,\Omega_T),
\]
where \(H^1(T,\Omega_T)\) denotes the first sheaf cohomology of the abelian sheaf \(\Omega_T\). The pullback of integral affine functions along \(f\) induces a homomorphism
\begin{equation*}
f^*\colon \Pic(Q)=H^1(Q,\Aff_Q)\longrightarrow \Pic(P)=H^1(P,\Aff_P).
\end{equation*}
This homomorphism is called the pullback of tropical line bundles; see also \cite[Sec.~3.1]{Sum21}.

For \(T\in\{P,Q\}\), \eqref{eq:aff-M-div-short-exact} induces the connecting homomorphism
\[
\delta_T\colon H^0(T,\Div_T)\to H^1(T,\Aff_T).
\]
Let \(D\in \Div(T)=H^0(T,\Div_T)\). A tropical rational function \(\psi\in H^0(U,\mathcal M_T)\) on an open subset \(U\subseteq T\) is called a local representative of \(D\) on \(U\) if
\[
q_{T,\scriptscriptstyle U}(\psi)=D|_U
\quad\text{in }H^0(U,\Div_T).
\]
Set
\[
\MO_T(D)\coloneqq \delta_T(D).
\]

Let \(L\in\Pic(T)\). Choose an open covering \(\mathcal U_T=\{U_i\}_{i\in I}\) of \(T\) and a Čech \(1\)-cocycle \(a=\{a_{ij}\}\in Z^1(\mathcal U_T,\Aff_T)\) corresponding to \(L\). A global section of \(L\) with respect to this open covering is a family \(s=\{s_i\}_{i\in I}\) such that {\(s_i\in H^0(U_i,\mathcal M_T)\) for all \(i\in I\) and} \(s_j-s_i=a_{ij}\) on \(U_i\cap U_j\) {for} {all} {\(i,j\in I\). Then}
{\(
q_{T,\scriptscriptstyle U_i\cap U_j}(s_j|_{U_i\cap U_j})=q_{T,\scriptscriptstyle U_i\cap U_j}(s_i|_{U_i\cap U_j}).
\)}
Hence the local sections \(q_{T,\scriptscriptstyle U_i}(s_i)\) glue to a tropical Cartier divisor on \(T\).

\begin{defi}\label{def:pic-zero}
We define the group \(\Pic^{0}(P)\coloneqq \ker\bigl(c_1\colon \Pic(P)\to H^1(P,\Omega_P)\bigr)\).
\end{defi}

\subsection{Intersection numbers}\label{subsec:intersection-numbers}
Let \(P\) be a compact purely \(n\)-dimensional tropical manifold.

\begin{defi}\label{def:intersection-number}
By \cite[Def.~6.5]{AR10}, for each integer \(k\) satisfying \(1\le k\le n\), there is a bilinear intersection product map
\begin{equation*}
\Div(P)\times Z_k(P)\longrightarrow Z_{k-1}(P),
\qquad
(D,Y)\longmapsto D\cdot Y.
\end{equation*}
Let \(D_1,\dots,D_n\in \Div(P)\). For \(A\in Z_k(P)\) and \(1\le r\le k\), define \(D_1\cdots D_r\cdot A\in Z_{k-r}(P)\) inductively as follows. For \(r=1\), let \(D_1\cdot A\) be the intersection product above. For \(r\ge 2\), set \(D_1\cdots D_r\cdot A\coloneqq D_1\cdot(D_2\cdots D_r\cdot A)\).
Let \(\deg\colon Z_0(P)\to \Z\) be defined by \(\deg(\sum_i a_i p_i)=\sum_i a_i\), where \(a_i\in\Z\) and \(p_i\in P\). In the case \(A\in Z_k(P)\), define
\begin{equation*}
(D_1\cdots D_k\cdot A)_P\coloneqq \deg(D_1\cdots D_k\cdot A).
\end{equation*}
Let \([P]\) denote the fundamental cycle of the compact tropical manifold \(P\). We write
\begin{equation*}
(D_1\cdots D_n)_P\coloneqq (D_1\cdots D_n\cdot [P])_P.
\end{equation*}
This is called the intersection number of \(D_1,\dots,D_n\) on \(P\). For \(D'\in \Div(P)\) and \(1\le r\le n\), define
\[
([D'])^r\coloneqq \underbrace{D'\cdots D'}_{r\ \mathrm{times}}\cdot [P],
\]
in \(Z_{n-r}(P)\). In particular, set \([D']\coloneqq [D']^1\coloneqq D'\cdot [P]\in Z_{n-1}(P)\). By convention, set \(([D'])^0\coloneqq[P]\).
\end{defi}

\subsection{Integral tori and polarizations}\label{subsec:integral-tori-polarizations}
Following \cite{RZ25}, this subsection recalls the notation and basic definitions for integral tori.

An \emph{integral torus} is a triple \(X=(\Lambda_X,\Lambda_X^{\prime},[\,\cdot,\cdot\,]_X)\), where \(\Lambda_X\) and \(\Lambda_X^{\prime}\) are lattices of the same rank and \([\,\cdot,\cdot\,]_X\colon \Lambda_X\times \Lambda_X^{\prime}\to \R\) is a nondegenerate bilinear pairing. Via \(\lambda^{\prime}\mapsto [\,\cdot,\lambda^{\prime}\,]_X\), we regard \(\Lambda_X^{\prime}\) as a lattice in \((\Lambda_X)_{\R}^{*}\), and we also write
\(
X= (\Lambda_X)_{\R}^{*}/\Lambda_X^{\prime}
\)
for the underlying real torus. Throughout this paper, we tacitly use these identifications. The \emph{dual integral torus} of \(X\) is \(X^{\vee}\coloneqq (\Lambda_X^{\prime},\Lambda_X,[\,\cdot,\cdot\,]_{X^{\vee}})\), with \([\lambda^{\prime},\lambda]_{X^{\vee}}\coloneqq [\lambda,\lambda^{\prime}]_X\) for \(\lambda^{\prime}\in \Lambda_X^{\prime}\) and \(\lambda\in \Lambda_X\).

Let \(\Theta\) be a tropical Cartier divisor on \(X\). By \cite[Sec.~5.1]{MZ08}, one has a canonical isomorphism \(H^1(X,\Omega_X)\cong \Hom(\Lambda_X^{\prime},\Lambda_X)\). Let \(\kappa_{\Theta}\colon \Lambda_X^{\prime}\to \Lambda_X\) be the homomorphism corresponding to \(c_1\bigl(\MO_X(\Theta)\bigr)\in H^1(X,\Omega_X)\). Define
\[
Q_{\Theta}\colon (\Lambda_X^{\prime})_{\R}\times (\Lambda_X^{\prime})_{\R}\to \R,\qquad
Q_{\Theta}(u,v)\coloneqq [\kappa_{\Theta}(u),v]_X=\kappa_{\Theta}(u)(v).
\]
By the tropical Appell--Humbert theorem \cite[Thm.~7.2]{GS23}, the bilinear form corresponding to \(c_1\bigl(\MO_X(\Theta)\bigr)\) is symmetric, i.e., \(Q_{\Theta}(u,v)=Q_{\Theta}(v,u)\) for all \(u,v\in(\Lambda_X^{\prime})_{\R}\). We say that \(\Theta\) is \emph{ample} if \(Q_{\Theta}\) is positive definite. We call \(\Theta\) a \emph{polarization} if \(\Theta\) is ample, and a \emph{principal polarization} if \(\Theta\) is a polarization and \(\kappa_{\Theta}\colon \Lambda_X^{\prime}\to \Lambda_X\) is an isomorphism. An integral torus \(X\) is called a \emph{tropical abelian variety} if it admits a polarization, and a pair \((X,\Theta)\) is called a \emph{polarized tropical abelian variety} (respectively, a \emph{principally polarized tropical abelian variety}) if \(\Theta\) is a polarization (respectively, a principal polarization).
\begin{remark}\label{rem:integral-torus-cohomology-description}
Let \(X=(\Lambda_X,\Lambda_X^{\prime},[\,\cdot,\cdot\,]_X)\) be an integral torus. Under the canonical identifications \(H^0(X,\Omega_X)\cong \Lambda_X\) and \(H_1(X,\Z)\cong \Lambda_X^{\prime}\), we also write \(X=\bigl(H^0(X,\Omega_X),H_1(X,\Z),[\,\cdot,\cdot\,]_X\bigr)\).
\end{remark}

\begin{defi}\label{def:translation-chart}

Let \(n=\dim X\), and let \(\pi_X\colon(\Lambda_X)_{\R}^{*}\to X\) be the quotient map. Choose a \(\Z\)-basis \(\{\eta_i^*\}_{i=1}^{n}\) of \((\Lambda_X)^{*}\). Let
\[
\iota\colon(\Lambda_X)_{\R}^{*}\to\R^n,\qquad
\sum_{i=1}^{n}z_i\eta_i^*\longmapsto (z_1,\ldots,z_n),
\quad z_i\in\R.
\]
Choose an open neighborhood \(\widetilde U_0\) of the origin in \((\Lambda_X)_{\R}^{*}\) such that \(\pi_X|_{\widetilde U_0}\) is injective. Set
\[
U_0\coloneqq\pi_X(\widetilde U_0),
\qquad
V_0\coloneqq\iota(\widetilde U_0),
\]
and define
\[
\psi_0\colon U_0\to V_0,\qquad
\psi_0\coloneqq
\iota|_{\widetilde U_0}\circ
\bigl(\pi_X|_{\widetilde U_0}\bigr)^{-1}.
\]
Then \((U_0,V_0,\psi_0)\) is a chart at the origin \(0_X\). For \(x\in X\), the \emph{translation by \(x\)} is the morphism \(t_x\colon X\to X\) defined by \(t_x(y)\coloneqq y+x\) for \(y\in X\). Set \(U_x\coloneqq t_x(U_0)\) and define \(\psi_x\colon U_x\to V_0\) by \(\psi_x(y)\coloneqq \psi_0(y-x)\). In this paper, a chart of \(X\) means one of the charts \((U_x,V_0,\psi_x)\) just defined.

In this paper, {let \(\Gamma\) be} any tropical curve in \(X\){.} {By} Remark~\ref{rem:global-face-structures-on-supports}, choose a global face structure on \(\Gamma\) and finitely many charts
\((U_{x_0},V_0,\psi_{x_0}),\ldots,(U_{x_N},V_0,\psi_{x_N})\)
of \(X\), with \(x_0=0_X\), such that every edge of \(\Gamma\) is contained in one of the \(U_{x_j}\). In particular, via the restriction \(\iota|_{(\Lambda_X)^*}\colon(\Lambda_X)^*\cong\Z^n\), we regard tangent directions of edges of \(\Gamma\) as elements of \((\Lambda_X)^*\).

\end{defi}

\subsection{Homomorphisms of integral tori}
Let \(X_i=(\Lambda_{X_i},\Lambda_{X_i}^{\prime},[\,\cdot,\cdot\,]_{X_i})\) be integral tori for \(i=1,2\). A \emph{homomorphism of integral tori} \(f=(f^\sharp,f_\sharp)\colon X_1\to X_2\) consists of lattice homomorphisms
\(
f^\sharp\colon \Lambda_{X_2}\to \Lambda_{X_1}
\) and \(
f_\sharp\colon \Lambda_{X_1}^{\prime}\to \Lambda_{X_2}^{\prime},
\)
such that
\begin{equation}\label{eq:dual-hom-pairing-compat}
[f^\sharp(\lambda_2),\lambda_1^{\prime}]_{X_1}=[\lambda_2,f_\sharp(\lambda_1^{\prime})]_{X_2},
\end{equation}
for all \(\lambda_2\in \Lambda_{X_2}\) and \(\lambda_1^{\prime}\in \Lambda_{X_1}^{\prime}\). When \(X_1\) and \(X_2\) are regarded as real tori, we also write \(f\colon X_1\to X_2\) for the homomorphism induced by \(f_\sharp\). We write \(\Hom(X_1,X_2)\) for the set of homomorphisms from \(X_1\) to \(X_2\).
The \emph{dual homomorphism} of \(f\) is
\(
f^{\vee}\coloneqq (f_\sharp,f^\sharp)\colon X_2^{\vee}\to X_1^{\vee}.
\)

We call \(f\) \emph{surjective} if for every \(x_2\in X_2\), there exists \(x_1\in X_1\) such that \(f(x_1)=x_2\). We call \(f\) \emph{finite} if \(f^{-1}(\{0_{X_2}\})\) is finite, and \emph{injective} if \(f^{-1}(\{0_{X_2}\})=\{0_{X_1}\}\), where \(0_{X_i}\) denotes the origin of \(X_i\). We call \(f\) an \emph{isomorphism} if there exists a homomorphism of integral tori \(h\colon X_2\to X_1\) such that \(h\circ f=\id_{X_1}\) and \(f\circ h=\id_{X_2}\). See also \cite[Def.~4.8]{RZ25} for equivalent definitions of these notions.

\begin{lemma}[{\cite[Sec.~4]{RZ25}}]\label{lem:homomorphism-basic}
Follow the notation above.
\begin{enumerate}[label=\textup{(\roman*)}]
\item \(f\) is surjective if and only if
\(f^\sharp\colon \Lambda_{X_2}\to \Lambda_{X_1}\)
is injective.
\item Let \(X_3\) be an integral torus and let
\(g\colon X_2\to X_3\) be a homomorphism of integral tori. Then
\[
(g\circ f)^{\vee}=f^{\vee}\circ g^{\vee}\colon X_3^{\vee}\to X_1^{\vee}.
\]
\end{enumerate}
\end{lemma}

\subsection{Products of integral tori}\label{subsec:products-integral-tori}
Let \(X_i=(\Lambda_{X_i},\Lambda_{X_i}^{\prime},[\,\cdot\,,\,\cdot\,]_{X_i})\) be integral tori for \(i=1,2\). The \emph{product} of \(X_1\) and \(X_2\) is
\[
X_1\times X_2\coloneqq \bigl(\Lambda_{X_1}\oplus\Lambda_{X_2},\Lambda_{X_1}^{\prime}\oplus\Lambda_{X_2}^{\prime},[\,\cdot,\cdot\,]_{X_1\times X_2}\bigr),
\]
where, for \(\lambda_i\in \Lambda_{X_i}\) and \(\lambda_i^{\prime}\in \Lambda_{X_i}^{\prime}\) with \(i=1,2\),
\begin{equation}\label{eq:product-pairing}
{[(\lambda_1,\lambda_2),(\lambda_1^{\prime},\lambda_2^{\prime})]_{X_1\times X_2}
\coloneqq
[\lambda_1,\lambda_1^{\prime}]_{X_1}+[\lambda_2,\lambda_2^{\prime}]_{X_2}.}
\end{equation}
We define the canonical projection of integral tori
\(\pr_1\colon X_1\times X_2\to X_1\) by
\[
(\pr_1)^{\sharp}(\lambda_1)=(\lambda_1,0),
\qquad
(\pr_1)_{\sharp}(\lambda_1^{\prime},\lambda_2^{\prime})
=\lambda_1^{\prime},
\]
for \(\lambda_1\in\Lambda_{X_1}\) and
\(\lambda_i^{\prime}\in\Lambda_{X_i}^{\prime}\) with \(i=1,2\).
We also define the canonical inclusion of integral tori
\(\iota_1\colon X_1\to X_1\times X_2\) by
\[
(\iota_1)^{\sharp}(\lambda_1,\lambda_2)=\lambda_1,
\qquad
(\iota_1)_{\sharp}(\lambda_1^{\prime})
=(\lambda_1^{\prime},0),
\]
for \(\lambda_i\in\Lambda_{X_i}\) with \(i=1,2\), and
\(\lambda_1^{\prime}\in\Lambda_{X_1}^{\prime}\).
The homomorphisms
\(\pr_2\colon X_1\times X_2\to X_2\) and
\(\iota_2\colon X_2\to X_1\times X_2\)
are defined similarly.

\subsection{Tropical \textup{(co)}homology, cycle class maps, and intersection numbers on integral tori}\label{subsec:tropical-cohomology}
The general definitions of tropical homology groups and tropical cohomology groups are not
recalled here; see \cite[Sec.~2]{IKMZ19}. We use only the notation and basic
properties for these groups in the case of integral tori. Let
\(X=(\Lambda_X,\Lambda_X^{\prime},[\,\cdot\,,\,\cdot\,]_X)\) be a
\(g\)-dimensional integral torus, and fix \(p,q,p',q'\in\Z_{\ge0}\). Let
\(H^{p,q}(X)\) and \(H_{p,q}(X)\) denote the tropical cohomology and
tropical homology groups of bidegree \((p,q)\), respectively. One has
\(H^{p,q}(X)=H_{p,q}(X)=0\) if \(p\in\Z_{>g}\) or \(q\in\Z_{>g}\).
Let \(\Omega_X\) denote the sheaf of tropical \(1\)-forms on \(X\), and set
\[
\Omega_X^0\coloneqq \underline{\Z},
\qquad
{\Omega_X^p\coloneqq \bigwedge^p\Omega_X,
\quad p\in\Z_{\ge1},}
\]
where \(\underline{\Z}\) denotes the constant sheaf on \(X\) with value
\(\Z\). There exists a natural isomorphism
\(
H^{p,q}(X)\cong H^q(X,\Omega_X^p),
\)
where the group on the right is the \(q\)-th sheaf cohomology group of the
abelian sheaf \(\Omega_X^p\). The natural cycle class map in dimension \(p\)
is denoted by
\[
\cyc_X\colon Z_p(X)\to H_{p,p}(X).
\]
For the construction of this map, see \cite[Sec.~5.2]{GS19}. In particular, since \(X\) is
connected, \(H_{0,0}(X)\cong\Z\), and for every point \(x\in X\), the class
\(\cyc_X(x)\) corresponds to \(1\in\Z\).

\begin{defi}\label{def:homological-equivalence}
For \(A,B\in Z_p(X)\), we write \(A\equiv_{\hom}B\) if
\(\cyc_X(A)=\cyc_X(B)\) in \(H_{p,p}(X)\).
\end{defi}

\begin{prop}[{\cite[Prop.~5.12]{GS19}}]\label{prop:cyc-compat}
Let \(p\in\Z_{\ge1}\), \(D\in\Div(X)\), and \(A\in Z_p(X)\). One has
\(
\cyc_X(D\cdot A)=c_1(\MO_X(D))\cap\cyc_X(A)
\)
in \(H_{p-1,p-1}(X)\).
\end{prop}

The following canonical isomorphisms and descriptions of the cup and cap
products are from \cite[Sec.~6]{GS23}. One has the canonical isomorphisms
\[
H_{p,q}(X)
\cong
\bigwedge^{q}\Lambda_X^{\prime}
\otimes_{\Z}
\bigwedge^{p}(\Lambda_X)^{*}
\]
and
\[
H^{p,q}(X)
\cong
\bigwedge^{q}(\Lambda_X^{\prime})^{*}
\otimes_{\Z}
\bigwedge^{p}\Lambda_X.
\]
Under these isomorphisms, the cup product
\[
H^{p,q}(X)\times H^{p',q'}(X)
\longrightarrow
H^{p+p',q+q'}(X)
\]
is given by
\begin{equation}\label{eq:cup-product-integral-tori}
(\alpha\otimes\omega)\cup(\beta\otimes\xi)
=
(\alpha\wedge\beta)\otimes(\omega\wedge\xi),
\end{equation}
where
\(\alpha\in\bigwedge^{q}(\Lambda_X^{\prime})^{*}\),
\(\beta\in\bigwedge^{q'}(\Lambda_X^{\prime})^{*}\),
\(\omega\in\bigwedge^{p}\Lambda_X\), and
\(\xi\in\bigwedge^{p'}\Lambda_X\).
If \(p'\le p\) and \(q'\le q\), the cap product
\[
H^{p',q'}(X)\times H_{p,q}(X)
\longrightarrow
H_{p-p',q-q'}(X)
\]
is given by
\begin{equation}\label{eq:cap-product-integral-tori}
(\varphi\otimes u)\cap(\lambda\otimes\eta)
=
(\varphi\mathbin{\lrcorner}\lambda)
\otimes
(u\mathbin{\lrcorner}\eta),
\end{equation}
where
\(\varphi\in\bigwedge^{q'}(\Lambda_X^{\prime})^{*}\),
\(u\in\bigwedge^{p'}\Lambda_X\),
\(\lambda\in\bigwedge^{q}\Lambda_X^{\prime}\), and
\(\eta\in\bigwedge^{p}(\Lambda_X)^{*}\), where \(\lrcorner\) denotes the
interior product on exterior algebras. In particular,
\[
H^{1,1}(X)
\cong
(\Lambda_X^{\prime})^{*}\otimes_{\Z}\Lambda_X
\cong
\Hom(\Lambda_X^{\prime},\Lambda_X).
\]

Fix \(\Z\)-bases
\(\{\lambda_i^{\prime}\}_{i=1}^{g}\) of \(\Lambda_X^{\prime}\) and
\(\{\eta_j\}_{j=1}^{g}\) of \(\Lambda_X\), with dual bases
\(\{(\lambda_i^{\prime})^{*}\}_{i=1}^{g}\) and
\(\{\eta_j^{*}\}_{j=1}^{g}\), respectively. The proof of
\cite[Lem.~6.7]{GS25} gives
\begin{equation}\label{eq:fundamental-cycle-class}
\cyc_X([X])
=
\varepsilon\,
(\lambda_1^{\prime}\wedge\cdots\wedge\lambda_g^{\prime})
\otimes
(\eta_1^{*}\wedge\cdots\wedge\eta_g^{*}),
\end{equation}
for some \(\varepsilon\in\{\pm1\}\), where \([X]\) denotes the fundamental
cycle. It follows from \eqref{eq:cap-product-integral-tori} and \eqref{eq:fundamental-cycle-class}
that the map
\[
H^{1,1}(X)
\longrightarrow
H_{g-1,g-1}(X),
\qquad
\alpha\longmapsto\alpha\cap\cyc_X([X]),
\]
is an isomorphism. Hence, by Proposition~\ref{prop:cyc-compat}, for
\(D_1,D_2\in\Div(X)\), one has
\begin{equation}\label{eq:chern-class-homological-equivalence}
[D_1]\equiv_{\hom}[D_2]
\quad\Longleftrightarrow\quad
c_1(\MO_X(D_1))=c_1(\MO_X(D_2)).
\end{equation}
For \(D\in\Div(X)\), write
\begin{equation}\label{eq:chern-class-matrix}
c_1(\MO_X(D))
=
\sum_{i,j=1}^{g}
E^{D}_{i,j}\,
(\lambda_i^{\prime})^{*}\otimes\eta_j,
\qquad
E^{D}=(E^{D}_{i,j})\in M_g(\Z).
\end{equation}
Under
\(H^{1,1}(X)\cong\Hom(\Lambda_X^{\prime},\Lambda_X)\), the class
\(c_1(\MO_X(D))\) corresponds to the lattice homomorphism
\begin{equation}\label{eq:kappaD-matrix}
\kappa_D\colon\Lambda_X^{\prime}\longrightarrow\Lambda_X,
\qquad
\lambda_i^{\prime}
\longmapsto
\sum_{j=1}^{g}E^{D}_{i,j}\eta_j,
\quad 1\le i\le g.
\end{equation}
Thus one has
\(
E^{D}
=
\bigl(Q_D(\lambda_i^{\prime},\eta_j^{*})\bigr)_{1\le i,j\le g},
\)
where \(Q_D\) is defined as in Subsection~\ref{subsec:integral-tori-polarizations}.

For \(D_1,\dots,D_g\in\Div(X)\), choose a point \(x\in X\). Since \(X\) is
connected, \(\cyc_X(x)\) is a generator of \(H_{0,0}(X)\cong\Z\). Hence
there exists a unique \(n\in\Z\) such that
\begin{equation}\label{eq:chern-cup-cap-intersection}
c_1(\MO_X(D_1))
\cup\cdots\cup
c_1(\MO_X(D_g))
\cap
\cyc_X([X])
=
n\cyc_X(x).
\end{equation}
By Definition~\ref{def:intersection-number} and
Proposition~\ref{prop:cyc-compat}, we have
\((D_1\cdot\cdots\cdot D_g)_X=n\).
Let \(S_g\) denote the symmetric group on \(\{1,\ldots,g\}\). Using
\eqref{eq:cup-product-integral-tori} and \eqref{eq:cap-product-integral-tori} to compute the left-hand side of
\eqref{eq:chern-cup-cap-intersection}, we obtain
\begin{equation}\label{eq:intersection-number-matrix}
(D_1\cdot\cdots\cdot D_g)_X
=
\varepsilon
\sum_{\sigma,\tau\in S_g}
\operatorname{sgn}(\sigma)
\operatorname{sgn}(\tau)
\prod_{\ell=1}^{g}
E^{D_\ell}_{\sigma(\ell),\tau(\ell)}.
\end{equation}

\begin{lemma}\label{lem:top-self-intersection-det}
Let \(D\in\Div(X)\). Let \(\varepsilon\in\{\pm1\}\) be the sign fixed in
\eqref{eq:fundamental-cycle-class}. Then the following hold.
\begin{enumerate}[label=\textup{(\roman*)}, itemsep=2pt]
\item One has
\(
(D^{g})_X=\varepsilon\,g!\,\det(E^{D}).
\)
\item If \(D\) is ample, then
\(
(D^{g})_X=g!\,\lvert\det(E^{D})\rvert.
\)
In particular, \((D^{g})_X\ge g!\). Moreover, if \(\det(E^D)=1\), then
\(\varepsilon=1\).
\item If \(D\) is ample, then \(D\) defines a principal polarization on
\(X\) if and only if \((D^{g})_X=g!\).
\end{enumerate}
\end{lemma}

\begin{proof}
The assertion \textup{(i)} follows from
\eqref{eq:intersection-number-matrix}. The equality in \textup{(ii)}
follows from \cite[Thm.~47]{Sum21}. If \(\det(E^D)=1\), then \textup{(i)}
gives \((D^g)_X=\varepsilon g!\), while the equality in \textup{(ii)}
gives \((D^g)_X=g!\). Hence \(\varepsilon=1\).

We prove \textup{(iii)}. Assume that \(D\) is ample. With respect to the
chosen \(\Z\)-bases, the homomorphism
\(\kappa_D\colon\Lambda_X^{\prime}\to\Lambda_X\) has matrix \(E^D\).
Hence \(\kappa_D\) is an isomorphism if and only if
\(E^D\in\mathrm{GL}_g(\Z)\). Equivalently,
\(\lvert\det(E^D)\rvert=1\). By \textup{(ii)}, this condition is equivalent
to \((D^g)_X=g!\). Thus \(D\) defines a principal polarization on \(X\) if
and only if \((D^g)_X=g!\).
\end{proof}

Let
\(
X_i=
(\Lambda_{X_i},\Lambda_{X_i}^{\prime},
[\,\cdot\,,\,\cdot\,]_{X_i})
\)
be an integral torus for \(i=1,2\), and let
\(
f=(f^{\sharp},f_{\sharp})\colon X_1\to X_2
\)
be a homomorphism of integral tori. For \(p,q\in\Z_{\ge0}\), we use the
induced maps
\(
f^{*}\colon H^{p,q}(X_2)\to H^{p,q}(X_1)
\)
and
\(
f_{*}\colon H_{p,q}(X_1)\to H_{p,q}(X_2);
\)
for details, see \cite[Sec.~4]{GS19}.

\begin{prop}[{\cite[Secs.~4 and 5]{GS19}}]\label{prop:GS19-toolkit}
Follow the notation above. The following hold.
\begin{enumerate}[label=\textup{(\roman*)}, itemsep=2pt]
\item Let \(c\in H^{p,q}(X_2)\) and
\(c'\in H^{p',q'}(X_2)\). One has
\(
f^{*}c\cup f^{*}c'
=
f^{*}(c\cup c')
\).
\item Assume that \(p'\le p\) and \(q'\le q\). Let
\(\alpha\in H_{p,q}(X_1)\) and
\(c^{\prime}\in H^{p',q'}(X_2)\). One has
\(
f_{*}(f^{*}c^{\prime}\cap\alpha)
=
c^{\prime}\cap f_{*}\alpha
\).
\item Let \(A\in Z_p(X_1)\). One has
\(
\cyc_{X_2}(f_*A)
=
f_*\cyc_{X_1}(A)
\).
\end{enumerate}
\end{prop}

\begin{remark}\label{rem:pullback-one-one}
Follow the notation above. For \(i=1,2\), the standard identifications in
Subsection~\ref{subsec:tropical-cohomology} give the following commutative diagram:
\begin{equation}\label{eq:cup-one-one-identifications}
\begin{tikzcd}[column sep=1.4em]
H^{0,1}(X_i)
\arrow[d, "\cong"']
\arrow[r, phantom, "\times" description]
&
H^{1,0}(X_i)
\arrow[d, "\cong"']
\arrow[r, "\cup"]
&
H^{1,1}(X_i)
\arrow[d, "\cong"]
\\
(\Lambda_{X_i}^{\prime})^{*}
\arrow[r, phantom, "\times" description]
&
\Lambda_{X_i}
\arrow[r]
&
(\Lambda_{X_i}^{\prime})^{*}\otimes\Lambda_{X_i}.
\end{tikzcd}
\end{equation}
The lower horizontal map in \eqref{eq:cup-one-one-identifications} is
induced by the cup product in \eqref{eq:cup-product-integral-tori}. The maps
\(
f^{*}\colon H^{0,1}(X_2)\to H^{0,1}(X_1)
\)
and
\(
f^{*}\colon H^{1,0}(X_2)\to H^{1,0}(X_1)
\)
fit into the following commutative diagrams:
\begin{equation}\label{eq:pullback-one-zero-zero-one-lattice}
\begin{tikzcd}[column sep=large]
H^{0,1}(X_2)
\arrow[r, "f^{*}"]
\arrow[d, "\cong"']
&
H^{0,1}(X_1)
\arrow[d, "\cong"]
\\
(\Lambda_{X_2}^{\prime})^{*}
\arrow[r, "{(f_{\sharp})^{*}}"']
&
(\Lambda_{X_1}^{\prime})^{*}
\end{tikzcd}
\qquad
\begin{tikzcd}[column sep=large]
H^{1,0}(X_2)
\arrow[r, "f^{*}"]
\arrow[d, "\cong"']
&
H^{1,0}(X_1)
\arrow[d, "\cong"]
\\
\Lambda_{X_2}
\arrow[r, "f^{\sharp}"']
&
\Lambda_{X_1}.
\end{tikzcd}
\end{equation}
By Proposition~\ref{prop:GS19-toolkit}\textup{(i)},
\eqref{eq:cup-one-one-identifications}, and
\eqref{eq:pullback-one-zero-zero-one-lattice}, we obtain the following
commutative diagram:
\begin{equation}\label{eq:pullback-one-one-lattice}
\begin{tikzcd}[column sep=large]
H^{1,1}(X_2)
\arrow[r, "f^{*}"]
\arrow[d, "\cong"']
&
H^{1,1}(X_1)
\arrow[d, "\cong"]
\\
(\Lambda_{X_2}^{\prime})^{*}\otimes\Lambda_{X_2}
\arrow[r, "{(f_{\sharp})^{*}\otimes f^{\sharp}}"']
&
(\Lambda_{X_1}^{\prime})^{*}\otimes\Lambda_{X_1}.
\end{tikzcd}
\end{equation}
The homomorphism
\(
(f_{\sharp})^{*}\colon
(\Lambda_{X_2}^{\prime})^{*}
\longrightarrow
(\Lambda_{X_1}^{\prime})^{*}
\)
is induced by \(f_{\sharp}\). Together with \eqref{eq:chern-class-matrix} and
\eqref{eq:kappaD-matrix}, the commutative diagram
\eqref{eq:pullback-one-one-lattice} gives, for every tropical Cartier
divisor \(D\) on \(X_2\),
\begin{equation}\label{eq:kappa-pullback-divisor}
\kappa_{f^*D}
=
f^\sharp\circ\kappa_D\circ f_\sharp
\colon
\Lambda_{X_1}^{\prime}\to\Lambda_{X_1}.
\end{equation}
\end{remark}

{\subsection{Products and isomorphisms of principally polarized tropical abelian varieties}\noindent}Let \((X_i,\Theta_i)\) be principally polarized tropical abelian varieties for \(i=1,2\). {We define their product and the notion of an isomorphism between them.}

\begin{defi}\label{def:product-ppav}
The \emph{product} of \((X_1,\Theta_1)\) and \((X_2,\Theta_2)\) is
\[
(X_1,\Theta_1)\times (X_2,\Theta_2)\coloneqq
\bigl(X_1\times X_2,\pr_1^{*}\Theta_1+\pr_2^{*}\Theta_2\bigr).
\]

For \(\lambda_i^{\prime}\in \Lambda_{X_i}^{\prime}\) with \(i=1,2\), one has
\(
\kappa_{\pr_1^{*}\Theta_1+\pr_2^{*}\Theta_2}
(\lambda_1^{\prime},\lambda_2^{\prime})
=
\bigl(
\kappa_{\Theta_1}(\lambda_1^{\prime}),
\kappa_{\Theta_2}(\lambda_2^{\prime})
\bigr).
\)
Since \(\kappa_{\Theta_1}\) and \(\kappa_{\Theta_2}\) are lattice
isomorphisms, so is
\(\kappa_{\pr_1^{*}\Theta_1+\pr_2^{*}\Theta_2}\).
Following the notation in Subsection~\ref{subsec:integral-tori-polarizations}, for
\(u_i,v_i\in(\Lambda_{X_i}^{\prime})_{\R}\) with \(i=1,2\), one has
\begin{equation}\label{eq:Q-product-polarization}
Q_{\pr_1^{*}\Theta_1+\pr_2^{*}\Theta_2}
\bigl((u_1,u_2),(v_1,v_2)\bigr)
=
Q_{\Theta_1}(u_1,v_1)+Q_{\Theta_2}(u_2,v_2).
\end{equation}
Since \(Q_{\Theta_1}\) and \(Q_{\Theta_2}\) are symmetric positive
definite, so is
\(Q_{\pr_1^{*}\Theta_1+\pr_2^{*}\Theta_2}\).
Hence \(\pr_1^{*}\Theta_1+\pr_2^{*}\Theta_2\) is a principal polarization
on \(X_1\times X_2\).
\end{defi}
\begin{defi}\label{def:isomorphism-ppav}
An isomorphism \(f\colon X_1\to X_2\) is called an \emph{isomorphism of principally polarized tropical abelian varieties}
\[
f\colon (X_1,\Theta_1)\longrightarrow (X_2,\Theta_2)
\]
if
\[
[f^{*}\Theta_2]\equiv_{\hom}[\Theta_1].
\]
Here \([\,\cdot\,]\) denotes the associated tropical cycle of a tropical Cartier divisor; see Definition~\ref{def:intersection-number}.
\end{defi}

\subsection{Spanning curves and integral tori}\label{subsec:spanning-curves-integral-tori}
Let \(X=(\Lambda_X,\Lambda_X^{\prime},[\,\cdot\,,\,\cdot\,]_X)\) be an \(n\)-dimensional integral torus.

\begin{defi}\label{def:spanning-curve}
Let \(C\) be an effective tropical \(1\)-cycle on \(X\) such that \(|C|\) is connected. Following the convention in Definition~\ref{def:translation-chart}, we fix a suitable global face structure on \(|C|\). We call \(C\) a \emph{spanning curve} if
\begin{equation*}
\mathrm{span}_{\R}
\{\boldsymbol{u}_{\tau,v}\mid \tau \text{ is an edge of } |C|,\ v \text{ is an endpoint of } \tau\}
=
\R^n,
\end{equation*}
where \(\boldsymbol{u}_{\tau,v}\in\Z^n\) denotes the tangent direction of \(\tau\) out of \(v\). This definition is independent of the chosen global face structure with respect to the charts of \(X\).
\end{defi}

\begin{remark}\label{rem:spanning-no-proper-real-torus}
For such a spanning curve \(C\), there is no proper real subtorus \(Y\subsetneq X\) such that \(|C|\) is contained in a translate of \(Y\). Indeed, suppose that such a real subtorus
\(
Y=V/\Lambda_Y
\)
exists, where \(V\subseteq(\Lambda_X)_{\R}^{*}\) is a real vector subspace and
\(\Lambda_Y=V\cap\Lambda_X^{\prime}\). Let
\(\iota\colon(\Lambda_X)_{\R}^{*}\to\R^n\) be the isomorphism induced by the
\(\Z\)-basis fixed in Definition~\ref{def:translation-chart}. If \(|C|\) is contained in a translate of \(Y\), then the tangent direction
\(\boldsymbol{u}_{\tau,v}\) lies in \(\iota(V)\) for every edge \(\tau\) of \(|C|\) and every endpoint \(v\) of \(\tau\). Since \(C\) is a spanning curve, these tangent directions span \(\R^n\). Hence \(\iota(V)=\R^n\), so \(V=(\Lambda_X)_{\R}^{*}\). Therefore
\(\Lambda_Y=V\cap\Lambda_X^{\prime}=\Lambda_X^{\prime}\), and hence \(Y=X\) as a real torus.
\end{remark}
\begin{lemma}\label{lem:spanning-implies-surjective}
Let \(f\colon X_1\to X_2\) be a homomorphism of integral tori, and let \(C\) be a spanning curve on \(X_2\). If \(|C|\subseteq f(X_1)\), then \(f\) is surjective.
\end{lemma}

\begin{proof}
The set \(f(X_1)\) is a real subtorus of \(X_2\). If \(f\) is not surjective, then \(f(X_1)\ne X_2\), so \(f(X_1)\) is a proper real subtorus of \(X_2\). Since \(|C|\subseteq f(X_1)\), this contradicts Remark~\ref{rem:spanning-no-proper-real-torus}. Hence \(f(X_1)=X_2\), and therefore \(f\) is surjective.
\end{proof}

\subsection{Basic concepts and properties of tropical Jacobian varieties}\label{subsec:tropical-jacobians}
Fix a smooth tropical curve \(\Gamma\) of genus \(g\) and a basepoint \(q\in\Gamma\).
We use the notation of simplicial homology from Definition~\ref{def:tropical-curve}. Throughout this subsection, whenever a global face structure on \(\Gamma\) is chosen, {we} {use the following notation. For every edge} \(\tau\) of this global face structure {and every} endpoint \(v\) {of \(\tau\)}, let \(c_{\tau,v}\) denote the \(1\)-simplex supported on \(\tau\) and directed out of \(v\). Following Definition~\ref{def:edge-tangent-length}, let {\(\widetilde{\boldsymbol{u}}_{\tau,v}\in T_v^{\Z}\Gamma\)} be the tangent functional of \(\tau\) at \(v\). These notations are always understood with respect to the chosen global face structure.

Let \(H^0(\Gamma,\Omega_{\Gamma})\) be the lattice of global tropical \(1\)-forms on \(\Gamma\). By a nontrivial path \(\eta\colon[0,1]\to\Gamma\), we mean a continuous map for which there exist a global face structure on \(\Gamma\) and a subdivision
\(
0=t_0<t_1<\cdots<t_m=1
\)
such that, for every \(1\le \nu\le m\), the restriction \(\eta|_{[t_{\nu-1},t_\nu]}\) is injective and the image
\[
\tau_\nu\coloneqq\eta([t_{\nu-1},t_\nu])
\]
is an edge of this global face structure. Fix such a global face structure and subdivision. Set \(p_\nu\coloneqq\eta(t_\nu)\) for \(0\le \nu\le m\). We regard
\(\eta\) as the simplicial \(1\)-chain
\[
\sum_{\nu=1}^m c_{\tau_\nu,p_{\nu-1}}.
\]
For \(\omega\in H^0(\Gamma,\Omega_{\Gamma})\), define
\begin{equation}\label{eq:integral-along-edge}
\int_{\eta}\omega
\coloneqq
\sum_{\nu=1}^m
\ell(\tau_\nu)\,
\widetilde{\boldsymbol{u}}_{\tau_\nu,p_{\nu-1}}(\omega_{p_{\nu-1}}),
\end{equation}
where \(\omega_{p_{\nu-1}}\in\Omega_{\Gamma,p_{\nu-1}}\) denotes the germ of
\(\omega\) at \(p_{\nu-1}\), and \(\ell(\tau_\nu)\) denotes the length of
\(\tau_\nu\). The value \(\int_{\eta}\omega\) is independent of the chosen
global face structure and subdivision. If \(\eta\) is a trivial path, we set
\(
\int_{\eta}\omega\coloneqq 0.
\)
Regard \(H_1(\Gamma,\Z)\) as the sublattice of \(C_1(\Gamma,\Z)\). Then one obtains a well-defined pairing
\begin{equation*}
[\,\cdot\,,\,\cdot\,]_{\Gamma}\colon H^0(\Gamma,\Omega_{\Gamma})\times H_1(\Gamma,\Z)\to\R,
\qquad
[\omega,\eta]_{\Gamma}\coloneqq \int_{\eta}\omega.
\end{equation*}
By \(\R\)-linearity, this pairing extends to a bilinear pairing
\(
[\,\cdot\,,\,\cdot\,]_{\Gamma}\colon (H^0(\Gamma,\Omega_{\Gamma}))_{\R}\times H_1(\Gamma,\R)\to\R.
\)
\begin{defi}[Tropical Jacobian varieties]\label{def:tropical-jacobian-pair}
The \emph{tropical Jacobian variety} of \(\Gamma\) is the integral torus
\begin{equation*}
\Jac(\Gamma)\coloneqq \bigl(H^0(\Gamma,\Omega_{\Gamma}),\,H_1(\Gamma,\Z),\,[\,\cdot\,,\,\cdot\,]_{\Gamma}\bigr),
\end{equation*}
which carries the canonical principal polarization \(\Theta_{\Gamma}\). The associated map
\(\kappa_{\Theta_{\Gamma}}\colon H_1(\Gamma,\Z)\to H^0(\Gamma,\Omega_{\Gamma})\)
is the canonical isomorphism; see \cite[Sec.~4]{GS23}. Choose a global
face structure on \(\Gamma\), and choose one endpoint \(v(\tau)\) of each edge
\(\tau\in E(\Gamma)\). Regard \(H_1(\Gamma,\R)\) as a subspace of
\(C_1(\Gamma,\R)\). For \(\gamma\in H_1(\Gamma,\R)\) written as
\[
\gamma=\sum_{\tau\in E(\Gamma)}a_{\tau}c_{\tau,v(\tau)},
\]
with \(a_{\tau}\in\R\) for all \(\tau\in E(\Gamma)\), the element
\(\kappa_{\Theta_{\Gamma}}(\gamma)\in (H^0(\Gamma,\Omega_{\Gamma}))_{\R}\) is characterized by
\begin{equation}\label{eq:kappa-theta-gamma-characterization}
\widetilde{\boldsymbol{u}}_{\tau,v(\tau)}
\bigl((\kappa_{\Theta_{\Gamma}}(\gamma))_{v(\tau)}\bigr)
=a_{\tau},
\qquad
\tau\in E(\Gamma),
\end{equation}
where \((\kappa_{\Theta_{\Gamma}}(\gamma))_{v(\tau)}\) denotes the germ of
\(\kappa_{\Theta_{\Gamma}}(\gamma)\) at \(v(\tau)\).
The canonical symmetric positive definite bilinear form
\(Q_{\Theta_{\Gamma}}\colon H_1(\Gamma,\R)\times H_1(\Gamma,\R)\to\R\)
is defined by
\begin{equation*}
Q_{\Theta_{\Gamma}}(\gamma,\gamma')\coloneqq
[\kappa_{\Theta_{\Gamma}}(\gamma),\gamma']_{\Gamma}.
\end{equation*}
Write \(\gamma'\in H_1(\Gamma,\R)\) as
\[
\gamma'=\sum_{\tau\in E(\Gamma)}b_{\tau}c_{\tau,v(\tau)},
\]
with \(b_{\tau}\in\R\) for all \(\tau\in E(\Gamma)\). Then
\begin{equation}\label{eq:Q-Theta-Gamma-edge-expansion}
\begin{aligned}
Q_{\Theta_{\Gamma}}(\gamma,\gamma')
&=[\kappa_{\Theta_{\Gamma}}(\gamma),\gamma']_{\Gamma}
=\sum_{\tau\in E(\Gamma)}b_{\tau}
\int_{c_{\tau,v(\tau)}}\kappa_{\Theta_{\Gamma}}(\gamma) \\
&=\sum_{\tau\in E(\Gamma)}\ell(\tau)b_{\tau}\,
\widetilde{\boldsymbol{u}}_{\tau,v(\tau)}
\bigl((\kappa_{\Theta_{\Gamma}}(\gamma))_{v(\tau)}\bigr)
=\sum_{\tau\in E(\Gamma)}\ell(\tau)a_{\tau}b_{\tau}.
\end{aligned}
\end{equation}
The \emph{tropical Abel--Jacobi map} based at \(q\) is defined by
\begin{equation*}
\phi_q\colon \Gamma\to \Jac(\Gamma),\qquad
p\mapsto \bigl(\omega\mapsto \int_{\eta_p}\omega\bigr),
\end{equation*}
where \(\eta_p\) is any path in \(\Gamma\) from \(q\) to \(p\). The expression on the right is understood as the class in
\((H^0(\Gamma,\Omega_{\Gamma}))_{\R}^{*}/H_1(\Gamma,\Z)\)
represented by the \(\R\)-linear extension of the map
\[
\lambda_{\eta_p}\colon H^0(\Gamma,\Omega_{\Gamma})\longrightarrow\R,
\qquad
\omega\longmapsto \int_{\eta_p}\omega.
\]
\end{defi}

\begin{prop}[{\cite[Prop.~4.14]{RZ25}}]\label{prop:universal-property}
Let \(X\) be an integral torus, and let \(\chi\colon\Gamma\to X\) be a morphism of rational polyhedral spaces. Then, for the basepoint \(q\in\Gamma\), there exists a unique homomorphism of integral tori \(f\colon \Jac(\Gamma)\to X\) making the following diagram commute:
\begin{equation}\label{eq:AbelJacobiCommute}
\begin{tikzcd}[row sep=huge, column sep=huge]
\Gamma \arrow[r, "\chi"] \arrow[d, "\phi_q"'] &
X \arrow[d, "t_{-\chi(q)}"'] \\
\Jac(\Gamma) \arrow[r, "f"'] &
X
\end{tikzcd}
\end{equation}
Here \(t_{-\chi(q)}\colon X\to X\) is the translation defined by \(-\chi(q)\in X\).
\end{prop}

\begin{remark}\label{rem:universal-property-sharp}
Following the notation in Proposition~\ref{prop:universal-property}, the homomorphism \(f=(f^{\sharp},f_{\sharp})\) satisfies
\begin{equation*}
f^{\sharp}=\chi^{*}\colon H^0(X,\Omega_X)\to H^0(\Gamma,\Omega_{\Gamma}),
\qquad
f_{\sharp}=\chi_{*}\colon H_1(\Gamma,\Z)\to H_1(X,\Z).
\end{equation*}
This follows from the construction in the proof of \cite[Prop.~4.14]{RZ25}.
\end{remark}

\section{Hodge-type inequalities on tropical abelian varieties}\label{sec:hodge-type-inequalities}
In analogy with the classical theory of abelian varieties, we establish Hodge-type inequalities for ample tropical Cartier divisors on tropical abelian varieties. These inequalities, in particular their equality conditions, will be used to prove Proposition~\ref{prop:equivalence-conditions-two-criteria}. Let \(g\in\Z_{>1}\).
\subsection{Mixed discriminants}\label{subsec:mixed-discriminants}
We begin by recalling mixed discriminants following \cite{Li23}, together with their basic properties, and then record the Alexandrov--Fenchel inequalities for mixed discriminants.

\begin{defi}[Mixed discriminants]\label{def:mixed-discriminant}
Let \(A_1,\dots,A_g\in M_g(\R)\). Define
\[
\Delta(A_1,\dots,A_g)\coloneqq \frac{1}{g!}\frac{\partial^g}{\partial t_1\cdots\partial t_g}\det \bigl(t_1A_1+\cdots+t_gA_g\bigr)\Big|_{t_1=\cdots=t_g=0}.
\]
Equivalently, \(g!\,\Delta(A_1,\dots,A_g)\) is the coefficient of \(t_1\cdots t_g\) in \(\det(t_1A_1+\cdots+t_gA_g)\).
\end{defi}

\begin{prop}\label{prop:mixed-discriminant-properties}
Fix real matrices \(A_1,\dots,A_g\in M_g(\R)\).
\begin{enumerate}[label=\textup{(\roman*)}, itemsep=2pt]
\item \(\Delta(A_1,\dots,A_g)\) is symmetric and multilinear in \(A_1,\dots,A_g\).
\item For any \(A\in M_g(\R)\), one has \(\Delta(A,\dots,A)=\det A\).
\item If \(A_1,\dots,A_g\) are symmetric positive semidefinite matrices, then \[\Delta(A_1,\dots,A_g) \ge \prod_{k=1}^g (\det A_k)^{1/g}.\]
\end{enumerate}
\end{prop}

\begin{proof}
Properties \textup{(i)} and \textup{(ii)} follow directly from Definition~\ref{def:mixed-discriminant}. The inequality in \textup{(iii)} is \cite[Thm.~8]{Bap89}.
\end{proof}

\begin{thm}[Alexandrov--Fenchel inequalities for mixed discriminants; see {\cite[Thm.~2.4]{Li23}}]\label{thm:AF-mixed-disc}
Let \(E^{(1)},\dots,E^{(g)}\) be symmetric positive definite \(g\times g\) matrices. For any distinct integers \(i,j\in\{1,\ldots,g\}\),
\[
\begin{aligned}
\Delta\bigl(E^{(1)},\dots,E^{(g)}\bigr)\ge{}& \Bigl[\Delta\bigl(E^{(1)},\dots,E^{(j-1)},E^{(i)},E^{(j+1)},\dots,E^{(g)}\bigr) \\
& \cdot \Delta\bigl(E^{(1)},\dots,E^{(i-1)},E^{(j)},E^{(i+1)},\dots,E^{(g)}\bigr)\Bigr]^{1/2}.
\end{aligned}
\]
Moreover, for any such pair \(i\ne j\), equality holds if and only if \(E^{(j)}=\alpha\,E^{(i)}\) for some \(\alpha\in \R_{>0}\).
\end{thm}

\begin{lemma}\label{lem:bapat-equality-positive-definite}
In Proposition~\ref{prop:mixed-discriminant-properties}\textup{(iii)}, assume moreover that
\(A_1,\dots,A_g\) are symmetric positive definite matrices. Then equality holds if and only if
there exist \(\lambda_2,\dots,\lambda_g\in\R_{>0}\) such that
\(A_i=\lambda_iA_1\) for all \(2\le i\le g\).
\end{lemma}

\begin{proof}
This follows from \cite[Cor.~3.3]{Li23}.
\end{proof}

\subsection{Mixed discriminants for intersection numbers on integral tori}\label{subsec:mixed-disc-intersection}
For this subsection and the next one, let \(X=(\Lambda_X,\Lambda_X^{\prime},[\,\cdot\,,\,\cdot\,]_X)\) be a \(g\)-dimensional integral torus, and retain the \(\Z\)-bases \(\{\lambda_i^{\prime}\}_{i=1}^g\) of \(\Lambda_X^{\prime}\) and \(\{\eta_i\}_{i=1}^g\) of \(\Lambda_X\), together with their dual bases, fixed in Subsection~\ref{subsec:tropical-cohomology}. We also retain the sign \(\varepsilon\in\{\pm1\}\) fixed in \eqref{eq:fundamental-cycle-class}. Recall from Subsection~\ref{subsec:integral-tori-polarizations} that for any \(D\in \Div(X)\), \(Q_D\) denotes the symmetric bilinear form associated to \(D\).
In this subsection, we relate intersection numbers of tropical Cartier divisors on \(X\) to mixed discriminants of their associated matrices.

\begin{prop}\label{prop:matrix-relations-SGD}
For any \(D\in \Div(X)\), let
\[
E^{D}=\bigl(Q_D(\lambda_i^{\prime},\eta_j^{*})\bigr)_{1\le i,j\le g},
\qquad
G^{D}=\bigl(Q_D(\eta_i^{*},\eta_j^{*})\bigr)_{1\le i,j\le g}.
\]
Then there exists \(S\in \mathrm{GL}_g(\R)\), independent of \(D\), such that the following hold.
\begin{enumerate}[label=\textup{(\roman*)}, itemsep=2pt]
\item One has
\[
E^{D}=S^{\top}G^{D}.
\]
\item Let \(D_1,\dots,D_g\in \Div(X)\). Then
\[
\Delta\bigl(E^{D_1},\dots,E^{D_g}\bigr)
=
\det(S)\,\Delta\bigl(G^{D_1},\dots,G^{D_g}\bigr).
\]
\item With \(D_1,\dots,D_g\) as in \textup{(ii)}, one has
\[
(D_1\cdots D_g)_X
=
\varepsilon\, g!\,\det(S)\,
\Delta\bigl(G^{D_1},\dots,G^{D_g}\bigr),
\]
where \(\varepsilon\in\{\pm1\}\) is the sign fixed in
\eqref{eq:fundamental-cycle-class}.
\end{enumerate}
\end{prop}

\begin{proof}
Let \(S=\bigl(S_{i,j}\bigr)_{1\le i,j\le g}\in M_g(\R)\) be defined by
\(\lambda_i^{\prime}=\sum_{k=1}^{g} S_{ki}\eta_k^{*}\quad\text{for all }1\le i\le g\).

\begin{enumerate}[label=\textup{(\roman*)}, itemsep=2pt]
\item By the definition of \(E^{D}\), one has
\[
E^{D}_{ij}=Q_D(\lambda_i^{\prime},\eta_j^{*})=\sum_{k=1}^g S_{ki}\,Q_D(\eta_k^{*},\eta_j^{*})=\sum_{k=1}^g (S^{\top})_{ik}(G^{D})_{kj},
\]
hence \(E^{D}=S^{\top}G^{D}\).

\item By \textup{(i)}, for all \(t_1,\dots,t_g\in\Z_{>0}\), one has
\[
\sum_{i=1}^g t_iE^{D_i}=S^{\top}\Bigl(\sum_{i=1}^g t_iG^{D_i}\Bigr).
\]
Taking determinants on both sides, one has
\begin{equation}\label{eq:mixed-discriminant-EG-polynomial}
\det\Bigl(\sum_{i=1}^g t_iE^{D_i}\Bigr)=\det(S)\,\det\Bigl(\sum_{i=1}^g t_iG^{D_i}\Bigr),
\end{equation}
for all \(t_1,\dots,t_g\in\Z_{>0}\). Since both sides are polynomials in \(t_1,\dots,t_g\), they are equal as polynomials. Comparing the coefficient of \(t_1\cdots t_g\) on both sides of \eqref{eq:mixed-discriminant-EG-polynomial}, we obtain
\[
\Delta\bigl(E^{D_1},\dots,E^{D_g}\bigr)=\det(S)\,\Delta\bigl(G^{D_1},\dots,G^{D_g}\bigr).
\]

\item For \(t=(t_1,\dots,t_g)\in (\Z_{>0})^g\), set \(D(t)\coloneqq \sum_{i=1}^g t_iD_i\). Then \(D(t)\in \Div(X)\), and one has
\[
Q_{D(t)}=\sum_{i=1}^g t_iQ_{D_i}
\qquad\text{and}\qquad
E^{D(t)}=\sum_{i=1}^g t_iE^{D_i}.
\]
By Lemma~\ref{lem:top-self-intersection-det}\textup{(i)}, one has
\(
\bigl(D(t)^g\bigr)_X=\varepsilon\, g!\,\det\bigl(E^{D(t)}\bigr)=\varepsilon\, g!\,\det\Bigl(\sum_{i=1}^g t_iE^{D_i}\Bigr).
\)
Comparing the coefficient of \(t_1\cdots t_g\) on both sides of the equality above, we get
\[
(D_1\cdots D_g)_X=\varepsilon\, g!\,\Delta\bigl(E^{D_1},\dots,E^{D_g}\bigr).
\]
Combining this with \textup{(ii)}, we obtain
\(
(D_1\cdots D_g)_X=\varepsilon\, g!\,\det(S)\,\Delta\bigl(G^{D_1},\dots,G^{D_g}\bigr).
\)\qedhere \end{enumerate}
\end{proof}

\subsection{Hodge-type inequalities on tropical abelian varieties}
In this subsection, we establish Hodge-type inequalities for ample tropical Cartier divisors on tropical abelian varieties.

\begin{prop}\label{prop:tropical-KT}
Let \(X\) be a \(g\)-dimensional tropical abelian variety, and let \(D_1,\dots,D_g\) be ample tropical Cartier divisors on \(X\). Then
\begin{equation}\label{eq:hodge-type-product-inequality}
(D_1\cdots D_g)_X\ge \prod_{i=1}^g (D_i^g)_X^{1/g}.
\end{equation}
Equality holds if and only if there exist \(m_1,\dots,m_g\in \Z_{>0}\) such that
\(
m_1[D_1]\equiv_{\hom}m_2[D_2]\equiv_{\hom}\cdots\equiv_{\hom}m_g[D_g].
\)
\end{prop}

\begin{proof}
\smallskip
\noindent\textbf{Step \textup{(i)}.}
Follow the notation in Proposition~\ref{prop:matrix-relations-SGD} for \(D_1,\dots,D_g\).
By Proposition~\ref{prop:matrix-relations-SGD}\textup{(iii)},
\begin{equation}\label{eq:intersection-via-G}
(D_1\cdots D_g)_X=\varepsilon\, g!\,\det(S)\,\Delta\bigl(G^{D_1},\dots,G^{D_g}\bigr).
\end{equation}
For each \(1\le i\le g\), Proposition~\ref{prop:matrix-relations-SGD}\textup{(i)} gives
\[
E^{D_i}=S^{\top}G^{D_i},
\qquad
\det(E^{D_i})=\det(S)\,\det(G^{D_i}).
\]
Therefore, by Proposition~\ref{prop:mixed-discriminant-properties}\textup{(ii)},
\begin{equation}\label{eq:self-intersection-via-G}
(D_i^g)_X=\varepsilon g!\Delta(E^{D_i},\dots,E^{D_i})=\varepsilon g!\det(E^{D_i})=\varepsilon g!\det(S)\det(G^{D_i}).
\end{equation}
Since each \(D_i\) is ample, each \(G^{D_i}\) is symmetric positive definite. Hence Proposition~\ref{prop:mixed-discriminant-properties}\textup{(iii)} gives
\(
\Delta\bigl(G^{D_1},\dots,G^{D_g}\bigr)>0.
\)
Since \(D_1\) is ample,
Lemma~\ref{lem:top-self-intersection-det}\textup{(ii)} gives
\(
(D_1^g)_X>0.
\)
Moreover, \(\det(G^{D_1})>0\) because \(G^{D_1}\) is symmetric positive
definite. Hence \eqref{eq:self-intersection-via-G} implies
\(
\varepsilon\,\det(S)>0.
\)
Together with \eqref{eq:intersection-via-G} and
\(
\Delta\bigl(G^{D_1},\dots,G^{D_g}\bigr)>0,
\)
this also yields
\(
(D_1\cdots D_g)_X>0.
\)

\smallskip
\noindent\textbf{Step \textup{(ii)}.}
Applying Proposition~\ref{prop:mixed-discriminant-properties}\textup{(iii)} to the symmetric positive definite matrices \(G^{D_1},\dots,G^{D_g}\), we obtain
\begin{equation}\label{eq:mixed-discriminant-AMGM-GD}
\Delta\bigl(G^{D_1},\dots,G^{D_g}\bigr)
\ge
\prod_{i=1}^g \det(G^{D_i})^{1/g}.
\end{equation}
Multiplying by \(\varepsilon\,g! \,\det(S)\) on both sides of \eqref{eq:mixed-discriminant-AMGM-GD}, and using \eqref{eq:intersection-via-G} and \eqref{eq:self-intersection-via-G}, we get
\[
\begin{aligned}
(D_1\cdots D_g)_X
&=\varepsilon\, g!\,\det(S)\,\Delta\bigl(G^{D_1},\dots,G^{D_g}\bigr)\\
&\ge \prod_{i=1}^g \bigl(\varepsilon\, g!\,\det(S)\,\det(G^{D_i})\bigr)^{1/g}
=\prod_{i=1}^g (D_i^g)_X^{1/g}.
\end{aligned}
\]

\smallskip
\noindent\textbf{Step \textup{(iii)}.}
We analyze the equality case. Assume equality holds in \eqref{eq:hodge-type-product-inequality}. By \eqref{eq:intersection-via-G}
and \eqref{eq:self-intersection-via-G}, after canceling the common positive factor
\(\varepsilon\, g!\,\det(S)\), we obtain
\begin{equation}\label{eq:mixed-AMGM}
\Delta\bigl(G^{D_1},\dots,G^{D_g}\bigr)
=
\prod_{k=1}^g \det(G^{D_k})^{1/g}.
\end{equation}
Since \(G^{D_1},\dots,G^{D_g}\) are symmetric positive definite, by Lemma~\ref{lem:bapat-equality-positive-definite}, there exist \(\lambda_2,\dots,\lambda_g\in \R_{>0}\) such that
\(G^{D_i}=\lambda_iG^{D_1}\) for all \(2\le i\le g\).
Since \(E^{D_i}=S^{\top}G^{D_i}\), it follows that
\(
E^{D_i}=\lambda_iE^{D_1}
\) for all \(2\le i\le g\).
Since \(E^{D_1},\dots,E^{D_g}\in M_g(\Z)\) and \(E^{D_1}\neq 0\), the equality
\(
E^{D_i}=\lambda_iE^{D_1}
\)
implies that \(\lambda_i\in\Q_{>0}\) for all \(2\le i\le g\). Clearing denominators, there exist \(m_1,\dots,m_g\in\Z_{>0}\) such that
\[
m_1E^{D_1}=m_2E^{D_2}=\cdots=m_gE^{D_g}.
\]
By \eqref{eq:chern-class-matrix} and \eqref{eq:chern-class-homological-equivalence}, it follows that
\begin{equation}\label{eq:hom-proportional-divisor-classes}
m_1[D_1]\equiv_{\hom}m_2[D_2]\equiv_{\hom}\cdots\equiv_{\hom}m_g[D_g].
\end{equation}
Conversely, assume that \eqref{eq:hom-proportional-divisor-classes} holds for some \(m_1,\dots,m_g\in \Z_{>0}\). Then
\[
E^{D_i}=\frac{m_1}{m_i}E^{D_1},
\qquad 2\le i\le g.
\]
Since \(E^{D_i}=S^{\top}G^{D_i}\), we have
\[
G^{D_i}=\frac{m_1}{m_i}G^{D_1},
\qquad 2\le i\le g.
\]
Hence equality holds in \eqref{eq:mixed-AMGM} by Lemma~\ref{lem:bapat-equality-positive-definite}. By \eqref{eq:intersection-via-G}
and \eqref{eq:self-intersection-via-G}, one obtains equality in \eqref{eq:hodge-type-product-inequality}.
\end{proof}

\begin{corollary}[Teissier--Khovanskii inequalities on tropical abelian varieties]\label{cor:TK-inequality}
Let \(X\) be a \(g\)-dimensional tropical abelian variety, and let \(D_1,\dots,D_g\) be ample tropical Cartier divisors on \(X\). Then for any two distinct integers \(i,j\in\{1,\ldots,g\}\),
\[
(D_1\cdots D_g)_X^2
\ge
(D_1\cdots D_{j-1}\cdot D_i\cdot D_{j+1}\cdots D_g)_X\cdot
\,
(D_1\cdots D_{i-1}\cdot D_j\cdot D_{i+1}\cdots D_g)_X.
\]
Equality holds if and only if there exist \(m_i,m_j\in \Z_{>0}\) such that
\(
m_i[D_i]\equiv_{\hom}m_j[D_j].
\)
\end{corollary}

\begin{proof}
By \eqref{eq:intersection-via-G}, the asserted inequality is exactly the
Alexandrov--Fenchel inequality in Theorem~\ref{thm:AF-mixed-disc} applied to
\(G^{D_1},\dots,G^{D_g}\).

Assume equality holds. Since \(G^{D_i}\) and \(G^{D_j}\) are symmetric positive definite, by the equality condition in Theorem~\ref{thm:AF-mixed-disc}, there exists \(\alpha\in\R_{>0}\) such that
\(
G^{D_j}=\alpha G^{D_i}.
\)
Since \(E^{D_k}=S^{\top}G^{D_k}\) for \(k=i,j\), it follows that
\(
E^{D_j}=\alpha E^{D_i}.
\)
Since \(E^{D_i},E^{D_j}\in M_g(\Z)\) and \(E^{D_i}\neq 0\), one has \(\alpha\in\Q_{>0}\). Hence there exist \(m_i,m_j\in\Z_{>0}\) such that
\[
m_iE^{D_i}=m_jE^{D_j}.
\]
By \eqref{eq:chern-class-matrix} and \eqref{eq:chern-class-homological-equivalence}, it follows that
\[
m_i[D_i]\equiv_{\hom}m_j[D_j].
\]

Conversely, assume that
\(
m_i[D_i]\equiv_{\hom}m_j[D_j]
\)
for some \(m_i,m_j\in\Z_{>0}\). By \eqref{eq:chern-class-homological-equivalence} and \eqref{eq:chern-class-matrix}, it follows that
\[
m_iE^{D_i}=m_jE^{D_j}.
\]
Since \(E^{D_k}=S^{\top}G^{D_k}\) for \(k=i,j\), it follows that
\[
G^{D_j}=\frac{m_i}{m_j}G^{D_i}.
\]
Since \(G^{D_i}\) and \(G^{D_j}\) are symmetric positive definite, by the equality condition in Theorem~\ref{thm:AF-mixed-disc}, equality holds in the asserted inequality.
\end{proof}

\begin{corollary}\label{cor:principal-intersection-equality}
Let \((X,\Theta)\) be a \(g\)-dimensional polarized tropical abelian variety, let \(D\) be an ample tropical Cartier divisor, and let \(n\) be an integer with \(1\le n\le g-1\). Then
\[
(D^n\cdot \Theta^{g-n})_X\ge g!,
\]
with equality if and only if \(\Theta\) is a principal polarization and
\([D]\equiv_{\hom}[\Theta]\).
\end{corollary}

\begin{proof}
Proposition~\ref{prop:tropical-KT} gives
\begin{equation}\label{eq:khovanskii-teissier-D-Theta}
(D^n\cdot \Theta^{g-n})_X
\ge
(D^g)_X^{n/g}(\Theta^g)_X^{(g-n)/g}.
\end{equation}
Since \(D\) and \(\Theta\) are ample,
Lemma~\ref{lem:top-self-intersection-det}\textup{(ii)} gives
\[
(D^g)_X\ge g!,
\qquad
(\Theta^g)_X\ge g!.
\]
Together with \eqref{eq:khovanskii-teissier-D-Theta}, this yields
\[
(D^n\cdot \Theta^{g-n})_X\ge g!.
\]

Assume that
\(
(D^n\cdot \Theta^{g-n})_X=g!.
\)
Then \eqref{eq:khovanskii-teissier-D-Theta} forces
\begin{equation}\label{eq:self-intersection-D-Theta}
(D^g)_X=(\Theta^g)_X=g!.
\end{equation}
Hence Lemma~\ref{lem:top-self-intersection-det}\textup{(iii)} shows that
\(\Theta\) is a principal polarization. Moreover, equality holds in
Proposition~\ref{prop:tropical-KT}, so there exist \(a,b\in\Z_{>0}\) such that
\[
a[D]\equiv_{\hom}b[\Theta].
\]
By \eqref{eq:self-intersection-D-Theta}, one has \(a=b\). Since
\(H_{g-1,g-1}(X)\) is a free \(\Z\)-module, it follows that
\[
[D]\equiv_{\hom}[\Theta].
\]

Conversely, assume that \(\Theta\) is a principal polarization and
\([D]\equiv_{\hom}[\Theta]\). Then
\[
(D^n\cdot \Theta^{g-n})_X
=
(\Theta^g)_X
=
g!. \qedhere
\]
\end{proof}

\section{Polarization homomorphisms and pushforwards on tropical homology}\label{sec:fin-dec}

Following the classical theory for abelian varieties as in \cite{BL04}, we collect auxiliary results on homomorphisms of integral tori that will be used later. These results are tropical analogues of standard facts concerning homomorphisms induced by polarizations and induced maps on tropical (co)homology.

\subsection{Homomorphisms of integral tori induced by polarizations}\label{subsec:polarization-induced-homomorphisms}
Let \(X=(\Lambda_X,\Lambda_X^{\prime},[\,\cdot\,,\,\cdot\,]_X)\) be an integral torus. In this subsection, we record several properties of the homomorphisms induced by polarizations on integral tori. 

\begin{defi}\label{def:phi-Theta}
Let \(D\in\Div(X)\). For \(x\in X\), let \(t_x\colon X\to X\) denote translation by \(x\). By \cite[Prop.~7.5]{GS23}, one has \(c_1(t_x^*\MO_X(D))=c_1(\MO_X(D))\), and hence \(t_x^*\MO_X(D)\otimes\MO_X(D)^{-1}\in\Pic^0(X)\). Define
\[
\varphi_D\colon X\to \Pic^{0}(X),
\qquad
x\mapsto t_x^*\MO_X(D)\otimes \MO_X(D)^{-1}.
\]
\end{defi}

\begin{prop}\label{prop:phi-Theta-dual-kappa}
The group \(\Pic^{0}(X)\) has a natural integral torus structure, canonically identified with \(X^\vee\). Let $\Theta$ be a polarization on $X$. The {map} \(\varphi_{\Theta}\colon X\to \Pic^{0}(X)\) is a homomorphism of integral tori such that
\begin{equation}\label{eq:phi-Theta-sharp-kappa}
\varphi_{\Theta}^{\sharp}=(\varphi_{\Theta})_{\sharp}=-\kappa_{\Theta}\colon \Lambda_X^{\prime}\to \Lambda_X.
\end{equation}
Moreover, one has
\(
\varphi_{\Theta}^{\vee}=\varphi_{\Theta}
\)
and \(\varphi_{\Theta}\) is finite surjective.
\end{prop}

\begin{proof}
By the tropical Appell--Humbert theorem \cite[Prop.~7.1, Thm.~7.2]{GS23}, we identify \(\Pic^{0}(X)\) with \(X^\vee\), and regard \(\varphi_{\Theta}\) as a homomorphism from \(X\) to \(X^\vee\). By the definition of the dual integral torus, one has \((X^\vee)^\vee=X\). Under these identifications, \eqref{eq:phi-Theta-sharp-kappa}, the equality \(\varphi_{\Theta}^{\vee}=\varphi_{\Theta}\), and the finiteness of \(\varphi_{\Theta}\) follow from the proof of \cite[Lem.~3.5]{GS25} and \cite[Rem.~3.6(ii)]{GS25}. Since \(\kappa_{\Theta}\colon(\Lambda_X^{\prime})_{\R}\to(\Lambda_X)_{\R}\) is surjective, \(\varphi_{\Theta}\) is surjective.
\end{proof}

\begin{lemma}\label{lem:phi-pullback-dual}
Let \(X_1\) and \(X_2\) be integral tori, let \(f\colon X_1\to X_2\) be a homomorphism of integral tori, and let \(\Theta\) be a polarization on \(X_2\).
Let \(f^{*}\big|_{\Pic^{0}(X_2)}\colon\Pic^{0}(X_2)\to\Pic^{0}(X_1)\) be the restriction of the pullback. Then the following hold.
\begin{enumerate}[label=\textup{(\roman*)}]
\item One has
\begin{equation}\label{eq:phi-pullback-pic}
\varphi_{f^*\Theta}
=
\bigl(f^{*}\big|_{\Pic^{0}(X_2)}\bigr)\circ\varphi_{\Theta}\circ f.
\end{equation}
\item Under the canonical identifications \(\Pic^{0}(X_i)\cong X_i^\vee\) for \(i=1,2\), one has
\begin{equation}\label{eq:pullback-pic-zero-dual}
f^{*}\big|_{\Pic^{0}(X_2)}=f^\vee.
\end{equation}
\end{enumerate}
\end{lemma}
\begin{proof}
\begin{enumerate}[label=\textup{(\roman*)}]
\item For \(i=1,2\) and \(x_i\in X_i\), let \(t_{x_i}\colon X_i\to X_i\) denote translation by \(x_i\). By definition,
\[
\varphi_{\Theta}(x_2)
=
\MO_{X_2}(t_{x_2}^*\Theta-\Theta),
\qquad
\varphi_{f^*\Theta}(x_1)
=
\MO_{X_1}(t_{x_1}^*(f^*\Theta)-f^*\Theta).
\]
Since \(f\circ t_{x_1}=t_{f(x_1)}\circ f\), one has
\[
t_{x_1}^*(f^*\Theta)-f^*\Theta
=
f^*(t_{f(x_1)}^*\Theta-\Theta).
\]
Hence \(\varphi_{f^*\Theta}(x_1)=f^*(\varphi_{\Theta}(f(x_1)))\), which proves \eqref{eq:phi-pullback-pic}.
\item The argument is analogous to the proof of \cite[Prop.~2.4.2]{BL04}. Since \(f^\vee=(f_{\sharp},f^{\sharp})\), the homomorphism \(f^\vee\colon X_2^\vee\to X_1^\vee\) is induced by the \(\R\)-linear map
\[
(f_{\sharp})^{*}\colon(\Lambda_{X_2}^{\prime})_{\R}^{*}\to(\Lambda_{X_1}^{\prime})_{\R}^{*},
\qquad
u\mapsto u\circ f_{\sharp}.
\]
By \cite[Prop.~7.1, Thm.~7.2]{GS23}, every element of
\(\Pic^{0}(X_2)\) is represented by \(L(0,l)\) for some
\(l\in(\Lambda_{X_2}^{\prime})_{\R}^{*}\). Since \(l\circ f_{\sharp}=(f_{\sharp})^{*}(l)\), one obtains \eqref{eq:pullback-pic-zero-dual}. \qedhere
\end{enumerate}
\end{proof}
\subsection{Pushforwards on tropical homology}
In this subsection, let \(X_i=(\Lambda_{X_i},\Lambda_{X_i}^{\prime},[\,\cdot\,,\,\cdot\,]_{X_i})\) be integral tori for \(i=1,2\), and let \(f=(f^{\sharp},f_\sharp)\colon X_1\to X_2\) be a homomorphism of integral tori. We write \((f^{\sharp})^{*}\colon (\Lambda_{X_1})^{*}\to(\Lambda_{X_2})^{*}\) and \((f_{\sharp})^{*}\colon (\Lambda_{X_2}^{\prime})^{*}\to(\Lambda_{X_1}^{\prime})^{*}\) for the dual maps of \(f^{\sharp}\) and \(f_{\sharp}\), respectively.
The symbols \(f^*\) and \(f_*\) below denote the pullback on tropical cohomology and the pushforward on tropical homology recalled in Subsection~\ref{subsec:tropical-cohomology}, respectively. We also write \(f_*\) for the pushforward of tropical cycles.

\begin{lemma}\label{lem:push-PPTAV-en}
Under the canonical isomorphisms
\(
H_{1,1}(X_i)\cong \Lambda_{X_i}^{\prime}\otimes(\Lambda_{X_i})^{*}
\)
for \(i=1,2\), the pushforward on tropical homology is given by
\[
f_{*}\colon
\Lambda_{X_1}^{\prime}\otimes(\Lambda_{X_1})^{*}
\longrightarrow
\Lambda_{X_2}^{\prime}\otimes(\Lambda_{X_2})^{*},
\qquad
v\otimes\phi
\longmapsto
f_{\sharp}(v)\otimes(f^{\sharp})^{*}\phi.
\]
\end{lemma}

\begin{proof}
Under the identifications in the statement and
\(
H^{1,1}(X_i)\cong(\Lambda_{X_i}^{\prime})^{*}\otimes\Lambda_{X_i}
\), write
\(\langle\cdot,\cdot\rangle_{X_i}\) for the pairing
\(
H^{1,1}(X_i)\times H_{1,1}(X_i)\to\Z.
\)
One has
\begin{equation}\label{eq:pairing-one-one}
\langle u\otimes\eta,\,v\otimes\phi\rangle_{X_i}
=
u(v)\,\phi(\eta),
\end{equation}
for \(i=1,2\), where \(u\in(\Lambda_{X_i}^{\prime})^{*}\), \(\eta\in\Lambda_{X_i}\),
\(v\in\Lambda_{X_i}^{\prime}\), and \(\phi\in(\Lambda_{X_i})^{*}\).
Define \(\widetilde f_{*}\) by
\[
\widetilde f_{*}\colon
\Lambda_{X_1}^{\prime}\otimes(\Lambda_{X_1})^{*}
\longrightarrow
\Lambda_{X_2}^{\prime}\otimes(\Lambda_{X_2})^{*},
\qquad
v\otimes\phi\longmapsto f_{\sharp}(v)\otimes(f^{\sharp})^{*}\phi.
\]
Let \(u\in(\Lambda_{X_2}^{\prime})^{*}\), \(\eta\in\Lambda_{X_2}\),
\(v\in\Lambda_{X_1}^{\prime}\), and \(\phi\in(\Lambda_{X_1})^{*}\). One obtains
\begin{equation}\label{eq:pairing-pullback-pushforward-simple-tensors}
\begin{aligned}
\langle f^{*}(u\otimes\eta),\,v\otimes\phi\rangle_{X_1}
&=
\left\langle (f_{\sharp})^{*}u\otimes f^{\sharp}(\eta),\,v\otimes\phi\right\rangle_{X_1}
=
((f_{\sharp})^{*}u)(v)\,\phi(f^{\sharp}(\eta)) \\
&=
u(f_{\sharp}(v))\,((f^{\sharp})^{*}\phi)(\eta)
=
\left\langle u\otimes\eta,\,f_{\sharp}(v)\otimes(f^{\sharp})^{*}\phi\right\rangle_{X_2} \\
&=
\langle u\otimes\eta,\,\widetilde f_{*}(v\otimes\phi)\rangle_{X_2},
\end{aligned}
\end{equation}
where the first equality follows from \eqref{eq:pullback-one-one-lattice}, and the second and fourth equalities follow from \eqref{eq:pairing-one-one}. By bilinearity and \eqref{eq:pairing-pullback-pushforward-simple-tensors}, for every
\(\alpha\in H^{1,1}(X_2)\) and every \(\beta\in H_{1,1}(X_1)\), one has
\[
\langle \alpha,\,\widetilde f_{*}\beta\rangle_{X_2}
=
\langle f^{*}\alpha,\,\beta\rangle_{X_1}
=
\langle \alpha,\,f_{*}\beta\rangle_{X_2},
\]
where the first equality follows from bilinearity and \eqref{eq:pairing-pullback-pushforward-simple-tensors}, while the second equality follows from Proposition~\ref{prop:GS19-toolkit}\textup{(ii)}. Hence
\[
\langle \alpha,\,\widetilde f_{*}\beta-f_{*}\beta\rangle_{X_2}=0,
\]
for every \(\alpha\in H^{1,1}(X_2)\). Since
\(\langle\cdot,\cdot\rangle_{X_2}\) is nondegenerate,
\(\widetilde f_{*}\beta=f_{*}\beta\). Since
\(\beta\in H_{1,1}(X_1)\) is arbitrary,
\(\widetilde f_{*}=f_{*}\).
\end{proof}

\begin{defi}\label{def:T-cycle-class}
Let \(X\) be a \(g\)-dimensional integral torus. Fix \(\Z\)-bases
\(\{\lambda_i'\}_{i=1}^{g}\) of \(\Lambda_X'\) and
\(\{\eta_j\}_{j=1}^{g}\) of \(\Lambda_X\), and let
\(\{\eta_j^*\}_{j=1}^{g}\) be the dual basis of
\(\{\eta_j\}_{j=1}^{g}\).
For a tropical \(1\)-cycle \(A\) on \(X\), write
\[
\cyc_X(A)
=
\sum_{i,j=1}^{g}
M^A_{i,j}\lambda_i'\otimes\eta_j^*,
\qquad
(M^A_{i,j})_{1\le i,j\le g}\in M_g(\Z).
\]
The \({\R}\)-linear map associated with \(\cyc_X(A)\) is defined by
\[
T_A\colon(\Lambda_X)_{\R}\longrightarrow(\Lambda_X')_{\R},
\qquad
v\longmapsto
\sum_{i,j=1}^{g}
M^A_{i,j}\eta_j^*(v)\lambda_i'.
\]
\end{defi}

\begin{lemma}\label{lem:T-pushforward}
For \(i=1,2\), put \(g_i=\dim X_i\), fix \(\Z\)-bases
\(\{\lambda_{i,k}'\}_{k=1}^{g_i}\) of \(\Lambda_{X_i}'\) and
\(\{\eta_{i,k}\}_{k=1}^{g_i}\) of \(\Lambda_{X_i}\), and let
\(\{\eta_{i,k}^*\}_{k=1}^{g_i}\) be the dual basis of
\(\{\eta_{i,k}\}_{k=1}^{g_i}\).
Let \(C\) be a tropical \(1\)-cycle on \(X_1\), and let \(T_C\) and
\(T_{f_*C}\) be the \({\R}\)-linear maps defined with respect to the fixed bases
related to \(X_1\) and \(X_2\), respectively, as in
Definition~\ref{def:T-cycle-class}. Then
\[
T_{f_*C}
=
f_\sharp\circ T_C\circ f^\sharp.
\]
\end{lemma}

\begin{proof}
Write
\[
\cyc_{X_1}(C)
=
\sum_{i,j=1}^{g_1}
M^C_{i,j}\lambda_{1,i}'\otimes\eta_{1,j}^*,
\qquad
(M^C_{i,j})_{1\le i,j\le g_1}\in M_{g_1}(\Z).
\]
By Proposition~\ref{prop:GS19-toolkit}\textup{(iii)} and
Lemma~\ref{lem:push-PPTAV-en}, one has
\[
\cyc_{X_2}(f_*C)
=
f_*\cyc_{X_1}(C)
=
\sum_{i,j=1}^{g_1}
M^C_{i,j}
f_\sharp(\lambda_{1,i}')
\otimes
(f^\sharp)^*(\eta_{1,j}^*).
\]
Hence, for every \(v\in(\Lambda_{X_2})_{\R}\), one has
\[
T_{f_*C}(v)
=
\sum_{i,j=1}^{g_1}
M^C_{i,j}
\eta_{1,j}^*\bigl(f^\sharp(v)\bigr)
f_\sharp(\lambda_{1,i}')
=
f_\sharp\bigl(T_C(f^\sharp(v))\bigr).
\]
This proves the assertion.
\end{proof}

\section{Tropical curves and morphisms}\label{sec:tropical-curves-morphisms}
This section establishes Lemma~\ref{lem:cycle-from-smooth-curve}, which will be used in the proofs of the Tropical Matsusaka criterion, Theorem~\ref{thm:tropical-matsusaka-criterion} and Proposition~\ref{prop:equivalence-conditions-two-criteria}.

\subsection{Parametrizing effective tropical \(1\)-cycles by smooth tropical curves}\label{subsec:weight-one-universal-property}
This subsection proves Lemma~\ref{lem:cycle-from-smooth-curve}, which will be used in Subsections~\ref{subsec:proof-tropical-matsusaka} and~\ref{subsec:proof-tropical-matsusaka-ran}. 
\begin{lemma}\label{lem:cycle-from-smooth-curve}
Let \(X\) be a \(g\)-dimensional integral torus, and let \(C\) be a nonzero effective tropical \(1\)-cycle on \(X\) such that \(|C|\) is connected. Then there exist a smooth tropical curve \(\Gamma\) and a morphism of rational polyhedral spaces \(\chi\colon \Gamma\to X\) such that \(\chi_*[\Gamma]=C\).
\end{lemma}
\begin{proof}
After refining the global face structure on \(|C|\) if necessary, assume that, for every vertex \(v\) of \(|C|\), the union of the edges incident to \(v\) is contained in a chart of \(X\). For each edge \(\tau\) of \(|C|\), let \(w_\tau\in\Z_{>0}\) denote its weight, and let \(\ell(\tau)\) denote its length as defined in Definition~\ref{def:edge-tangent-length}\textup{(ii)}.

We divide the proof into steps.

\smallskip
\noindent\textbf{Step \textup{(i)}.}
We construct a smooth tropical curve \(\Gamma\). For each vertex \(v\) of \(|C|\), let \(\tau_{v,1},\dots,\tau_{v,m_v}\) be the edges incident to \(v\), where
\(
m_v\coloneqq\#\{\tau\mid v\in\tau\}\in\Z_{\ge2}.
\)
Let
\[
L_{m_v-1}^{v}\coloneqq
\bigcup_{i=1}^{m_v}\R_{\ge0}\overline{\boldsymbol e}_{i}^{v}
\subseteq
\R^{m_v}/\R(1,\dots,1),
\]
where \(\boldsymbol e_i^v\) is the \(i\)-th standard basis vector in \(\R^{m_v}\), and \(\overline{\boldsymbol e}_{i}^{v}\) denotes its image in \(\R^{m_v}/\R(1,\dots,1)\). Define
\[
W_v\coloneqq
\bigcup_{i=1}^{m_v}
\left[0,\frac{\ell(\tau_{v,i})}{w_{\tau_{v,i}}}\right)
\overline{\boldsymbol e}_{i}^{v}
\subseteq L_{m_v-1}^{v}.
\]
Let \(\tau\) be an edge of \(|C|\) with endpoints \(v\) and \(w\), and choose indices \(i_v\) and \(j_w\) such that
\(
\tau=\tau_{v,i_v}=\tau_{w,j_w}.
\)
Set
\[
I_{v,\tau}\coloneqq
\left(0,\frac{\ell(\tau)}{w_\tau}\right)
\overline{\boldsymbol e}_{i_v}^{v},
\qquad
I_{w,\tau}\coloneqq
\left(0,\frac{\ell(\tau)}{w_\tau}\right)
\overline{\boldsymbol e}_{j_w}^{w}.
\]
We identify \(I_{v,\tau}\subseteq W_v\) with \(I_{w,\tau}\subseteq W_w\) via the integral affine isomorphism
\[
\theta_{vw}\colon I_{v,\tau}\longrightarrow I_{w,\tau},
\qquad
\theta_{vw}\bigl(r\overline{\boldsymbol e}_{i_v}^{v}\bigr)
=
\left(\frac{\ell(\tau)}{w_\tau}-r\right)
\overline{\boldsymbol e}_{j_w}^{w}.
\]
These identifications give a compact connected rational polyhedral space \(\Gamma\). Since \(W_v\) is an open subset of \(L_{m_v-1}^{v}\) for every vertex \(v\) of \(|C|\), \(\Gamma\) is a smooth tropical curve. For each vertex \(v\) of \(|C|\), we use the same symbol \(v\) for the vertex of \(\Gamma\) represented by the origin in \(W_v\). For each edge \(\tau\) of \(|C|\), let \(\widetilde{\tau}\) denote the corresponding edge of \(\Gamma\). These vertices and edges form a global face structure on \(\Gamma\).

\smallskip
\noindent\textbf{Step \textup{(ii)}.}
We construct a continuous map \(\chi\colon\Gamma\to X\). For each vertex \(v\) of \(|C|\), choose a chart
\[
\varphi_v\colon U_v\longrightarrow V_v\subseteq\R^g
\]
of \(X\) such that \(\varphi_v(v)=\boldsymbol 0\) and every edge of \(|C|\) incident to \(v\) is contained in \(U_v\). For \(1\le i\le m_v\), let
\(
\boldsymbol u_{\tau_{v,i},v}\in\Z^g
\)
be the tangent direction of \(\tau_{v,i}\) out of \(v\) in this chart. The balancing condition at \(v\) gives
\[
\sum_{i=1}^{m_v}
w_{\tau_{v,i}}\boldsymbol u_{\tau_{v,i},v}
=
\boldsymbol 0.
\]
Since
\(
\overline{\boldsymbol e}_{1}^{v}+\cdots+
\overline{\boldsymbol e}_{m_v}^{v}=\boldsymbol 0,
\)
the balancing equality defines an integral linear map
\[
\chi_v\colon
\R^{m_v}/\R(1,\dots,1)
\longrightarrow
\R^g
\]
such that
\begin{equation}\label{eq:chi-v-local-map}
\chi_v(\overline{\boldsymbol e}_{i}^{v})
=
w_{\tau_{v,i}}\boldsymbol u_{\tau_{v,i},v},
\qquad
1\le i\le m_v.
\end{equation}
One has \(\chi_v(W_v)\subseteq V_v\), so the map
\[
\varphi_v^{-1}\circ(\chi_v|_{W_v})\colon W_v\longrightarrow X
\]
is well-defined. It sends the origin to \(v\).
{For \(0\le r<\ell(\tau_{v,i})/w_{\tau_{v,i}}\), the map \(\varphi_v^{-1}\circ(\chi_v|_{W_v})\) sends \(r\overline{\boldsymbol e}_{i}^{v}\) to a point of \(\tau_{v,i}\), and the segment from \(v\) to this point has length \(w_{\tau_{v,i}}r\).}

Let \(\tau\) be an edge of \(|C|\) with endpoints \(v\) and \(w\), and retain the notation from Step~\textup{(i)}. For
\(
0<r<\ell(\tau)/w_\tau,
\)
one has
\[
\bigl(\varphi_v^{-1}\circ(\chi_v|_{W_v})\bigr)
\bigl(r\overline{\boldsymbol e}_{i_v}^{v}\bigr)
=
\bigl(\varphi_w^{-1}\circ(\chi_w|_{W_w})\bigr)
\left(
\left(\frac{\ell(\tau)}{w_\tau}-r\right)
\overline{\boldsymbol e}_{j_w}^{w}
\right).
\]
Hence these local maps are compatible with the identifications \(\theta_{vw}\) and glue to a continuous map
\[
\chi\colon\Gamma\longrightarrow X.
\]
For every vertex \(v\), one has
\[
\varphi_v\circ(\chi|_{W_v})
=
\chi_v|_{W_v}.
\]
By construction, \(\chi\) maps each edge \(\widetilde{\tau}\) homeomorphically onto \(\tau\), and its image is \(|C|\).

\smallskip
\noindent\textbf{Step \textup{(iii)}.}
We verify that \(\chi\) is a morphism of rational polyhedral spaces. Let \(U\subseteq X\) be open and let \(h\in\operatorname{Aff}_X(U)\). {Let} {\(v\) be any} vertex of \(\Gamma\){.} {Then}
\[
\varphi_v\circ(\chi|_{W_v})
=
\chi_v|_{W_v}.
\]
{Since} \(\chi_v\) is integral linear{,} \(h\circ\chi\) is integral affine on every connected component of \(W_v\cap\chi^{-1}(U)\). Hence pullback along \(\chi\) induces a morphism of sheaves
\[
\chi^*\colon
\chi^{-1}\operatorname{Aff}_X
\longrightarrow
\operatorname{Aff}_{\Gamma}.
\]
{The} {map} \(\chi\) {is} {continuous by} Step~\textup{(ii)}{.} {Hence} \(\chi\) is a morphism of rational polyhedral spaces {by Definition~\ref{def:rational-polyhedral-morphism}}.

\smallskip
\noindent\textbf{Step \textup{(iv)}.}
We prove that \(\chi_*[\Gamma]=C\). Since \(\Gamma\) is compact, \(\chi\) is proper. The morphism \(\chi\) induces a morphism
\[
\bar\chi\colon\Gamma\longrightarrow|C|
\]
whose underlying continuous map is a homeomorphism. Let \(\tau\) be an edge of \(|C|\), let \(\widetilde{\tau}\) be the corresponding edge of \(\Gamma\), and let \(x\in\relint(\widetilde{\tau})\). By \eqref{eq:chi-v-local-map} and the construction of \(\chi\), one has
\[
d_x\bar\chi(T_x^{\Z}\Gamma)
=
w_\tau T_{\bar\chi(x)}^{\Z}|C|.
\]
Hence
\[
\bigl[
T_{\bar\chi(x)}^{\Z}|C|
:
d_x\bar\chi(T_x^{\Z}\Gamma)
\bigr]
=
w_\tau.
\]
Since \([\Gamma]\) has weight \(1\) on \(\widetilde{\tau}\), Definition~\ref{def:proper-pushforward} shows that \(\chi_*[\Gamma]\) has weight \(w_\tau\) along \(\tau\). Therefore \(\chi_*[\Gamma]=C\).
\end{proof}

\section{Orthogonal decompositions}\label{sec:orthogonal-decompositions}
{In this section, we study orthogonal decompositions of principally polarized tropical abelian varieties and prove the splitting criterion, Proposition~\ref{prop:split-off-subppav}. Then we prove Proposition~\ref{prop:vanishing-on-other-Gamma}. This proposition prepares for the proof of the tropical Matsusaka criterion in Section~\ref{sec:tropical-matsusaka}.}
{Throughout this section, let \((X,\Theta)\) be a principally polarized tropical abelian variety. Write \(X=(\Lambda_X,\Lambda_X^{\prime},[\,\cdot\,,\,\cdot\,]_X)\).}
{Let} \(\kappa_\Theta\colon\Lambda_X^{\prime}\to\Lambda_X\) correspond to
\(c_1(\MO_X(\Theta))\), and set
\[
{Q_\Theta(u,v)=[\kappa_\Theta(u),v]_X,
\qquad u,v\in(\Lambda_X^{\prime})_{\R}.}
\]
{As recalled in Subsection~\ref{subsec:integral-tori-polarizations}, \(\kappa_\Theta\) is a lattice isomorphism and \(Q_\Theta\) is symmetric positive definite.}

{
\begin{defi}\label{def:Q-orthogonal-decomposition}
Let \(Y,Z\) be integral tori and let \(\alpha=(\alpha^\sharp,\alpha_\sharp)\colon Y\times Z\to X\) be an isomorphism of integral tori. We say that \(\alpha\) is \emph{\(Q_\Theta\)-orthogonal} if
\begin{equation}\label{eq:Q-alpha-orthogonal}
Q_\Theta\bigl(\alpha_\sharp(u',0),\alpha_\sharp(0,v')\bigr)=0
\end{equation}
for all \(u'\in(\Lambda_Y^{\prime})_{\R}\) and \(v'\in(\Lambda_Z^{\prime})_{\R}\).
\end{defi}
}

{
\begin{lemma}\label{lem:orthogonal-splitting-from-split-off}
Let
\(\beta=(\beta^\sharp,\beta_\sharp)\colon(Y,\Theta_Y)\times(Z,\Theta_Z)\to(X,\Theta)\)
be an isomorphism of principally polarized tropical abelian varieties. Set \(L^{\prime}\coloneqq\beta_\sharp(\Lambda_Y^{\prime}\oplus0)\) and \(L^{\prime\prime}\coloneqq\beta_\sharp(0\oplus\Lambda_Z^{\prime})\).
Then \(\Lambda_X^{\prime}=L^{\prime}\oplus L^{\prime\prime}\) and
\(Q_\Theta(L^{\prime},L^{\prime\prime})=0\).
\end{lemma}
\begin{proof}
{Since \(\beta_\sharp\) is a lattice isomorphism,
\[
\Lambda_X^{\prime}
=\beta_\sharp(\Lambda_Y^{\prime}\oplus\Lambda_Z^{\prime})
=\beta_\sharp(\Lambda_Y^{\prime}\oplus0)\oplus\beta_\sharp(0\oplus\Lambda_Z^{\prime})
=L^{\prime}\oplus L^{\prime\prime}.
\]
}
Let \(u'\in(\Lambda_Y^{\prime})_{\R}\) and \(v'\in(\Lambda_Z^{\prime})_{\R}\). By \eqref{eq:kappa-pullback-divisor} and \eqref{eq:dual-hom-pairing-compat},
\[
Q_\Theta\bigl(\beta_\sharp(u',0),\beta_\sharp(0,v')\bigr)
=Q_{\beta^*\Theta}\bigl((u',0),(0,v')\bigr).
\]

By Definition~\ref{def:isomorphism-ppav} and \eqref{eq:chern-class-homological-equivalence},
\[
\begin{split}
Q_{\beta^*\Theta}\bigl((u',0),(0,v')\bigr)
&=Q_{{\pr_1^*}\Theta_Y+{\pr_2^*}\Theta_Z}\bigl((u',0),(0,v')\bigr)\\
&=Q_{\Theta_Y}(u',0)+Q_{\Theta_Z}(0,v')=0,
\end{split}
\]
where the second equality follows from \eqref{eq:Q-product-polarization}.
\end{proof}
}

\subsection{The splitting criterion}\label{subsec:splitting-criterion}
{In this subsection, we prove the splitting criterion, Proposition~\ref{prop:split-off-subppav}.} Let \(Y=(\Lambda_Y,\Lambda_Y^{\prime},[\,\cdot\,,\,\cdot\,]_Y)\) be an integral torus and let
\(i=(i^{\sharp},i_{\sharp})\colon Y\to X\) be an injective homomorphism of integral tori.

\begin{prop}\label{prop:split-off-subppav}
Set \(\Theta_Y\coloneqq i^{*}\Theta\) and assume that \(\Theta_Y\) is a principal polarization. Then there exist a principally polarized tropical abelian variety \((Z,\Theta_Z)\) and an isomorphism
\[
\beta\colon(Y,\Theta_Y)\times(Z,\Theta_Z)\longrightarrow(X,\Theta)
\]
of principally polarized tropical abelian varieties such that
\(\beta\circ\iota_1=i\), where
\(\iota_1\colon Y\to Y\times Z\) is the canonical inclusion of the first factor, defined in Subsection~\ref{subsec:products-integral-tori}.
\end{prop}

\begin{proof}
The lattice homomorphisms of \(i\) are
\[
i^\sharp\colon\Lambda_X\longrightarrow\Lambda_Y,
\qquad i_\sharp\colon\Lambda_Y^{\prime}\longrightarrow\Lambda_X^{\prime}.
\]
The dual homomorphism is \(i^\vee=(i_\sharp,i^\sharp)\colon X^\vee\to Y^\vee\).
Write \(\varphi_\Theta\colon X\to X^\vee\) and
\(\varphi_{\Theta_Y}\colon Y\to Y^\vee\) for the homomorphisms induced by
\(\MO_X(\Theta)\) and \(\MO_Y(\Theta_Y)\), as in Subsection~\ref{subsec:polarization-induced-homomorphisms}.
Proposition~\ref{prop:phi-Theta-dual-kappa} gives
\[
\varphi_\Theta^\sharp=(\varphi_\Theta)_\sharp=-\kappa_\Theta,
\qquad
\varphi_{\Theta_Y}^\sharp=(\varphi_{\Theta_Y})_\sharp=-\kappa_{\Theta_Y}.
\]
Since \(\Theta\) and \(\Theta_Y\) are principal, \(\kappa_\Theta\) and \(\kappa_{\Theta_Y}\) are lattice isomorphisms.
{Thus \(\varphi_\Theta\) and \(\varphi_{\Theta_Y}\) are isomorphisms of integral tori.}
Set
\begin{equation}\label{eq:r-definition}
r\coloneqq\varphi_{\Theta_Y}^{-1}\circ i^\vee\circ\varphi_\Theta
\colon X\longrightarrow Y.
\end{equation}
{The lattice homomorphisms of \(r\) are \(r^\sharp\colon\Lambda_Y\to\Lambda_X\) and \(r_\sharp\colon\Lambda_X^{\prime}\to\Lambda_Y^{\prime}\).}
{Since \((i^\vee)_\sharp=i^\sharp\) and
\((\varphi_{\Theta_Y}^{-1})_\sharp=\bigl((\varphi_{\Theta_Y})_\sharp\bigr)^{-1}\),
\eqref{eq:r-definition} gives
\begin{equation}\label{eq:r-lower-sharp}
r_\sharp=(\varphi_{\Theta_Y}^{-1})_\sharp\circ(i^\vee)_\sharp\circ(\varphi_\Theta)_\sharp=\kappa_{\Theta_Y}^{-1}\circ i^\sharp\circ\kappa_\Theta.
\end{equation}}
Since \(\Theta_Y=i^*\Theta\), Lemma~\ref{lem:phi-pullback-dual} gives
\(\varphi_{\Theta_Y}=i^\vee\circ\varphi_\Theta\circ i\). Hence
\begin{equation}\label{eq:r-i-identity}
{r\circ i
=\varphi_{\Theta_Y}^{-1}\circ(i^\vee\circ\varphi_\Theta\circ i)
=\varphi_{\Theta_Y}^{-1}\circ\varphi_{\Theta_Y}
=\id_Y.}
\end{equation}

\smallskip
\noindent\textbf{Step \textup{(i)}.}
We construct the integral torus \(Z\) in Proposition~\ref{prop:split-off-subppav}. {We first prove two claims.}

\begin{claim}\label{claim:lattice-decompositions}
One has
\begin{equation}\label{eq:lattice-decompositions}
\Lambda_X^{\prime}=i_\sharp(\Lambda_Y^{\prime})\oplus\ker(r_\sharp),
\qquad
\Lambda_X=r^\sharp(\Lambda_Y)\oplus\ker(i^\sharp).
\end{equation}
\end{claim}

\begin{proof}[Proof of Claim~\ref{claim:lattice-decompositions}]
By \eqref{eq:r-i-identity}, \(r_\sharp\circ i_\sharp=\id_{\Lambda_Y^{\prime}}\) and \(i^\sharp\circ r^\sharp=\id_{\Lambda_Y}\). For every \(\lambda'\in\Lambda_X^{\prime}\),
\[
\lambda'=(i_\sharp\circ r_\sharp)(\lambda')+
\bigl(\lambda'-(i_\sharp\circ r_\sharp)(\lambda')\bigr).
\]
Applying \(r_\sharp\) to the second summand above gives
\[
r_\sharp\bigl(\lambda'-(i_\sharp\circ r_\sharp)(\lambda')\bigr)
=r_\sharp(\lambda')-r_\sharp(\lambda')=0.
\]
Thus \(\lambda'-(i_\sharp\circ r_\sharp)(\lambda')\in\ker(r_\sharp)\). If \(u'\in\Lambda_Y^{\prime}\) satisfies \(i_\sharp(u')\in\ker(r_\sharp)\), then \(u'=(r_\sharp\circ i_\sharp)(u')=0\). Hence \(i_\sharp(\Lambda_Y^{\prime})\cap\ker(r_\sharp)=\{0\}\).
The second equality in \eqref{eq:lattice-decompositions} follows from the same argument.
\end{proof}

Set \(\Lambda_Z\coloneqq\ker(i^\sharp)\) and \(\Lambda_Z^{\prime}\coloneqq\ker(r_\sharp)\).

\begin{claim}\label{claim:complementary-pairing-nondegenerate}
The restriction of \([\,\cdot\,,\,\cdot\,]_X\) to \(\Lambda_Z\times\Lambda_Z^{\prime}\) is nondegenerate.
\end{claim}

\begin{proof}[Proof of Claim~\ref{claim:complementary-pairing-nondegenerate}]
For \(\mu\in(\Lambda_Y)_{\R}\), \(u'\in(\Lambda_Y^{\prime})_{\R}\), \(z\in(\Lambda_Z)_{\R}\), and \(z'\in(\Lambda_Z^{\prime})_{\R}\),
\eqref{eq:dual-hom-pairing-compat} gives
\begin{equation}\label{eq:complementary-pairings}
[r^\sharp(\mu),z']_X=[\mu,r_\sharp(z')]_Y=0,
\qquad
[z,i_\sharp(u')]_X=[i^\sharp(z),u']_Y=0.
\end{equation}

Let \(z\in(\Lambda_Z)_{\R}\), and assume that \([z,z']_X=0\) for every \(z'\in(\Lambda_Z^{\prime})_{\R}\).
The second equality in \eqref{eq:complementary-pairings} gives \([z,i_\sharp(u')]_X=0\) for every \(u'\in(\Lambda_Y^{\prime})_{\R}\).
The first equality in \eqref{eq:lattice-decompositions} therefore gives \([z,\lambda']_X=0\) for every \(\lambda'\in(\Lambda_X^{\prime})_{\R}\), so \(z=0\).
Similarly, if \(z'\in(\Lambda_Z^{\prime})_{\R}\) satisfies \([z,z']_X=0\) for every \(z\in(\Lambda_Z)_{\R}\), the first equality in \eqref{eq:complementary-pairings} and the second equality in \eqref{eq:lattice-decompositions} give \(z'=0\).
\end{proof}

By Claim~\ref{claim:lattice-decompositions} and Claim~\ref{claim:complementary-pairing-nondegenerate}, \(Z\coloneqq\bigl(\Lambda_Z,\Lambda_Z^{\prime},[\,\cdot\,,\,\cdot\,]_X|_{\Lambda_Z\times\Lambda_Z^{\prime}}\bigr)\) is an integral torus.

\smallskip
\noindent\textbf{Step \textup{(ii)}.}
We construct a homomorphism of integral tori \(j\colon Z\to X\). There are natural inclusions \(\iota_Z\) from \(\Lambda_Z\) to \(\Lambda_X\) and \(\iota_Z'\) from \(\Lambda_Z^{\prime}\) to \(\Lambda_X^{\prime}\); we identify each lattice with its image. Consider the lattice homomorphisms
\(j^\sharp\colon\Lambda_X\to\Lambda_X\) and
\(j_\sharp\colon\Lambda_Z^{\prime}\to\Lambda_X^{\prime}\) given respectively by
\begin{equation}\label{eq:j-lattice-homomorphisms}
{j^\sharp(\lambda)\coloneqq\lambda-(r^\sharp\circ i^\sharp)(\lambda),}
\qquad
j_\sharp(z')\coloneqq z'.
\end{equation}
{Since \(i^\sharp\circ r^\sharp=\id_{\Lambda_Y}\), for \(\lambda\in\Lambda_X\) we have \((i^\sharp\circ j^\sharp)(\lambda)=i^\sharp(\lambda)-i^\sharp(\lambda)=0\).}
Thus \(j^\sharp(\Lambda_X)\subseteq\Lambda_Z\), and hence \(j^\sharp\colon\Lambda_X\to\Lambda_Z\).
Let \(\lambda\in\Lambda_X\) and \(z'\in\Lambda_Z^{\prime}\). Recall that \(\Lambda_Z^{\prime}=\ker(r_\sharp)\), so \(r_\sharp(z')=0\). Applying \eqref{eq:dual-hom-pairing-compat} to \(r\colon X\to Y\) gives
\[
[(r^\sharp\circ i^\sharp)(\lambda),z']_X
=[i^\sharp(\lambda),r_\sharp(z')]_Y=0.
\]
For \(\lambda\in\Lambda_X\) and \(z'\in\Lambda_Z^{\prime}\), the preceding equality and \eqref{eq:j-lattice-homomorphisms} give
\begin{equation}\label{eq:j-pairing-compatibility}
[\lambda,j_\sharp(z')]_X
=[\lambda,z']_X
=[\lambda-(r^\sharp\circ i^\sharp)(\lambda),z']_X
=[j^\sharp(\lambda),z']_Z.
\end{equation}
Therefore, \(j\coloneqq(j^\sharp,j_\sharp)\colon Z\to X\) is a homomorphism of integral tori by \eqref{eq:j-pairing-compatibility}.

\smallskip
\noindent\textbf{Step \textup{(iii)}.}
{We prove that \(\Theta_Z\coloneqq j^*\Theta\) is a principal polarization. Since \(\kappa_{\Theta_Y}\) is a lattice isomorphism, \eqref{eq:r-lower-sharp} gives, for \(z'\in\Lambda_X^{\prime}\), \(r_\sharp(z')=0\) if and only if \((i^\sharp\circ\kappa_\Theta)(z')=0\).}
By the definitions of \(\Lambda_Z\) and \(\Lambda_Z^{\prime}\),
\[
\Lambda_Z^{\prime}
=\ker(r_\sharp)
=\ker(i^\sharp\circ\kappa_\Theta)
=\kappa_\Theta^{-1}(\Lambda_Z).
\]
Then one has \(\kappa_\Theta(\Lambda_Z^{\prime})=\Lambda_Z\).
By \eqref{eq:kappa-pullback-divisor}, \eqref{eq:j-lattice-homomorphisms}, and \(\kappa_\Theta(\Lambda_Z^{\prime})=\Lambda_Z\), the homomorphism \(\kappa_{\Theta_Z}\colon\Lambda_Z^{\prime}\to\Lambda_Z\) is a lattice isomorphism. For \(z_1',z_2'\in(\Lambda_Z^{\prime})_{\R}\), the definitions of \(Q_{\Theta_Z}\) and \(Q_\Theta\), \eqref{eq:kappa-pullback-divisor}, and \eqref{eq:j-pairing-compatibility} give
\begin{equation}\label{eq:Q-Theta-Z-restriction}
\begin{split}
Q_{\Theta_Z}(z_1',z_2')
&=[\kappa_{\Theta_Z}(z_1'),z_2']_Z
=\bigl[(j^\sharp\circ\kappa_\Theta\circ j_\sharp)(z_1'),z_2'\bigr]_Z
=[\kappa_\Theta(j_\sharp(z_1')),j_\sharp(z_2')]_X\\
&=Q_\Theta\bigl(j_\sharp(z_1'),j_\sharp(z_2')\bigr)
=Q_\Theta(z_1',z_2'),
\end{split}
\end{equation}
where the last equality follows from \eqref{eq:j-lattice-homomorphisms}. Hence \(\Theta_Z\) is a principal polarization.
The two sublattices \(i_\sharp(\Lambda_Y^{\prime})\) and
\(\Lambda_Z^{\prime}\) are \(Q_\Theta\)-orthogonal.
{Indeed, for \(u'\in(\Lambda_Y^{\prime})_{\R}\) and \(z'\in(\Lambda_Z^{\prime})_{\R}\), the symmetry of \(Q_\Theta\), \eqref{eq:dual-hom-pairing-compat}, and \eqref{eq:r-lower-sharp} give}
\begin{equation}\label{eq:Q-Theta-complementary-orthogonality}
\begin{split}
Q_\Theta(i_\sharp(u'),z')
&=[\kappa_\Theta(z'),i_\sharp(u')]_X
=[(i^\sharp\circ\kappa_\Theta)(z'),u']_Y\\
&=[(\kappa_{\Theta_Y}\circ r_\sharp)(z'),u']_Y=0,
\end{split}
\end{equation}
where \(r_\sharp(z')=0\) since \(\Lambda_Z^{\prime}=\ker(r_\sharp)\).

\smallskip
\noindent\textbf{Step \textup{(iv)}.}
{We construct the isomorphism in Proposition~\ref{prop:split-off-subppav}. Define lattice homomorphisms \(\beta^\sharp\colon\Lambda_X\to\Lambda_Y\oplus\Lambda_Z\) and \(\beta_\sharp\colon\Lambda_Y^{\prime}\oplus\Lambda_Z^{\prime}\to\Lambda_X^{\prime}\) by
\begin{equation}\label{eq:beta-lattice-homomorphisms}
\beta^\sharp(\lambda)=
\bigl(i^\sharp(\lambda),\lambda-(r^\sharp\circ i^\sharp)(\lambda)\bigr),
\qquad
\beta_\sharp(u',z')=i_\sharp(u')+z'.
\end{equation}
Let \(\lambda\in\Lambda_X\). The definition of \(j^\sharp\) gives \(\beta^\sharp(\lambda)=(i^\sharp(\lambda),j^\sharp(\lambda))\). For \(u'\in\Lambda_Y^{\prime}\) and \(z'\in\Lambda_Z^{\prime}\), \eqref{eq:product-pairing} gives
\begin{equation}\label{eq:beta-product-pairing}
[\beta^\sharp(\lambda),(u',z')]_{Y\times Z}
=[i^\sharp(\lambda),u']_Y+[j^\sharp(\lambda),z']_Z.
\end{equation}
By \eqref{eq:dual-hom-pairing-compat} and \eqref{eq:j-pairing-compatibility},
\begin{equation}\label{eq:beta-component-pairings}
[i^\sharp(\lambda),u']_Y=[\lambda,i_\sharp(u')]_X,
\qquad
[j^\sharp(\lambda),z']_Z=[\lambda,z']_X.
\end{equation}
By \eqref{eq:beta-product-pairing} and \eqref{eq:beta-component-pairings},
\[
[\beta^\sharp(\lambda),(u',z')]_{Y\times Z}
=[\lambda,i_\sharp(u')+z']_X
=[\lambda,\beta_\sharp(u',z')]_X.
\]
Thus \(\beta\coloneqq(\beta^\sharp,\beta_\sharp)\colon Y\times Z\to X\) is a homomorphism of integral tori.}
The two equalities in \eqref{eq:lattice-decompositions} give the inverse lattice homomorphisms
\[
(\beta^\sharp)^{-1}(\mu,z)=r^\sharp(\mu)+z,\ \text{and}\ 
(\beta_\sharp)^{-1}(\lambda')=\bigl(r_\sharp(\lambda'),\lambda'-(i_\sharp\circ r_\sharp)(\lambda')\bigr),
\quad(\mu,z,\lambda')\in\Lambda_Y\times\Lambda_Z\times\Lambda_X^{\prime}.
\]
Therefore \(\beta\) is an isomorphism of integral tori.
{For \(\lambda\in\Lambda_X\) and \(u'\in\Lambda_Y^{\prime}\), the definition of \(\iota_1\) and \eqref{eq:beta-lattice-homomorphisms} give}
\[
(\beta\circ\iota_1)^\sharp(\lambda)
=(\iota_1)^\sharp\bigl(i^\sharp(\lambda),j^\sharp(\lambda)\bigr)
=i^\sharp(\lambda),
\qquad
(\beta\circ\iota_1)_\sharp(u')
=\beta_\sharp(u',0)=i_\sharp(u').
\]
Hence \(\beta\circ\iota_1=i\).
{Let \(u_1',u_2'\in(\Lambda_Y^{\prime})_{\R}\) and \(z_1',z_2'\in(\Lambda_Z^{\prime})_{\R}\).
{By the second equality in \eqref{eq:beta-lattice-homomorphisms} and \eqref{eq:Q-Theta-complementary-orthogonality},}
\begin{equation}\label{eq:beta-cross-term-vanishing}
\begin{split}
Q_\Theta(i_\sharp(u_1'),z_2')
&=Q_\Theta\bigl(\beta_\sharp(u_1',0),\beta_\sharp(0,z_2')\bigr)=0,\\
Q_\Theta(z_1',i_\sharp(u_2'))
&=Q_\Theta(i_\sharp(u_2'),z_1')
=Q_\Theta\bigl(\beta_\sharp(u_2',0),\beta_\sharp(0,z_1')\bigr)=0.
\end{split}
\end{equation}
{Thus \(\beta\) is \(Q_\Theta\)-orthogonal by \eqref{eq:Q-alpha-orthogonal}.}
Since \(\Theta_Y=i^*\Theta\), \eqref{eq:dual-hom-pairing-compat} and \eqref{eq:kappa-pullback-divisor} give
\begin{equation}\label{eq:Q-Theta-Y-pullback}
\begin{split}
Q_\Theta(i_\sharp(u_1'),i_\sharp(u_2'))
&=[\kappa_\Theta(i_\sharp(u_1')),i_\sharp(u_2')]_X
=[(i^\sharp\circ\kappa_\Theta\circ i_\sharp)(u_1'),u_2']_Y\\
&=[\kappa_{\Theta_Y}(u_1'),u_2']_Y
=Q_{\Theta_Y}(u_1',u_2').
\end{split}
\end{equation}
By \eqref{eq:kappa-pullback-divisor}, \eqref{eq:dual-hom-pairing-compat}, and the second equality in \eqref{eq:beta-lattice-homomorphisms},
\[
\begin{split}
Q_{\beta^*\Theta}\bigl((u_1',z_1'),(u_2',z_2')\bigr)
&=Q_\Theta\bigl(i_\sharp(u_1')+z_1',i_\sharp(u_2')+z_2'\bigr)
=Q_\Theta(i_\sharp(u_1'),i_\sharp(u_2'))+Q_\Theta(z_1',z_2')\\
&=Q_{\Theta_Y}(u_1',u_2')+Q_{\Theta_Z}(z_1',z_2'),
\end{split}
\]
where the second equality follows from \eqref{eq:beta-cross-term-vanishing}, and the last equality uses \eqref{eq:Q-Theta-Z-restriction} and \eqref{eq:Q-Theta-Y-pullback}.
}
{Therefore \(Q_{\beta^*\Theta}=Q_{{\pr_1^*}\Theta_Y+{\pr_2^*}\Theta_Z}\), and hence
\[
c_1\bigl(\MO_{Y\times Z}(\beta^*\Theta)\bigr)
=c_1\bigl(\MO_{Y\times Z}({\pr_1^*}\Theta_Y+{\pr_2^*}\Theta_Z)\bigr).
\]
}
By \eqref{eq:chern-class-homological-equivalence} and Definition~\ref{def:isomorphism-ppav}, \(\beta\) is an isomorphism of principally polarized tropical abelian varieties.
\end{proof}

\begin{remark}\label{rem:tropical-rosati-adjoint}
The homomorphism \(r\) in \eqref{eq:r-definition} is the tropical Rosati adjoint of \(i\) with respect to \(\Theta_Y\) and \(\Theta\).
For the classical construction, see \cite[Section~5.1]{BL04}.
\end{remark}

\subsection{\(Q_{\Theta_{\Gamma}}\)-orthogonal decompositions and simple simplicial cycles}\label{subsec:Q-orthogonal-splittings}

Let \(\Gamma\) be a smooth tropical curve of genus \(g\), and let \(\Theta_{\Gamma}\) be the canonical principal polarization on \(\Jac(\Gamma)\). Fix a finite simplicial structure on \(\Gamma\) {and use the notation} in Definition~\ref{def:tropical-curve}{.} {Let} \(E(\Gamma)\) and \(V(\Gamma)\) {be} its sets of edges and vertices. In this subsection, we consider \(Q_{\Theta_{\Gamma}}\)-orthogonal decompositions of the lattice \(H_1(\Gamma,\Z)\), where \(Q_{\Theta_{\Gamma}}\) is the bilinear form in \eqref{eq:Q-Theta-Gamma-edge-expansion}. Write \(\ell(\tau)\) for the length of each edge \(\tau\in E(\Gamma)\).

\begin{notation}\label{notation}
Choose an oriented \(1\)-simplex \(s_\tau\) on each edge \(\tau\in E(\Gamma)\). The simplicial chain groups on \(\Gamma\) are
\[
C_1(\Gamma,\Z)=\bigoplus_{\tau\in E(\Gamma)}\Z s_\tau,
\qquad C_0(\Gamma,\Z)=\bigoplus_{v\in V(\Gamma)}\Z v.
\]
If \(s\) is directed from \(v\) to \(w\), then \(\partial s=w-v\).
{Let \(Z_1(\Gamma,\Z)\) denote the group of simplicial \(1\)-cycles on \(\Gamma\).}
Since \(\dim\Gamma=1\), one has \(C_2(\Gamma,\Z)=0\), and hence
\[
Z_1(\Gamma,\Z)
=\ker\bigl(\partial\colon C_1(\Gamma,\Z)\to C_0(\Gamma,\Z)\bigr)
=H_1(\Gamma,\Z).
\]
Thus we regard \(H_1(\Gamma,\Z)\) as a sublattice of \(C_1(\Gamma,\Z)\).
For \(c\in H_1(\Gamma,\Z)\), write
\[
c=\sum_{\tau\in E(\Gamma)}c(s_\tau)s_\tau,
\qquad c(s_\tau)\in\Z.
\]
For the opposite orientation \(\bar s_\tau\), one has \(\bar s_\tau=-s_\tau\) and \(c(\bar s_\tau)=-c(s_\tau)\).

Let \(\Omega_\Gamma\) be the sheaf of integral tropical \(1\)-forms on \(\Gamma\).
The canonical principal polarization gives the isomorphism
\[
\kappa_{\Theta_\Gamma}\colon H_1(\Gamma,\Z)\longrightarrow H^0(\Gamma,\Omega_\Gamma).
\]
By \eqref{eq:integral-along-edge} and \eqref{eq:kappa-theta-gamma-characterization},
\begin{equation}\label{eq:cycle-edge-integral}
{\int_{s_\tau}\kappa_{\Theta_\Gamma}(c)=\ell(\tau)c(s_\tau),
\qquad c\in H_1(\Gamma,\Z).}
\end{equation}
For \(\bar s_\tau\), both the integral and \(c(s_\tau)\) change sign.
\end{notation}

\begin{defi}\label{def:QGamma-orthogonal-splitting}
A \emph{\(Q_{\Theta_{\Gamma}}\)-orthogonal decomposition} of \(H_1(\Gamma,\Z)\) is an expression
\(
H_1(\Gamma,\Z)=L^{\prime}\oplus L^{\prime\prime}
\)
with nonzero sublattices \(L^{\prime},L^{\prime\prime}\subseteq H_1(\Gamma,\Z)\) such that \(Q_{\Theta_{\Gamma}}(\alpha,\beta)=0\) for all \(\alpha\in L^{\prime}\) and \(\beta\in L^{\prime\prime}\).
\end{defi}

\begin{defi}\label{def:cycle-support}
Follow Notation~\ref{notation}. For a simplicial \(1\)-chain
\(
b=\sum_s b(s)s\in C_1(\Gamma,\Z),
\)
set
\[
|b|\coloneqq
\bigcup_{\substack{s\\ b(s)\ne 0}}\tau_s,
\]
where \(\tau_s\in E(\Gamma)\) is the support of \(s\).
\end{defi}
\begin{defi}\label{def:positive-cone-one-cycles}
Let \(S\) be a finite set of oriented \(1\)-simplices in \(\Gamma\), no two of which have the same support. Define
\[
Z_1(\Gamma,\Z)_{\ge 0}^{S}
\coloneqq
\left\{
z=\sum_{s\in S}z(s)s\in Z_1(\Gamma,\Z)
\ \middle|\
z(s)\in\Z_{\ge 0}\text{ for every }s\in S
\right\}.
\]
Here \(z(s)\ge0\) uses the orientation of \(s\in S\), and \(z(s)=0\) on the other edges.
\end{defi}
\begin{defi}\label{def:oriented-simple-cycle}
A \emph{simple simplicial cycle} in \(\Gamma\) is a simplicial \(1\)-cycle
\(
c=c_1+\cdots+c_m,
\)
{where \(m\ge 3\), the \(c_i\) are oriented \(1\)-simplices of \(\Gamma\) directed from \(v_{i-1}\) to \(v_i\) for \(1\le i\le m\), \(v_m=v_0\), and the vertices \(v_0,\dots,v_{m-1}\) are pairwise distinct.}
\end{defi}

\begin{lemma}\label{lem:positive-cycle-decomposition}
Let \(S\) be a finite set of oriented \(1\)-simplices in \(\Gamma\), no two of which have the same support. Let
\(0\ne z\in Z_1(\Gamma,\Z)_{\ge 0}^{S}\). Then there exist simple simplicial cycles
\(d_1,\ldots,d_N\in Z_1(\Gamma,\Z)_{\ge 0}^{S}\) such that
\[
z=d_1+\cdots+d_N.
\]
Moreover, if \(s_0\in S\) and \(z(s_0)>0\), then \(d_i(s_0)>0\) for some \(i\in\{1,\ldots,N\}\).
\end{lemma}

\begin{proof}
We first show that, for every nonzero 
\(y\in Z_1(\Gamma,\Z)_{\ge 0}^{S}\), there exist a simple simplicial cycle
\(d\in Z_1(\Gamma,\Z)_{\ge 0}^{S}\) and 
\(y_1\in Z_1(\Gamma,\Z)_{\ge 0}^{S}\) such that
\[
y=d+y_1.
\]
Choose \(s_1\in S\) with \(y(s_1)>0\), and write \(\partial s_1=v_1-v_0\).
Since \(\partial y=0\), some \(s_2\in S\) is directed out of \(v_1\) and has \(y(s_2)>0\).
Repeating this argument gives vertices \(v_0,v_1,\ldots\) and oriented \(1\)-simplices \(s_1,s_2,\ldots\) such that
\[
\partial s_i=v_i-v_{i-1},\qquad y(s_i)>0,\qquad {i\in\Z_{\ge1}}.
\]
The set \(V(\Gamma)\) is finite. Choose
\[
b\coloneqq\min\{j\ge1\mid v_j\in\{v_0,\ldots,v_{j-1}\}\},
\]
and choose \(0\le a<b\) such that \(v_b=v_a\).
Then \(v_0,\ldots,v_{b-1}\) are pairwise distinct. Since the edges form a simplicial structure and \(S\) has no two oriented simplices with the same support, \(b-a\ge3\).
Thus
\[
d\coloneqq s_{a+1}+\cdots+s_b
\]
is a simple simplicial cycle in \(Z_1(\Gamma,\Z)_{\ge0}^{S}\).
For every \(s\in S\), one has \(d(s)\in\{0,1\}\), and \(d(s)=1\) only if \(y(s)>0\).
Since \(y(s)\) is an integer, \(0\le d(s)\le y(s)\). Therefore \(y_1\coloneqq y-d\) belongs to \(Z_1(\Gamma,\Z)_{\ge0}^{S}\), and \(y=d+y_1\).

Starting with \(z^{(0)}\coloneqq z\), apply the preceding paragraph
successively. {At each step \(r\in\Z_{\ge0}\) with \(z^{(r)}\neq0\), write}
\[
z^{(r)}=d_{r+1}+z^{(r+1)}
\]
with \(d_{r+1}\) a simple simplicial cycle in
\(Z_1(\Gamma,\Z)_{\ge 0}^{S}\) and
\(z^{(r+1)}\in Z_1(\Gamma,\Z)_{\ge 0}^{S}\). At each step,
\[
\sum_{s\in S}z^{(r+1)}(s)<\sum_{s\in S}z^{(r)}(s).
\]
Since the left-hand side {above} is a nonnegative integer, the construction stops
after finitely many steps. Therefore
\[
z=d_1+\cdots+d_N.
\]
If \(s_0\in S\) and \(z(s_0)>0\), then taking the coefficient of \(s_0\) gives
\(z(s_0)=\sum_{i=1}^{N}d_i(s_0)\). Since \(z(s_0)>0\), there exists \(i\)
such that \(d_i(s_0)>0\).
\end{proof}

\begin{lemma}\label{lem:harmonic-generated-by-simple-cycles-and-one-summand}
The following hold.
\begin{enumerate}[label=\textup{(\roman*)}]
\item The lattice \(H_1(\Gamma,\Z)\) is generated by the simple simplicial cycles in \(\Gamma\).
\item Assume that \(H_1(\Gamma,\Z)=L^{\prime}\oplus L^{\prime\prime}\) is a \(Q_{\Theta_{\Gamma}}\)-orthogonal decomposition. Then for every simple simplicial cycle \(c\) in \(\Gamma\), precisely one of \(c\in L^{\prime}\) and \(c\in L^{\prime\prime}\) holds.
\end{enumerate}
\end{lemma}

\begin{proof}
For assertion~\textup{(i)}, let \(0\ne c\in H_1(\Gamma,\Z)\).
On each edge in \(|c|\), choose the orientation \(s\) for which \(c(s)>0\), and let \(S\) be the set of these oriented simplices.
Then \(c\in Z_1(\Gamma,\Z)_{\ge0}^{S}\). By Lemma~\ref{lem:positive-cycle-decomposition}, one has \(c=d_1+\cdots+d_N\) with simple simplicial cycles \(d_j\).
The assertion also holds for \(c=0\). {This proves \textup{(i)}.}

For assertion~\textup{(ii)}, let \(c\) be a simple simplicial cycle in
\(\Gamma\). For each edge \(\tau\in E(\Gamma)\), choose an oriented \(1\)-simplex \(s_{\tau}\) supported on \(\tau\), with the convention that \(c(s_{\tau})=1\) whenever \(\tau\subseteq |c|\). Set
\[
S\coloneqq \{s_{\tau}\mid \tau\in E(\Gamma)\}.
\]
Then \(c\in Z_1(\Gamma,\Z)_{\ge 0}^{S}\), and \(c(s_{\tau})=1\) if \(\tau\subseteq |c|\), while \(c(s_{\tau})=0\) if \(\tau\not\subseteq |c|\).
Write \(c=\alpha^{\prime}+\alpha^{\prime\prime}\) with
\(\alpha^{\prime}\in L^{\prime}\) and
\(\alpha^{\prime\prime}\in L^{\prime\prime}\). Then
\begin{equation}\label{eq:Q-alpha-double-prime-alpha-prime}
0
=
Q_{\Theta_{\Gamma}}(\alpha^{\prime\prime},\alpha^{\prime})
=
Q_{\Theta_{\Gamma}}(c,\alpha^{\prime})
-
Q_{\Theta_{\Gamma}}(\alpha^{\prime},\alpha^{\prime}).
\end{equation}

By Notation~\ref{notation}, write \(\alpha^{\prime}=\sum_{\tau\in E(\Gamma)}\alpha^{\prime}(s_{\tau})s_{\tau}\). By \eqref{eq:Q-alpha-double-prime-alpha-prime} and {\eqref{eq:Q-Theta-Gamma-edge-expansion}}, one has
\begin{equation}\label{eq:q-alpha-prime-sum}
0
=
\sum_{\substack{\tau\in E(\Gamma)\\ \tau\subseteq |c|}}
\ell(\tau)\alpha^{\prime}(s_{\tau})
\bigl(1-\alpha^{\prime}(s_{\tau})\bigr)
-
\sum_{\substack{\tau\in E(\Gamma)\\ \tau\not\subseteq |c|}}
\ell(\tau)\alpha^{\prime}(s_{\tau})^2.
\end{equation}
Since \(\alpha^{\prime}\in C_1(\Gamma,\Z)\), every coefficient \(\alpha^{\prime}(s_\tau)\) is an integer. For every \(\tau\subseteq |c|\), one has
\(\alpha^{\prime}(s_{\tau})(1-\alpha^{\prime}(s_{\tau}))\le 0\), and for
every \(\tau\not\subseteq |c|\), one has
\(-\alpha^{\prime}(s_{\tau})^2\le 0\).
Since \(\ell(\tau)>0\) for every edge, {every term on the right-hand side of \eqref{eq:q-alpha-prime-sum} is at most zero}. The sum is zero, so every term is zero. It follows that
\(\alpha^{\prime}(s_{\tau})\in\{0,1\}\) for every
\(\tau\subseteq |c|\), and \(\alpha^{\prime}(s_{\tau})=0\) for every
\(\tau\not\subseteq |c|\). Therefore
\[
\alpha^{\prime}
=
\sum_{\substack{\tau\in E(\Gamma)\\ \tau\subseteq |c|}}
\alpha^{\prime}(s_{\tau})s_{\tau}
\in Z_1(\Gamma,\Z)_{\ge 0}^{S}.
\]
Write \(c=c_1+\cdots+c_m\), with \(c_i\) directed from \(v_{i-1}\) to \(v_i\), as in Definition~\ref{def:oriented-simple-cycle}.
Set \(c_{m+1}=c_1\). For \(1\le i\le m\), the coefficient of \(v_i\) in \(\partial\alpha^{\prime}\) is
\[
\alpha^{\prime}(c_i)-\alpha^{\prime}(c_{i+1})=0.
\]
Thus the coefficient on every edge of \(|c|\) is the same. Hence \(\alpha^{\prime}=0\) or
\(\alpha^{\prime}=c\). Therefore \(c\in L^{\prime\prime}\) or
\(c\in L^{\prime}\). Since \(c\neq 0\) and
\(L^{\prime}\cap L^{\prime\prime}=\{0\}\), precisely one of
\(c\in L^{\prime}\) and \(c\in L^{\prime\prime}\) holds.
\end{proof}

\subsection{Vanishing along complementary edges in \(Q_{\Theta_\Gamma}\)-orthogonal decompositions}\label{subsec:edge-separation}

Follow the notation from Subsection~\ref{subsec:Q-orthogonal-splittings}. Let
\(
H_1(\Gamma,\Z)=L^{\prime}\oplus L^{\prime\prime}
\)
be a \(Q_{\Theta_{\Gamma}}\)-orthogonal decomposition. The goal of this subsection is to prove Proposition~\ref{prop:vanishing-on-other-Gamma}, which will be used later in Section~\ref{sec:tropical-matsusaka}. The proof is adapted, with slight modifications, from the proof of \cite[Lem.~3]{BHN97}.

\begin{defi}\label{def:Gamma-prime-Gamma-double-prime}
Let \(\Gamma^{\prime}\) be the union of the supports \(|c|\), where \(c\) ranges over all simple simplicial cycles that lie in \(L^\prime\). Define
\[
E(\Gamma^{\prime})\coloneqq\{\,\tau\in E(\Gamma)\mid \tau\subseteq \Gamma^{\prime}\,\}.
\]
Define \(\Gamma^{\prime\prime}\) and \(E(\Gamma^{\prime\prime})\) in the same way, replacing \(L^{\prime}\) by \(L^{\prime\prime}\).
\end{defi}

\begin{lemma}\label{lem:Gamma-prime-Gamma-double-prime-disjoint}
One has \(E(\Gamma^{\prime})\cap E(\Gamma^{\prime\prime})=\varnothing\).
\end{lemma}

\begin{proof}
Assume \(E(\Gamma^{\prime})\cap E(\Gamma^{\prime\prime})\neq\varnothing\), and choose
\(\tau_0\in E(\Gamma^{\prime})\cap E(\Gamma^{\prime\prime})\). By the definitions of
\(E(\Gamma^{\prime})\) and \(E(\Gamma^{\prime\prime})\), choose simple simplicial cycles
\(c^{\prime}\in L^{\prime}\) and \(c^{\prime\prime}\in L^{\prime\prime}\) such that
\(\tau_0\subseteq |c^{\prime}|\cap |c^{\prime\prime}|\). Choose an oriented \(1\)-simplex \(s_0\) whose support is \(\tau_0\). We may assume that
\begin{equation}\label{eq:c-prime-c-double-prime-s0}
c^{\prime}(s_0)=c^{\prime\prime}(s_0)=1.
\end{equation}
Indeed, if \(c^{\prime}(s_0)<0\), replace \(c^{\prime}\) by
\(-c^{\prime}\), and if \(c^{\prime\prime}(s_0)<0\), replace
\(c^{\prime\prime}\) by \(-c^{\prime\prime}\). These replacements do not
change the fact that \(c^{\prime}\) and \(c^{\prime\prime}\) are simple simplicial cycles.
We then set
\begin{equation}\label{eq:z-c-prime-c-double-prime}
z\coloneqq c^{\prime}+c^{\prime\prime}\in Z_1(\Gamma,\Z).
\end{equation}
Choose \(S\) so that
\(z=\sum_{s\in S}z(s)s\in Z_1(\Gamma,\Z)_{\ge 0}^{S}\), with \(s_0\in S\). By
\eqref{eq:c-prime-c-double-prime-s0}, one has
\[
z(s_0)=2.
\]
By Lemma~\ref{lem:positive-cycle-decomposition}, there exist simple simplicial cycles
\(d_1,\ldots,d_N\in Z_1(\Gamma,\Z)_{\ge 0}^{S}\) such that
\[
z=d_1+\cdots+d_N.
\]
Moreover, there exists \(i_0\) such that \(d_{i_0}(s_0)>0\). Set
\(d\coloneqq d_{i_0}\). Then
\[
z-d\in Z_1(\Gamma,\Z)_{\ge 0}^{S}.
\]
By Lemma~\ref{lem:harmonic-generated-by-simple-cycles-and-one-summand}\textup{(ii)},
one has \(d\in L^{\prime}\cup L^{\prime\prime}\). By symmetry between
\((L^{\prime},c^{\prime})\) and \((L^{\prime\prime},c^{\prime\prime})\), we may
assume that \(d\in L^{\prime}\). Since \(c^{\prime\prime}\in L^{\prime\prime}\)
and \(Q_{\Theta_{\Gamma}}(L^{\prime},L^{\prime\prime})=0\), one has
\begin{equation}\label{eq:Q-d-c-double-prime-zero}
Q_{\Theta_{\Gamma}}(d,c^{\prime\prime})=0.
\end{equation}
We show that this is impossible. For each edge \(\tau\subseteq |d|\), let
\(s_{\tau}\in S\) be the oriented \(1\)-simplex on \(\tau\) such that
\(d(s_{\tau})=1\). Let \(\tau\subseteq |d|\cap |c^{\prime\prime}|\). Then
\(c^{\prime\prime}(s_{\tau})\in\{-1,1\}\). Since
\(z-d\in Z_1(\Gamma,\Z)_{\ge 0}^{S}\), the equality \(d(s_{\tau})=1\) gives
\(
z(s_{\tau})>0.
\)
By \eqref{eq:z-c-prime-c-double-prime}, one has
\[
c^{\prime}(s_{\tau})+c^{\prime\prime}(s_{\tau})>0.
\]
Since \(c^{\prime}(s_{\tau})\in \Z_{\le1}\), this forces
\(
c^{\prime\prime}(s_{\tau})=1.
\)
By {\eqref{eq:Q-Theta-Gamma-edge-expansion}},
\[
Q_{\Theta_{\Gamma}}(d,c^{\prime\prime})
=
\sum_{\tau\subseteq |d|}d(s_{\tau})c^{\prime\prime}(s_{\tau})\ell(\tau)
=
\sum_{\tau\subseteq |d|\cap |c^{\prime\prime}|}\ell(\tau)
\ge
\ell(\tau_0)>0.
\]
This contradicts \eqref{eq:Q-d-c-double-prime-zero}. Hence the lemma follows.
\end{proof}

\begin{prop}\label{prop:vanishing-on-other-Gamma}
For any \(\tau\in E(\Gamma^{\prime\prime})\) and any endpoint \(v\) of \(\tau\), one has
\begin{equation}\label{eq:vanishing-on-other-Gamma}
\int_{c_{\tau,v}}\kappa_{\Theta_{\Gamma}}(\alpha)=0
\quad\text{for all }\alpha\in L^{\prime}.
\end{equation}
\end{prop}

\begin{proof}
Let \(\alpha\in L^{\prime}\). {If \(\alpha=0\), \eqref{eq:vanishing-on-other-Gamma} holds. Assume that \(\alpha\ne0\).}
By Lemma~\ref{lem:harmonic-generated-by-simple-cycles-and-one-summand}\textup{(i)}, write \(\alpha=c_1+\cdots+c_N\) with simple simplicial cycles \(c_j\).
By assertion~\textup{(ii)} of the same lemma, every \(c_j\) lies in \(L^{\prime}\) or \(L^{\prime\prime}\).
Since \(H_1(\Gamma,\Z)=L^{\prime}\oplus L^{\prime\prime}\) and \(\alpha\in L^{\prime}\), taking the \(L^{\prime}\)-component gives
\[
\alpha=\sum_{c_j\in L^{\prime}}c_j.
\]
Thus \(L^{\prime}\) is generated by the simple simplicial cycles contained in \(L^{\prime}\). It is enough to prove \eqref{eq:vanishing-on-other-Gamma} for \(\alpha=c\), where \(c\in L^{\prime}\) is a simple simplicial cycle. Fix such a cycle \(c\). Fix \(\tau\in E(\Gamma^{\prime\prime})\) and an endpoint \(v\) of \(\tau\).

Choose an endpoint \(v(\tau')\) of each edge \(\tau' \in E(\Gamma)\), with \(v(\tau)=v\) for the particular edge \(\tau\). Write
\[
c=\sum_{\tau'\in E(\Gamma)}a_{\tau'}c_{\tau',v(\tau')},
\qquad
a_{\tau'}\in\{-1,0,1\}.
\]
Then \(a_{\tau'}\ne 0\) if and only if \(\tau'\subseteq |c|\). Using \eqref{eq:cycle-edge-integral} with this choice of endpoints,
\[
\int_{c_{\tau,v}}\kappa_{\Theta_{\Gamma}}(c)
=
\ell(\tau)a_{\tau}.
\]
If this integral is nonzero, then \(a_{\tau}\ne 0\). Hence \(\tau\subseteq |c|\), and so \(\tau\in E(\Gamma^{\prime})\), contradicting Lemma~\ref{lem:Gamma-prime-Gamma-double-prime-disjoint}. Therefore \eqref{eq:vanishing-on-other-Gamma} holds for \(\alpha=c\), and hence for all \(\alpha\in L^{\prime}\).
\end{proof}

\section{Tropical Matsusaka criterion}\label{sec:tropical-matsusaka}
In this section, we prove the tropical Matsusaka criterion, Theorem~\ref{thm:tropical-matsusaka-criterion}. Therefore we answer the question posed in \cite[Sec.~7]{BMV11}.

\subsection{Cycle classes and spanning curves}\label{subsec:cycle-classes-spanning-curves}

Let \((X,\Theta)\) be a \(g\)-dimensional principally polarized tropical abelian variety, where we write
\(
X=(\Lambda_X,\Lambda_X^{\prime},[\,\cdot\,,\,\cdot\,]_X).
\)
Since \(\Theta\) is a principal polarization, the associated lattice homomorphism
\(\kappa_{\Theta}\colon \Lambda_X^{\prime}\to \Lambda_X\) is an isomorphism.
Choose a \(\Z\)-basis \(\{\lambda_{X,i}^{\prime}\}_{i=1}^{g}\) of
\(\Lambda_X^{\prime}\), and set
\(\eta_{X,i}\coloneqq\kappa_{\Theta}(\lambda_{X,i}^{\prime})\).
Then \(\{\eta_{X,i}\}_{i=1}^{g}\) is a \(\Z\)-basis of \(\Lambda_X\).
Let \(\{(\lambda_{X,i}^{\prime})^{*}\}_{i=1}^{g}\) and
\(\{\eta_{X,i}^{*}\}_{i=1}^{g}\) be the dual bases.

\begin{lemma}\label{lem:GS25-normal-form}
Follow the notation above. Then the following hold.
\begin{enumerate}[label=\textup{(\roman*)}]
\item One has
\(
c_1(\MO_X(\Theta))
=
\sum_{i=1}^{g}(\lambda_{X,i}^{\prime})^{*}\otimes \eta_{X,i}.
\)
\item One also has
\(
\cyc_X\bigl(([\Theta])^{g-1}\bigr)
=
(g-1)!\sum_{i=1}^{g}\lambda_{X,i}^{\prime}\otimes \eta_{X,i}^{*}.
\)
\end{enumerate}
\end{lemma}

\begin{proof}
Comparing \eqref{eq:kappaD-matrix} with
\[
\eta_{X,i}
=
\kappa_{\Theta}(\lambda_{X,i}^{\prime}),
\qquad
1\le i\le g,
\]
one obtains \(E^{\Theta}=I_g\), where \(I_g\) denotes the \(g\times g\) identity matrix. Hence \eqref{eq:chern-class-matrix} gives assertion \textup{(i)}. We prove \textup{(ii)}. By \eqref{eq:fundamental-cycle-class} and Lemma~\ref{lem:top-self-intersection-det}\textup{(ii)}, one has
\[
\cyc_X([X])
=
(\lambda_{X,1}^{\prime}\wedge\cdots\wedge \lambda_{X,g}^{\prime})
\otimes
(\eta_{X,1}^{*}\wedge\cdots\wedge \eta_{X,g}^{*}),
\]
for the chosen \(\Z\)-bases. By \eqref{eq:cup-product-integral-tori}, one has
\[
c_1(\MO_X(\Theta))^{g-1}
=
(g-1)!\sum_{i=1}^{g}
(\lambda_{X,1}^{\prime})^{*}\wedge\cdots\wedge\widehat{(\lambda_{X,i}^{\prime})^{*}}\wedge\cdots\wedge(\lambda_{X,g}^{\prime})^{*}
\otimes
\eta_{X,1}\wedge\cdots\wedge\widehat{\eta_{X,i}}\wedge\cdots\wedge \eta_{X,g},
\]
where the hat means that the indicated factor is omitted from the wedge product. Hence, by \eqref{eq:cap-product-integral-tori},
\begin{align*}
\cyc_X\bigl(([\Theta])^{g-1}\bigr)
&=
c_1(\MO_X(\Theta))^{g-1}\cap \cyc_X([X]) \\
&=
(g-1)!\sum_{i=1}^{g}
\Bigl(
((\lambda_{X,1}^{\prime})^{*}\wedge\cdots\wedge\widehat{(\lambda_{X,i}^{\prime})^{*}}\wedge\cdots\wedge(\lambda_{X,g}^{\prime})^{*})
\mathbin{\lrcorner}
(\lambda_{X,1}^{\prime}\wedge\cdots\wedge \lambda_{X,g}^{\prime})
\Bigr) \\
&\qquad\qquad\qquad\qquad\otimes
\Bigl(
(\eta_{X,1}\wedge\cdots\wedge\widehat{\eta_{X,i}}\wedge\cdots\wedge \eta_{X,g})
\mathbin{\lrcorner}
(\eta_{X,1}^{*}\wedge\cdots\wedge \eta_{X,g}^{*})
\Bigr) \\
&=
(g-1)!\sum_{i=1}^{g}\lambda_{X,i}^{\prime}\otimes \eta_{X,i}^{*}. \qedhere
\end{align*}
\end{proof}

\begin{lemma}\label{lem:tensor-cycle-decomposition}
Let \(A\) and \(B\) be free abelian groups, let \(M\) be a free abelian group, and let \(\partial\colon A\to B\) be a homomorphism. Let \(c_1,\ldots,c_r\in A\) and \(u_1,\ldots,u_r\in M\). If
\[
\sum_{i=1}^{r}\partial(c_i)\otimes u_i=0
\]
in \(B\otimes_{\Z}M\), then there exist \(d_1,\ldots,d_r\in\ker(\partial)\) such that
\[
\sum_{i=1}^{r}c_i\otimes u_i
=
\sum_{i=1}^{r}d_i\otimes u_i
\]
in \(A\otimes_{\Z}M\).
\end{lemma}

\begin{proof}
Let \(M_0\subseteq M\) be the subgroup generated by \(u_1,\ldots,u_r\), and fix a \(\Z\)-basis \(\{m_j\}_{j=1}^{s}\) of \(M_0\). Write
\[
u_i=\sum_{j=1}^{s}a_{j,i}m_j,
\qquad
a_{j,i}\in\Z.
\]
Since \(B\) is free and \(m_1,\ldots,m_s\) are \(\Z\)-linearly independent in \(M\), the equality
\[
0
=
\sum_{i=1}^{r}\partial(c_i)\otimes u_i
=
\sum_{j=1}^{s}
\partial\left(
\sum_{i=1}^{r}a_{j,i}c_i
\right)
\otimes m_j
\]
implies
\[
\partial\left(
\sum_{i=1}^{r}a_{j,i}c_i
\right)
=
0
\]
for every \(1\le j\le s\).
Hence
\(
\sum_{i=1}^{r}a_{j,i}c_i\in\ker(\partial)
\)
for every \(1\le j\le s\). Since \(u_1,\ldots,u_r\) generate \(M_0\), write
\[
m_j=\sum_{i=1}^{r}b_{i,j}u_i,
\qquad
b_{i,j}\in\Z.
\]
For \(1\le i\le r\), set
\[
d_i
=
\sum_{j=1}^{s}
b_{i,j}
\left(
\sum_{k=1}^{r}a_{j,k}c_k
\right)
\in\ker(\partial).
\]
Then
\[
\sum_{i=1}^{r}c_i\otimes u_i
=
\sum_{j=1}^{s}
\left(
\sum_{i=1}^{r}a_{j,i}c_i
\right)
\otimes m_j
=
\sum_{i=1}^{r}d_i\otimes u_i.\quad\qedhere
\]
\end{proof}

\begin{lemma}\label{lem:C-must-be-spanning-in-proof}
Let \(C\) be an effective tropical \(1\)-cycle on \(X\) such that \(|C|\) is connected. If
\((g-1)!C\equiv_{\hom}([\Theta])^{g-1}\),
then \(C\) is a spanning curve on \(X\).
\end{lemma}

\begin{proof}
By Lemma~\ref{lem:GS25-normal-form}\textup{(ii)}, one has
\begin{equation}\label{eq:theta-power-normal-form}
\cyc_X\bigl(([\Theta])^{g-1}\bigr)
=
(g-1)!\sum_{i=1}^{g}\lambda_{X,i}^{\prime}\otimes\eta_{X,i}^{*}.
\end{equation}
Follow the convention in Definition~\ref{def:translation-chart} and give a suitable global face structure on \(|C|\). By \cite[Sec.~3E]{GS23}, the cycle class \(\cyc_X(C)\) lies in
\(H_{1,1}(X)\cong \Lambda_X^{\prime}\otimes_{\Z}(\Lambda_X)^*\). Since
\(\{\lambda_{X,i}^{\prime}\otimes\eta_{X,j}^{*}\}_{1\le i,j\le g}\)
is a \(\Z\)-basis of \(\Lambda_X^{\prime}\otimes_{\Z}(\Lambda_X)^*\), there are uniquely determined integers
\(N^C_{ij}\in\Z\), \(1\le i,j\le g\), such that
\begin{equation}\label{eq:cycle-class-C-matrix}
\cyc_X(C)
=
\sum_{i=1}^{g}\sum_{j=1}^{g}
N^C_{ij}\lambda_{X,i}^{\prime}\otimes\eta_{X,j}^{*}.
\end{equation}
Set \(N^C=(N^C_{ij})_{1\le i,j\le g}\in M_g(\Z)\).

Since \((g-1)!C\equiv_{\hom}([\Theta])^{g-1}\), \eqref{eq:theta-power-normal-form} and~\eqref{eq:cycle-class-C-matrix} give
{
\[
N^C=I_g.
\]
}

Suppose to the contrary that \(C\) is not a spanning curve. We show that
\(\rank_{\R}N^C<g\).  For every edge
\(\tau\) of \(|C|\), fix a vertex \(v(\tau)\) of \(\tau\). Let
\(\boldsymbol{u}_{\tau,v(\tau)}\in(\Lambda_X)^*\) be the tangent direction of
\(\tau\) out of \(v(\tau)\), and let \(\omega_C(\tau)\in\Z_{>0}\) be the weight
of \(\tau\). By Definition~\ref{def:spanning-curve}, there exists a nonzero element \(v_0\in(\Lambda_X)_{\R}\) such that
\[
\boldsymbol{u}_{\tau,v(\tau)}(v_0)=0,
\]
for every edge \(\tau\) of \(|C|\).
For each edge \(\tau\), let \(c_{\tau,v(\tau)}\) be the oriented \(1\)-simplex directed out of \(v(\tau)\), as in Definition~\ref{def:tropical-curve}, and set
\[
z_C
\coloneqq
\sum_{\tau}
c_{\tau,v(\tau)}
\otimes
\omega_C(\tau)\boldsymbol{u}_{\tau,v(\tau)}
\in
C_1(|C|,\Z)\otimes_{\Z}(\Lambda_X)^*.
\]
By \cite[Sec.~2]{IKMZ19} and \cite[Sec.~6]{GS23}, in this description the boundary map on tropical \((1,1)\)-chains is
\[
\partial\otimes\id_{(\Lambda_X)^*}
\colon
C_1(|C|,\Z)\otimes_{\Z}(\Lambda_X)^*
\longrightarrow
C_0(|C|,\Z)\otimes_{\Z}(\Lambda_X)^*,
\]
and the balancing condition at all vertices of $|C|$ gives
\[
(\partial\otimes\id_{(\Lambda_X)^*})(z_C)=0.
\]
{Let \(j_C\colon |C|\hookrightarrow X\) be the inclusion, and let \((j_C)_*\colon H_1(|C|,\Z)\to H_1(X,\Z)\) be the induced homomorphism.}
By the construction of the tropical cycle class map in \cite[Sec.~3E]{GS23}, \(\cyc_X(C)\) is the image of the homology class of \(z_C\). {Apply} Lemma~\ref{lem:tensor-cycle-decomposition} with
\[
A=C_1(|C|,\Z),
\qquad
B=C_0(|C|,\Z),
\qquad
M=(\Lambda_X)^*.
\]
{There} {exist elements} \(d_\tau\in\ker(\partial)\){, one} for every edge \(\tau\) {of \(|C|\),} such that
\[
z_C
=
\sum_{\tau}
d_\tau\otimes
\omega_C(\tau)\boldsymbol{u}_{\tau,v(\tau)}.
\]
Since \(|C|\) is one-dimensional, \(\ker(\partial)=H_1(|C|,\Z)\). {Set \(\gamma_\tau\coloneqq(j_C)_*(d_\tau)\in H_1(X,\Z)=\Lambda_X^{\prime}\).} Then
\[
\cyc_X(C)
=
\sum_{\tau}\omega_C(\tau)\gamma_{\tau}\otimes\boldsymbol{u}_{\tau,v(\tau)}.
\]
Then \(\cyc_X(C)\) defines a $\R$-linear map

\[
\Phi_C\colon(\Lambda_X)_{\R}\to(\Lambda_X^{\prime})_{\R},
\qquad
x\longmapsto
\sum_{\tau}\omega_C(\tau)\boldsymbol{u}_{\tau,v(\tau)}(x)\gamma_{\tau}.
\]

Since \(\boldsymbol{u}_{\tau,v(\tau)}(v_0)=0\) for every \(\tau\), one has \(\Phi_C(v_0)=0\). Hence \(\Phi_C\) is not injective. On the other hand, \eqref{eq:cycle-class-C-matrix} gives

\[
\Phi_C(\eta_{X,j})
=
\sum_{i=1}^{g}N^C_{ij}\lambda_{X,i}^{\prime},
\qquad
1\le j\le g.
\]

Therefore, with respect to the bases \(\{\eta_{X,j}\}_{j=1}^{g}\) and \(\{\lambda_{X,i}^{\prime}\}_{i=1}^{g}\), the matrix of \(\Phi_C\) is \(N^C\). Since \(\Phi_C\) is not injective, one has \(\rank_{\R}N^C<g\). This contradicts \(N^C=I_g\). Therefore \(C\) is a spanning curve on \(X\).
\end{proof}

\subsection{Splitting maps}\label{subsec:tropical-splitting-map}

We follow the notation of Subsection~\ref{subsec:cycle-classes-spanning-curves}.
Let \(\Gamma\) be a smooth tropical curve of genus {\(r\ge1\)}. Set
\(J\coloneqq\Jac(\Gamma)\), and let \(\Theta_{\Gamma}\) be the canonical principal polarization on \(J\).
Fix \(q\in\Gamma\), let \(\phi_q\colon\Gamma\to J\) be the tropical Abel--Jacobi map with basepoint \(q\), and let
\(\chi\colon\Gamma\to X\) be a morphism of rational polyhedral spaces.
{
Define \(t_{-\chi(q)}\colon X\to X\) by \(t_{-\chi(q)}(x)=x-\chi(q)\) for \(x\in X\).
Then \((t_{-\chi(q)}\circ\chi)(q)=0\).
}
Let \(f=(f^\sharp,f_\sharp)\colon J\to X\) be the homomorphism induced by \(\chi\) as in Proposition~\ref{prop:universal-property}, so that the following diagram commutes:
\begin{equation}\label{eq:abel-jacobi-factorization}
\begin{tikzcd}[ampersand replacement=\&, column sep=large]
\Gamma
\arrow[r,"\phi_q"]
\arrow[dr,"t_{-\chi(q)}\circ\chi"']
\&
J
\arrow[d,"f"]
\\
\&
X.
\end{tikzcd}
\end{equation}

Choose a \(\Z\)-basis \(\{\alpha_k\}_{k=1}^{r}\) of \(H_1(\Gamma,\Z)\), and set
\(\omega_k\coloneqq\kappa_{\Theta_{\Gamma}}(\alpha_k)\) for \(1\le k\le r\).
Since \(\Theta_{\Gamma}\) is a principal polarization,
\(\{\omega_k\}_{k=1}^{r}\) is a \(\Z\)-basis of \(H^0(J,\Omega_J)\).
Let \(\{\omega_k^{*}\}_{k=1}^{r}\) be the dual basis of
\(\{\omega_k\}_{k=1}^{r}\).
Then there exist unique integers \(a_{i,k}\in\Z\) and \(b_{k,j}\in\Z\) such that
\begin{equation}\label{eq:fsharp-fsharp-expansions}
f_\sharp(\alpha_k)=\sum_{i=1}^{g}a_{i,k}\lambda_{X,i}^{\prime},
\qquad 1\le k\le r,
\qquad
f^\sharp(\eta_{X,j})=\sum_{k=1}^{r}b_{k,j}\omega_k,
\qquad 1\le j\le g.
\end{equation}

Set
\begin{equation}\label{eq:def-i}
i\coloneqq\varphi_{\Theta_{\Gamma}}^{-1}\circ f^\vee\circ\varphi_{\Theta}\colon X\to J.
\end{equation}

\medskip
We recall the following special case of the tropical Poincar\'e formula
\cite[Thm.~A]{GS23}.

\begin{thm}\label{Thm:Poincare}
Set \([\widetilde W_1]\coloneqq(\phi_q)_*[\Gamma]\). Then
\[
(r-1)![\widetilde W_1]\equiv_{\hom}([\Theta_\Gamma])^{r-1}.
\]
\end{thm}

\begin{lemma}\label{lem:W1-pushforward-chi}
One has \(f_*[\widetilde W_1]\equiv_{\hom}\chi_*[\Gamma]\).
\end{lemma}

\begin{proof}
By \eqref{eq:abel-jacobi-factorization}, one has
\(f_*[\widetilde W_1]=(t_{-\chi(q)})_*\chi_*[\Gamma]\).
Since \(t_{-\chi(q)}\) induces the identity on \(H_{1,1}(X)\), the assertion follows.
\end{proof}

By Theorem~\ref{Thm:Poincare} and Lemma~\ref{lem:GS25-normal-form}\textup{(ii)}, one has
\begin{equation}\label{eq:cycle-W1-normal-form}
\cyc_J([\widetilde W_1])
=
\sum_{k=1}^{r}\alpha_k\otimes\omega_k^{*}.
\end{equation}
By Proposition~\ref{prop:GS19-toolkit}\textup{(iii)} and Lemma~\ref{lem:push-PPTAV-en}, one obtains
\begin{equation}\label{eq:pushforward-W1-cycle-class}
\cyc_X\bigl(f_*[\widetilde W_1]\bigr)
=
\sum_{i,j=1}^{g}
\Bigl(\sum_{k=1}^{r}a_{i,k}b_{k,j}\Bigr)
\lambda_{X,i}^{\prime}\otimes\eta_{X,j}^{*}.
\end{equation}

\begin{lemma}\label{lem:splitting-map-identity}
If
\(
(g-1)!\chi_*[\Gamma]\equiv_{\hom}([\Theta])^{g-1},
\)
then \(f\circ i=\id_X\).
\end{lemma}

\begin{proof}
By Lemma~\ref{lem:W1-pushforward-chi}, one has
\((g-1)!f_*[\widetilde W_1]\equiv_{\hom}([\Theta])^{g-1}\).
By \eqref{eq:pushforward-W1-cycle-class}, \eqref{eq:theta-power-normal-form} and Definition~\ref{def:homological-equivalence}, the preceding homological equivalence gives
\begin{equation}\label{eq:splitting-coeff-condition}
\sum_{k=1}^{r}a_{i,k}b_{k,j}=\delta_{i,j},
\qquad 1\le i,j\le g,
\end{equation}
where \(\delta_{i,j}\) is the Kronecker delta.

By \eqref{eq:def-i}, the definition of \(f^\vee\), and
\eqref{eq:phi-Theta-sharp-kappa}, one has
\[
i^\sharp
=
\varphi_\Theta^\sharp\circ(f^\vee)^\sharp
\circ(\varphi_{\Theta_\Gamma}^{-1})^\sharp
=
\kappa_\Theta\circ f_\sharp\circ\kappa_{\Theta_\Gamma}^{-1}.
\]
Thus, {by \eqref{eq:fsharp-fsharp-expansions}},
{
\[
i^\sharp(\omega_k)=\kappa_\Theta\bigl(f_\sharp(\alpha_k)\bigr)=\sum_{i=1}^{g}a_{i,k}\eta_{X,i},
\qquad 1\le k\le r.
\]
}
For \(1\le j\le g\), by \eqref{eq:fsharp-fsharp-expansions} and \eqref{eq:splitting-coeff-condition},
\[
(f\circ i)^\sharp(\eta_{X,j})
=
i^\sharp\Bigl(\sum_{k=1}^{r}b_{k,j}\omega_k\Bigr)
=
\sum_{i=1}^{g}\Bigl(\sum_{k=1}^{r}a_{i,k}b_{k,j}\Bigr)\eta_{X,i}
=
\eta_{X,j}.
\]
Hence \((f\circ i)^\sharp=\id_{\Lambda_X}\).
{
For \(\lambda\in\Lambda_X\) and \(\lambda'\in\Lambda_X^{\prime}\), \eqref{eq:dual-hom-pairing-compat} gives
\[
[\lambda,(f\circ i)_\sharp(\lambda')-\lambda']_X=[(f\circ i)^\sharp(\lambda),\lambda']_X-[\lambda,\lambda']_X=0.
\]
By the nondegeneracy of \([\,\cdot\,,\,\cdot\,]_X\), \((f\circ i)_\sharp=\id_{\Lambda_X^{\prime}}\).
}
Therefore \(f\circ i=\id_X\).
\end{proof}

\subsection{Pullback of polarizations}\label{subsec:pullback-polarization}

We follow the notation of Subsections~\ref{subsec:cycle-classes-spanning-curves}
and~\ref{subsec:tropical-splitting-map}.

\begin{defi}\label{def:dual-polarization}
Let \((Y,\Theta_Y)\) be a principally polarized tropical abelian variety, and let
\(\varphi_{\Theta_Y}\colon Y\to Y^\vee\) be the isomorphism defined in
Definition~\ref{def:phi-Theta}. Set
\(
\Theta_Y^\vee\coloneqq(\varphi_{\Theta_Y}^{-1})^{*}\Theta_Y.
\)
We call \(\Theta_Y^\vee\) the dual principal polarization on \(Y^\vee\).
\end{defi}

\begin{lemma}\label{lem:dual-principal-polarization-kappa}
Follow the notation in Definition~\ref{def:dual-polarization}. Recall from Proposition~\ref{prop:phi-Theta-dual-kappa} the identity
\(
\varphi_{\Theta_Y}^{\sharp}=(\varphi_{\Theta_Y})_{\sharp}=-\kappa_{\Theta_Y}.
\)
Then
\(
\kappa_{\Theta_Y^\vee}=\kappa_{\Theta_Y}^{-1}\colon
\Lambda_Y\to \Lambda_Y^{\prime}.
\)
\end{lemma}

\begin{proof}
Since \(\Theta_Y^\vee=(\varphi_{\Theta_Y}^{-1})^{*}\Theta_Y\), by
\eqref{eq:kappa-pullback-divisor} and Proposition~\ref{prop:phi-Theta-dual-kappa},
\[
\kappa_{\Theta_Y^\vee}
=
(\varphi_{\Theta_Y}^{-1})^{\sharp}\circ
\kappa_{\Theta_Y}\circ
(\varphi_{\Theta_Y}^{-1})_{\sharp}
=
(-\kappa_{\Theta_Y}^{-1})\circ
\kappa_{\Theta_Y}\circ
(-\kappa_{\Theta_Y}^{-1})
=
\kappa_{\Theta_Y}^{-1}. \qedhere
\]
\end{proof}

\begin{lemma}\label{lem:welters-cycle-polarization}
For \(p\in \Z_{>0}\),
\begin{equation*}
(g-1)!\chi_*[\Gamma]\equiv_{\hom}p\,([\Theta])^{g-1}
\quad\text{if and only if}\quad
c_1(i^{*}\MO_J(\Theta_{\Gamma}))=p\,c_1(\MO_X(\Theta)).
\end{equation*}
\end{lemma}

\begin{proof}
\smallskip
\noindent\textbf{Step \textup{(i)}.}
By Lemma~\ref{lem:W1-pushforward-chi} and
\eqref{eq:pushforward-W1-cycle-class}, one has
\[
\cyc_X(\chi_*[\Gamma])
=
\sum_{i,j=1}^{g}
\Bigl(\sum_{k=1}^{r}a_{i,k}b_{k,j}\Bigr)
\lambda_{X,i}^{\prime}\otimes\eta_{X,j}^{*}.
\]
On the other hand, by Lemma~\ref{lem:GS25-normal-form}\textup{(ii)}, one has
\[
\cyc_X\bigl(p([\Theta])^{g-1}\bigr)
=
(g-1)!\sum_{i,j=1}^{g}
p\,\delta_{i,j}\lambda_{X,i}^{\prime}\otimes\eta_{X,j}^{*},
\]
where \(\delta_{i,j}\) is the Kronecker delta. Hence, by
Definition~\ref{def:homological-equivalence},
\[
(g-1)!\chi_*[\Gamma]\equiv_{\hom}p([\Theta])^{g-1}
\]
if and only if
\begin{equation}\label{eq:welters-coeff-condition}
\sum_{k=1}^{r}a_{i,k}b_{k,j}=p\,\delta_{i,j},
\qquad 1\le i,j\le g.
\end{equation}

\smallskip
\noindent\textbf{Step \textup{(ii)}.}
We next prove that \eqref{eq:welters-coeff-condition} is equivalent to
\begin{equation}\label{eq:c1-dual-p}
c_1\bigl((f^\vee)^{*}\MO_{J^\vee}(\Theta_{\Gamma}^{\vee})\bigr)
=
p\,c_1(\MO_{X^\vee}(\Theta^\vee)).
\end{equation}
By Lemma~\ref{lem:dual-principal-polarization-kappa}, one has
\(\omega_k=\kappa_{\Theta_{\Gamma}}(\alpha_k)\) and hence
\(\kappa_{\Theta_{\Gamma}^{\vee}}(\omega_k)
=\kappa_{\Theta_{\Gamma}}^{-1}(\omega_k)=\alpha_k\)
for \(1\le k\le r\). Since
\(
f^\vee=(f_{\sharp},f^{\sharp}),
\)
by \eqref{eq:kappa-pullback-divisor}, one obtains
\[
\kappa_{(f^\vee)^{*}\Theta_{\Gamma}^{\vee}}
=
(f^\vee)^{\sharp}\circ \kappa_{\Theta_{\Gamma}^{\vee}}\circ (f^\vee)_{\sharp}
=
f_{\sharp}\circ \kappa_{\Theta_{\Gamma}^{\vee}}\circ f^{\sharp}.
\]
Using \eqref{eq:fsharp-fsharp-expansions}, for \(1\le j\le g\), one has
\[
\kappa_{(f^\vee)^{*}\Theta_{\Gamma}^{\vee}}(\eta_{X,j})
=
\sum_{i=1}^{g}\Bigl(\sum_{k=1}^{r}a_{i,k}b_{k,j}\Bigr)\lambda_{X,i}^{\prime}.
\]
Then by \eqref{eq:chern-class-matrix} and \eqref{eq:kappaD-matrix}, one has
\begin{equation*}
c_1\bigl((f^\vee)^{*}\MO_{J^\vee}(\Theta_{\Gamma}^{\vee})\bigr)
=
\sum_{i,j=1}^{g}\Bigl(\sum_{k=1}^{r}a_{i,k}b_{k,j}\Bigr)(\eta_{X,j})^{*}\otimes \lambda_{X,i}^{\prime}.
\end{equation*}
Since \(\eta_{X,j}=\kappa_{\Theta}(\lambda_{X,j}^{\prime})\) for
\(1\le j\le g\), by Lemma~\ref{lem:dual-principal-polarization-kappa}, one has
\(\kappa_{\Theta^\vee}(\eta_{X,j})
=\kappa_{\Theta}^{-1}(\eta_{X,j})
=\lambda_{X,j}^{\prime}\)
for \(1\le j\le g\). By \eqref{eq:chern-class-matrix} and \eqref{eq:kappaD-matrix}, one has
\begin{equation*}
c_1(\MO_{X^\vee}(\Theta^\vee))
=
\sum_{i,j=1}^{g}\delta_{i,j}\,(\eta_{X,j})^{*}\otimes \lambda_{X,i}^{\prime}.
\end{equation*}
Therefore \eqref{eq:c1-dual-p} holds if and only if
\eqref{eq:welters-coeff-condition} holds.

\smallskip
\noindent\textbf{Step \textup{(iii)}.}
Combining the two equivalences above, we obtain that
\begin{equation}\label{eq:welters-hom-condition}
(g-1)!\chi_*[\Gamma]\equiv_{\hom}p\,([\Theta])^{g-1}
\end{equation}
holds if and only if \eqref{eq:c1-dual-p} holds. Since
\[
(\varphi_{\Theta_{\Gamma}}^{-1})^{*}\MO_J(\Theta_{\Gamma})
=
\MO_{J^\vee}(\Theta_{\Gamma}^\vee),
\]
by \eqref{eq:def-i}, one has
\begin{equation}\label{eq:i-pullback-Theta-Gamma}
i^{*}\MO_J(\Theta_{\Gamma})
=
\varphi_{\Theta}^{*}(f^\vee)^{*}\MO_{J^\vee}(\Theta_{\Gamma}^\vee).
\end{equation}
Moreover,
\begin{equation}\label{eq:phi-Theta-pullback-Theta-dual}
\varphi_{\Theta}^{*}\MO_{X^\vee}(\Theta^\vee)=\MO_X(\Theta).
\end{equation}

Assume first that \eqref{eq:c1-dual-p} holds. Then
\[
\varphi_{\Theta}^{*}c_1\bigl((f^\vee)^{*}\MO_{J^\vee}(\Theta_{\Gamma}^{\vee})\bigr)
=
p\,\varphi_{\Theta}^{*}c_1(\MO_{X^\vee}(\Theta^\vee)).
\]
By applying \(c_1\) to \eqref{eq:i-pullback-Theta-Gamma} and
\eqref{eq:phi-Theta-pullback-Theta-dual} {and using the equality above}, one obtains
\begin{equation}\label{eq:welters-polarization-condition}
c_1(i^{*}\MO_J(\Theta_{\Gamma}))=p\,c_1(\MO_X(\Theta)).
\end{equation}

Conversely, assume that \eqref{eq:welters-polarization-condition} holds. By
applying \(c_1\) to \eqref{eq:i-pullback-Theta-Gamma} and
\eqref{eq:phi-Theta-pullback-Theta-dual}, one has
\[
\varphi_{\Theta}^{*}c_1\bigl((f^\vee)^{*}\MO_{J^\vee}(\Theta_{\Gamma}^{\vee})\bigr)
=
p\,\varphi_{\Theta}^{*}c_1(\MO_{X^\vee}(\Theta^\vee)).
\]
Since \(\varphi_{\Theta}\) is an isomorphism,
\eqref{eq:c1-dual-p} holds. Hence \eqref{eq:c1-dual-p} is equivalent to
\eqref{eq:welters-polarization-condition}. Consequently,
\eqref{eq:welters-hom-condition} holds if and only if
\eqref{eq:welters-polarization-condition} holds.
\end{proof}

\subsection{Proof of the tropical Matsusaka criterion}\label{subsec:proof-tropical-matsusaka}
In this subsection, we give a complete proof of the tropical Matsusaka criterion.
\begin{thm}[Tropical Matsusaka criterion]\label{thm:tropical-matsusaka-criterion}
Let \((X,\Theta)\) be a \(g\)-dimensional principally polarized tropical abelian variety. Let \(C\) be an effective tropical \(1\)-cycle on \(X\) such that \(|C|\) is connected, and suppose that \((g-1)!C\equiv_{\hom}([\Theta])^{g-1}\). Then there exist a smooth tropical curve \(\Gamma\) of genus \(g\) and a morphism of rational polyhedral spaces \(\chi\colon \Gamma\to X\) such that \(\chi_*[\Gamma]=C\) and the homomorphism \(f\colon\Jac(\Gamma)\to X\) induced by \(\chi\) is an isomorphism {of principally polarized tropical abelian varieties}{:}
\begin{equation}\label{eq:matsusaka-ppav-isomorphism}
\bigl(\Jac(\Gamma),\Theta_{\Gamma}\bigr)\cong(X,\Theta),
\end{equation}
where \(\Theta_{\Gamma}\) denotes the canonical principal polarization on \(\Jac(\Gamma)\).
\end{thm}

\begin{proof}
We divide the proof into the following steps.

\smallskip
\noindent\textbf{Step \textup{(i)}.}
We compute the intersection number \((\Theta\cdot C)_X\). Since \(\Theta\) is a principal polarization, Lemma~\ref{lem:top-self-intersection-det}\textup{(iii)} gives \((\Theta^g)_X=g!\). Since \((g-1)!C\equiv_{\hom}([\Theta])^{g-1}\), one has
\begin{equation*}
(g-1)!(\Theta\cdot C)_X
=
(\Theta^g)_X
=
g!.
\end{equation*}
Hence \((\Theta\cdot C)_X=g\).

\smallskip
\noindent\textbf{Step \textup{(ii)}.}
We construct a smooth tropical curve mapping to \(X\).
By Lemma~\ref{lem:C-must-be-spanning-in-proof}, \(C\) is a spanning curve.
Hence Lemma~\ref{lem:cycle-from-smooth-curve} gives a smooth tropical curve \(\Gamma\) and a morphism of rational polyhedral spaces
\(\chi\colon\Gamma\to X\) such that \(\chi_*[\Gamma]=C\).
Set \(J\coloneqq\Jac(\Gamma)\), and let
\(f=(f^{\sharp},f_{\sharp})\colon J\to X\)
be the homomorphism appearing in Proposition~\ref{prop:universal-property}.

\smallskip
\noindent\textbf{Step \textup{(iii)}.}
Recall the homomorphism
\begin{equation}\label{eq:def-i-map-to-J1}
i\coloneqq
\varphi_{\Theta_{\Gamma}}^{-1}\circ f^\vee\circ\varphi_{\Theta}\colon X\to J,
\end{equation}
where \(\varphi_{\Theta}\colon X\to X^\vee\) and
\(\varphi_{\Theta_{\Gamma}}\colon J\to J^\vee\) are the homomorphisms associated to
the tropical line bundles \(\MO_X(\Theta)\) and \(\MO_J(\Theta_{\Gamma})\),
respectively, and \(f^\vee\colon X^\vee\to J^\vee\) is the dual homomorphism of
\(f\colon J\to X\).
By Lemma~\ref{lem:splitting-map-identity}, one has
\(
f\circ i=\id_X.
\)

\smallskip
\noindent\textbf{Step \textup{(iv)}.}
We verify the polarization condition. Since \(f\circ i=\id_X\), the homomorphism \(i\) is injective. Moreover, \((g-1)!\chi_*[\Gamma]=(g-1)!C\equiv_{\hom}([\Theta])^{g-1}\), so Lemma~\ref{lem:welters-cycle-polarization} gives
\(
c_1\bigl(i^*\MO_J(\Theta_{\Gamma})\bigr)
=
c_1\bigl(\MO_X(\Theta)\bigr).
\)
Hence \(i^{*}\Theta_{\Gamma}\) is a principal polarization on \(X\) and \begin{equation}\label{eq:pullback-polarization-ppav-isomorphism}
(X,\Theta)\cong(X,i^{*}\Theta_{\Gamma}).
\end{equation}

\smallskip
\noindent\textbf{Step \textup{(v)}.}
Applying Proposition~\ref{prop:split-off-subppav} to \(i\colon X\to J\)
and using \eqref{eq:pullback-polarization-ppav-isomorphism}, we obtain a principally
polarized tropical abelian variety \((Y,\Theta_Y)\) and an isomorphism of
principally polarized tropical abelian varieties
\begin{equation}\label{eq:ppav-splitting-beta}
\beta\colon (X,\Theta)\times (Y,\Theta_Y)
\cong (J,\Theta_{\Gamma}).
\end{equation}
{Moreover,}
\(\beta\circ\iota_1=i\). Hence
\(
\beta_{\sharp}(\Lambda_X^{\prime}\oplus 0)=\im(i_{\sharp})\subseteq H_1(\Gamma,\Z).
\)
Set
\[
L^{\prime}\coloneqq \beta_{\sharp}(\Lambda_X^{\prime}\oplus 0)
=
\im(i_{\sharp})
\subseteq H_1(\Gamma,\Z),\qquad
L^{\prime\prime}\coloneqq
\beta_{\sharp}(0\oplus \Lambda_Y^{\prime})
\subseteq H_1(\Gamma,\Z).
\]
Then
\(
H_1(\Gamma,\Z)=L^{\prime}\oplus L^{\prime\prime},
\)
and Lemma~\ref{lem:orthogonal-splitting-from-split-off} gives
\(
Q_{\Theta_{\Gamma}}(L^{\prime},L^{\prime\prime})=0.
\)

\smallskip
\noindent\textbf{Step \textup{(vi)}.}
We prove that \(Y\) is trivial. Suppose, to the contrary, that \(\dim Y>0\). Then \(L^{\prime\prime}\neq 0\). By Remark~\ref{rem:universal-property-sharp}, one has
\[
f^{\sharp}=\chi^{*}\colon H^0(X,\Omega_X)\longrightarrow H^0(\Gamma,\Omega_{\Gamma}).
\]
Moreover, \eqref{eq:def-i-map-to-J1} and Proposition~\ref{prop:phi-Theta-dual-kappa} give
\[
\kappa_{\Theta_{\Gamma}}\circ i_{\sharp}
=
f^{\sharp}\circ\kappa_{\Theta}.
\]
Since \(\kappa_{\Theta}\colon\Lambda_X^{\prime}\to\Lambda_X\) is an isomorphism and \(L^{\prime}=\im(i_{\sharp})\), it follows that
\begin{equation}\label{eq:image-chi-star-kappa-L-prime}
\im\bigl(H^0(X,\Omega_X)\xrightarrow{\ \chi^{*}\ }H^0(\Gamma,\Omega_{\Gamma})\bigr)
=
\im(f^{\sharp})
=
\kappa_{\Theta_{\Gamma}}(L^{\prime}).
\end{equation}
Let \(\Gamma^{\prime\prime}\) and \(E(\Gamma^{\prime\prime})\) be defined from \(L^{\prime\prime}\) as in Definition~\ref{def:Gamma-prime-Gamma-double-prime}. Lemma~\ref{lem:harmonic-generated-by-simple-cycles-and-one-summand}\textup{(i), (ii)} implies that \(L^{\prime\prime}\) is generated by the simple simplicial cycles contained in \(L^{\prime\prime}\). Hence \(E(\Gamma^{\prime\prime})\neq\varnothing\).

Choose \(\tau\in E(\Gamma^{\prime\prime})\) and an endpoint \(v\) of \(\tau\), and let \(c_{\tau,v}\in C_1(\Gamma,\Z)\) be the oriented \(1\)-simplex supported on \(\tau\) and directed out of \(v\). Let
\(
\widetilde{\boldsymbol{u}}_{\tau,v}\in T_v^{\Z}\Gamma
\)
denote the tangent functional of \(\tau\) at \(v\) defined in Definition~\ref{def:edge-tangent-length}\textup{(i)}. By Proposition~\ref{prop:vanishing-on-other-Gamma}, \eqref{eq:image-chi-star-kappa-L-prime}, and \eqref{eq:integral-along-edge}, for every \(\omega\in H^0(X,\Omega_X)\),
\[
0
=
\int_{c_{\tau,v}}\chi^{*}\omega
=
\ell(\tau)\,
\widetilde{\boldsymbol{u}}_{\tau,v}
\bigl((\chi^{*}\omega)_v\bigr).
\]
Since the length \(\ell(\tau)>0\), one has
\begin{equation}\label{eq:vanishing-tangent-functional-pullback}
\widetilde{\boldsymbol{u}}_{\tau,v}
\bigl((\chi^{*}\omega)_v\bigr)
=
0
\end{equation}
for every \(\omega\in H^0(X,\Omega_X)\).

We use the following claim, whose proof is given immediately after the proof of the theorem.

\begin{claim}\label{claim:nonvanishing-pullback-tangent}
There exists \(\omega_0\in H^0(X,\Omega_X)\) such that \(\widetilde{\boldsymbol{u}}_{\tau,v}\bigl((\chi^{*}\omega_0)_v\bigr)\neq 0\).
\end{claim}

Claim~\ref{claim:nonvanishing-pullback-tangent} contradicts \eqref{eq:vanishing-tangent-functional-pullback}. Hence \(\dim Y=0\), so \(Y\) is trivial and \(\iota_1\colon X\to X\times Y\) is an isomorphism. Since \(\beta\circ\iota_1=i\), one has
\(
i\colon (X,\Theta)\cong(J,\Theta_{\Gamma}).
\)
Since \(f\circ i=\id_X\) and \(i\) is an isomorphism, one has \(f=i^{-1}\). Hence
\[
f\colon(J,\Theta_{\Gamma})\cong(X,\Theta)
\]
is an isomorphism of principally polarized tropical abelian varieties.
\end{proof}

\begin{proof}[Proof of Claim~\ref{claim:nonvanishing-pullback-tangent}]
Keep the notation from Step~\textup{(vi)}. By \eqref{eq:chi-v-local-map} and the construction of \(\chi\), one has
\[
d_v\chi\bigl(\widetilde{\boldsymbol{u}}_{\tau,v}\bigr)
\neq
0
\]
in \(T_{\chi(v)}^{\Z}X\). Under the canonical identification
\(
T_{\chi(v)}^{\Z}X\cong(\Lambda_X)^*,
\)
there exists \(\eta\in\Lambda_X\) such that
\[
d_v\chi\bigl(\widetilde{\boldsymbol{u}}_{\tau,v}\bigr)(\eta)
\neq
0.
\]
Let \(\omega_0\in H^0(X,\Omega_X)\) correspond to \(\eta\) under the canonical identification \(H^0(X,\Omega_X)\cong\Lambda_X\). Since \(\Omega_X\) is a constant sheaf, the germ \((\omega_0)_{\chi(v)}\) corresponds to \(\eta\). Hence \eqref{eq:differential-local-representative} gives
\[
\widetilde{\boldsymbol{u}}_{\tau,v}
\bigl((\chi^*\omega_0)_v\bigr)
=
d_v\chi\bigl(\widetilde{\boldsymbol{u}}_{\tau,v}\bigr)
\bigl((\omega_0)_{\chi(v)}\bigr)
\neq
0.
\]
This proves the claim.
\end{proof}

\section{Tropical Matsusaka--Ran criterion and the main theorem}
\label{sec:relation-two-criteria}
This section proves the tropical Matsusaka--Ran criterion, Theorem~\ref{thm:tropical-matsusaka-ran-criterion}, and the main theorem, Theorem~\ref{thm:main-equivalence}. Subsection~\ref{subsec:pontryagin-products-integral-tori} recalls tropical Pontryagin products on integral tori. Subsection~\ref{subsec:sumi-theta-functions} records the facts on tropical theta functions and regular sections needed below, and Subsection~\ref{subsec:preparations-two-criteria} prepares for the proof in Subsection~\ref{subsec:proof-tropical-matsusaka-ran}. Throughout this section, let \((X,\Theta)\) be a \(g\)-dimensional polarized tropical abelian variety with \(g\ge 2\).
\subsection{Pontryagin products on integral tori}
\label{subsec:pontryagin-products-integral-tori}

This subsection follows \cite[Sec.~6]{GS23} and recalls some basic properties of tropical Pontryagin products on integral tori. Fix \(p,q\in\Z_{\ge0}\).

\begin{defi}\label{def:pontryagin-product-integral-tori}
Let
\(
\mu_X\colon X\times X\to X
\)
be the group law. The tropical Pontryagin product on \(X\) is defined by
\[
\alpha\star\beta\coloneqq \mu_{X,*}(\alpha\times\beta),
\]
where \(\alpha\) and \(\beta\) are either tropical cycles on \(X\) or tropical homology classes on \(X\). On tropical cycles, this gives a homomorphism
\[
\star\colon Z_p(X)\otimes_{\Z}Z_q(X)\longrightarrow Z_{p+q}(X).
\]
Let \(p^{\prime},q^{\prime}\in\Z_{\ge0}\). On tropical homology classes, this gives a homomorphism
\[
\star\colon H_{p,q}(X)\otimes_{\Z}H_{p',q'}(X)
\longrightarrow H_{p+p',q+q'}(X).
\]
\end{defi}

\begin{prop}[{\cite[Prop.~6.2]{GS23}}]\label{prop:cycle-class-pontryagin-product}
The tropical cycle class map is compatible with tropical Pontryagin products. Equivalently, the {following} diagram {is commutative}{:}
\begin{equation}\label{diag:cycle-class-pontryagin-product}
\begin{tikzcd}
Z_p(X)\otimes_{\Z}Z_q(X)
\arrow[r,"\star"]
\arrow[d,"\cyc_X\otimes\cyc_X"']
&
Z_{p+q}(X)
\arrow[d,"\cyc_X"]
\\
H_{p,p}(X)\otimes_{\Z}H_{q,q}(X)
\arrow[r,"\star"]
&
H_{p+q,p+q}(X)
\end{tikzcd}
\end{equation}
\end{prop}

Let \(g=\dim X\). Retain the \(\Z\)-bases
\(\{\lambda_i^{\prime}\}_{i=1}^{g}\) of \(\Lambda_X^{\prime}\) and
\(\{\eta_i\}_{i=1}^{g}\) of \(\Lambda_X\) fixed in
Subsection~\ref{subsec:tropical-cohomology}, together with their dual bases
\(\{(\lambda_i^{\prime})^{*}\}_{i=1}^{g}\) and
\(\{\eta_i^{*}\}_{i=1}^{g}\). Under
\[
H_{p,p}(X)
\cong
\bigwedge^p\Lambda_X'\otimes\bigwedge^p(\Lambda_X)^*,
\]
the tropical Pontryagin product is computed by exterior products in each tensor factor. Let
\[
\alpha
=
\sum_{i,j=1}^{g}M_{i,j}\lambda_i'\otimes\eta_j^{*}
\in H_{1,1}(X),
\qquad
M=(M_{i,j})\in M_g(\Z).
\]
Then
\begin{equation}\label{eq:pontryagin-product-g-minus-one-minors}
\underbrace{\alpha\star\cdots \star\alpha}_{g-1\ \mathrm{factors}}
=
(g-1)! \sum_{i,j=1}^{g}
\det(M_{\widehat{i},\widehat{j}})\,
\lambda'_{\widehat{i}}\otimes\eta^*_{\widehat{j}},
\end{equation} where \(M_{\widehat{i},\widehat{j}}\) is the matrix obtained from \(M\) by deleting the \(i\)-th row and the \(j\)-th column, and
\[
\lambda'_{\widehat{i}}
\coloneqq
\lambda_1'\wedge\cdots\wedge
\widehat{\lambda_i'}\wedge\cdots\wedge\lambda_g',
\qquad
\eta^*_{\widehat{j}}
\coloneqq
\eta_1^{*}\wedge\cdots\wedge
\widehat{\eta_j^{*}}\wedge\cdots\wedge\eta_g^{*}.
\]

\begin{lemma}\label{lem:pushforward-pontryagin-product}
Let \(X_1\) and \(X_2\) be integral tori, let \(f\colon X_1\to X_2\) be a homomorphism of integral tori, and let \(\alpha\in Z_p(X_1)\) and \(\beta\in Z_q(X_1)\). Then the following hold.
\begin{enumerate}[label=\textup{(\roman*)}]
\item One has \(f_*(\alpha\star\beta)=f_*\alpha\star f_*\beta\) in \(Z_{p+q}(X_2)\).
\item One has \(f_*\cyc_{X_1}(\alpha)\star f_*\cyc_{X_1}(\beta)=\cyc_{X_2}\bigl(f_*(\alpha\star\beta)\bigr)\) in \(H_{p+q,p+q}(X_2)\).
\end{enumerate}
\end{lemma}

\begin{proof}
For \textup{(i)}, let \(\mu_i\colon X_i\times X_i\to X_i\) be the group law for \(i=1,2\). Since \(f\) is a homomorphism of integral tori, one has
\[
f\circ\mu_1=\mu_2\circ(f\times f).
\]
Hence
\[
\begin{split}
f_*(\alpha\star\beta)
&=
f_*\bigl((\mu_1)_*(\alpha\times\beta)\bigr)
=
(\mu_2)_*\bigl((f\times f)_*(\alpha\times\beta)\bigr)\\
&=
(\mu_2)_*(f_*\alpha\times f_*\beta)
=
f_*\alpha\star f_*\beta.
\end{split}
\]
For \textup{(ii)}, Proposition~\ref{prop:GS19-toolkit}\textup{(iii)} gives
\[
f_*\cyc_{X_1}(\alpha)=\cyc_{X_2}(f_*\alpha),
\qquad
f_*\cyc_{X_1}(\beta)=\cyc_{X_2}(f_*\beta).
\]
By Proposition~\ref{prop:cycle-class-pontryagin-product}, one has
\[
\cyc_{X_2}(f_*\alpha)\star\cyc_{X_2}(f_*\beta)
=
\cyc_{X_2}(f_*\alpha\star f_*\beta).
\]
Together with \textup{(i)}, this gives  \textup{(ii)}.
\end{proof}

\subsection{Regular sections and symmetric bilinear forms}\label{subsec:sumi-theta-functions}
This subsection records the facts on regular sections and symmetric bilinear forms needed below. Let \(X=(\Lambda_X,\Lambda_X^{\prime},[\,\cdot,\cdot\,]_X)\) be a \(g\)-dimensional integral torus. For \(D\in\Div(X)\), let \(Q_D\) denote the symmetric bilinear form associated with \(D\), as defined in Subsection~\ref{subsec:integral-tori-polarizations}. 
\begin{defi}\label{def:effective-tropical-Cartier-divisor}
A tropical Cartier divisor \(D\in\Div(X)\) is called effective if it is locally represented by the minimum of finitely many integral affine functions.
\end{defi}
\begin{remark}\label{rem:min-convention-theta}
The definition of effectiveness above follows from the min-convention in \cite[Rem.~3.4]{GS23}. In particular, a tropical Cartier divisor \(D\) {on \(X\)} is effective if and only if its associated tropical cycle \([D]\) is effective; see \cite[Def.~3.1]{GS23} and the discussion following it.
\end{remark}
\begin{defi}\label{def:regular-section-theta-function}
Let \(D\in\Div(X)\), and let \(s\) be a global section of \(\MO_X(D)\).
Let \(D_s\) be the tropical Cartier divisor associated with \(s\).
We call \(s\) a global regular section of \(\MO_X(D)\) if \(D_s\) is effective.
\end{defi}

\begin{prop}\label{prop:sumi-global-sections}
Let \(D\in\Div(X)\), and let \(Q_D\) be the symmetric bilinear form associated with \(D\). Then the following hold.
\begin{enumerate}[label=\textup{(\roman*)}, itemsep=2pt]
\item If \(Q_D\) has a negative eigenvalue, then \(\MO_X(D)\) has no global regular section.
\item If \(Q_D\) is positive definite, then \(\MO_X(D)\) has at least one global regular section.
\end{enumerate}
\end{prop}

\begin{proof}
Assertion~\textup{(i)} follows from \cite[Rem.~29]{Sum21}, and
assertion~\textup{(ii)} follows from \cite[Sec.~3.4]{Sum21}.
(Both assertions remain valid under the min-convention used here.)
\end{proof}

\begin{lemma}\label{lem:effective-divisor-semipositive}
Let \(D\in\Div(X)\) be effective. Then \(Q_D\) is symmetric positive semidefinite.
\end{lemma}

\begin{proof}
Since \(D\) is effective, \(\MO_X(D)\) has a global regular section whose associated tropical Cartier divisor is \(D\).
If \(Q_D\) had a negative eigenvalue, then Proposition~\ref{prop:sumi-global-sections}\textup{(i)} would imply that \(\MO_X(D)\) has no global regular section, a contradiction.
Hence \(Q_D\) has no negative eigenvalue.
Since \(Q_D\) is symmetric, it is positive semidefinite.
\end{proof}

\subsection{Preparations for the Matsusaka--Ran criterion}
\label{subsec:preparations-two-criteria}

This subsection prepares for Subsection~\ref{subsec:proof-tropical-matsusaka-ran}. The notation follows Subsections~\ref{subsec:tropical-cohomology} and~\ref{subsec:pontryagin-products-integral-tori}.
Recall from Subsection~\ref{subsec:tropical-cohomology} that \(\Z\)-bases
\(\{\lambda_i^{\prime}\}_{i=1}^{g}\) of \(\Lambda_X^{\prime}\) and
\(\{\eta_i\}_{i=1}^{g}\) of \(\Lambda_X\) have been fixed, with dual bases
\(\{(\lambda_i^{\prime})^{*}\}_{i=1}^{g}\) and
\(\{\eta_i^{*}\}_{i=1}^{g}\), respectively. As in \eqref{eq:chern-class-matrix}, we write
\[
c_1(\MO_X(\Theta))
=
\sum_{i,j=1}^{g}E^\Theta_{i,j}(\lambda_i')^*\otimes\eta_j,
\qquad
E^\Theta=(E^\Theta_{i,j})\in M_g(\Z).
\]

Let \(\Gamma\) be a smooth tropical curve of genus $g(\Gamma)$, and let
\(\chi\colon\Gamma\to X\) be a morphism of rational polyhedral spaces. Set
\(C\coloneqq \chi_*[\Gamma]\). Fix a basepoint \(q\in\Gamma\), let
\(\Theta_{\Gamma}\) be the canonical principal polarization on
\(\Jac(\Gamma)\), and let
\(\phi_q\colon\Gamma\to\Jac(\Gamma)\) be the tropical Abel--Jacobi map with
basepoint \(q\). The map \(\phi_q\) is a morphism of rational polyhedral
spaces; see \cite[Lem.~6.3]{MZ08}. Hence
\((\phi_q)_*[\Gamma]\) is a tropical \(1\)-cycle on \(\Jac(\Gamma)\).

For \(1\le p\le g(\Gamma)\), let
\[
\Phi_q^p\colon \Gamma^p\to\Jac(\Gamma),
\qquad
(x_1,\ldots,x_p)\longmapsto
\phi_q(x_1)+\cdots+\phi_q(x_p),
\]
and set
\[
\widetilde W_p\coloneqq \Phi_q^p(\Gamma^p).
\]
By \cite[Prop.~8.3]{GS23}, \(\widetilde W_p\) has a fundamental cycle \([\widetilde W_p]\), and
\begin{equation}\label{eq:pushforward-Gamma-p-Wp}
(\Phi_q^p)_*[\Gamma^p]=p![\widetilde W_p].
\end{equation}
Moreover, \cite[Cor.~8.4]{GS23} gives the following description of \([\widetilde W_p]\) in terms of the tropical Pontryagin product:
\begin{equation}\label{eq:Wp-pontryagin-product-W1}
p! [\widetilde W_p]
=
\underbrace{[\widetilde W_1]\star\cdots \star[\widetilde W_1]}_{p\ \mathrm{factors}},
\end{equation}
where \(\star\) denotes the tropical Pontryagin product on \(\Jac(\Gamma)\).

\medskip
We recall Theorem~\ref{Thm:Poincare}.
Set \(J\coloneqq\Jac(\Gamma)\). Choose a \(\Z\)-basis
\(\{\alpha_k\}_{k=1}^{g(\Gamma)}\) of \(H_1(\Gamma,\Z)\), and set
\[
\omega_k\coloneqq\kappa_{\Theta_\Gamma}(\alpha_k),
\qquad
1\le k\le g(\Gamma).
\]
Since \(\Theta_\Gamma\) is a principal polarization,
\(\{\omega_k\}_{k=1}^{g(\Gamma)}\) is a \(\Z\)-basis of
\(H^0(\Gamma,\Omega_\Gamma)\). Let
\(\{\omega_k^*\}_{k=1}^{g(\Gamma)}\) be the dual basis.

Let \(t_{-\chi(q)}\colon X\to X\) be the translation on \(X\) by
\(-\chi(q)\). Let \(f\colon J\to X\) be the homomorphism of integral tori satisfying
\(f\circ\phi_q=t_{-\chi(q)}\circ\chi\) as in
Proposition~\ref{prop:universal-property}. Set
\[
C^{\prime}\coloneqq(t_{-\chi(q)})_*C.
\]
Since \((\phi_q)_*[\Gamma]=[\widetilde W_1]\), one has
\begin{equation}\label{eq:f-pushforward-W1}
f_*[\widetilde W_1]=C^{\prime}.
\end{equation}

By Theorem~\ref{Thm:Poincare} and Lemma~\ref{lem:GS25-normal-form}\textup{(ii)}, one has
\[
\cyc_J([\widetilde W_1])
=
\sum_{k=1}^{g(\Gamma)}
\alpha_k\otimes\omega_k^{*}.
\]
Hence, by Definition~\ref{def:T-cycle-class}, with respect to the \(\Z\)-bases fixed above, one has
\[
T_{[\widetilde W_1]}
=
\kappa_{\Theta_\Gamma}^{-1}.
\]
By \eqref{eq:f-pushforward-W1} and Lemma~\ref{lem:T-pushforward}, there exists a matrix \(M^{\prime}=(M^{\prime}_{i,j})\in M_g(\Z)\) such that
\begin{equation}\label{eq:T-Cprime-via-f}
\begin{aligned}
T_{C^{\prime}}
=
f_\sharp\circ\kappa_{\Theta_\Gamma}^{-1}\circ f^\sharp
\colon(\Lambda_X)_{\R}
&\longrightarrow
(\Lambda_X^{\prime})_{\R},\\
v
&\longmapsto
\sum_{i,j=1}^{g}
M^{\prime}_{i,j}\eta_j^*(v)\lambda_i'.
\end{aligned}
\end{equation}

With the notation of \eqref{eq:pontryagin-product-g-minus-one-minors}, let
\(\varepsilon\in\{\pm1\}\) be the sign fixed in
{\eqref{eq:fundamental-cycle-class}}, and define
\[
\operatorname{adj}(M^{\prime})_{i,j}
\coloneqq
(-1)^{i+j}\det(M^{\prime}_{\widehat{j},\widehat{i}}),
\qquad
m_{C^{\prime}}\coloneqq\varepsilon\det(M^{\prime}).
\]

\begin{lemma}\label{lem:C-p-effective-cycle}
Assume that \(C^{\prime}\) is a spanning curve on \(X\). Then \(g\le g(\Gamma)\). For an integer \(p\) with \(1\le p\le g\), define
\begin{equation}\label{eq:C-p-definition}
C^{\prime[p]}\coloneqq f_*[\widetilde W_p].
\end{equation}
Then the following hold.
\begin{enumerate}[label=\textup{(\roman*)}]
\item The tropical cycle \(C^{\prime[p]}\) is a nonzero effective tropical \(p\)-cycle on \(X\).
\item One has
\begin{equation}\label{eq:cycle-class-C-p-pontryagin}
p! \cyc_X(C^{\prime[p]})
=
\underbrace{\cyc_X(C^{\prime})\star\cdots\star\cyc_X(C^{\prime})}_{p\ \mathrm{factors}}.
\end{equation}
\item Assume that \(p=g-1\). There exists an effective tropical Cartier divisor
\(D_{C^{\prime}}\) on \(X\) such that the associated tropical cycle
\([D_{C^{\prime}}]=C^{\prime[g-1]}\){.} {Every} such tropical Cartier divisor is ample. Moreover,
\begin{equation}\label{eq:c1-DC-adjM-main}
c_1(\MO_X(D_{C^{\prime}}))
=
\sum_{i,j=1}^{g}
\varepsilon(-1)^{i+j}
\det(M^{\prime}_{\widehat{i},\widehat{j}})
(\lambda_i^{\prime})^*\otimes\eta_j.
\end{equation}
With the notation in \eqref{eq:chern-class-matrix}, one has
\begin{equation}\label{eq:EDC-adjM}
E^{D_{C^{\prime}}}
=
\varepsilon\operatorname{adj}(M^{\prime})^{\mathsf T}.
\end{equation}
\item Assume that \(p=g-1\), and let \(D_{C^{\prime}}\) be as in \textup{(iii)}. Then
\begin{equation}\label{eq:DC-top-self-intersection}
(D_{C^{\prime}}^g)_X
=
g!m_{C^{\prime}}^{g-1},
\end{equation}
and
\begin{equation}\label{eq:DC-power-cycle-class}
([D_{C^{\prime}}])^{g-1}
\equiv_{\hom}
(g-1)! m_{C^{\prime}}^{g-2}C^{\prime}.
\end{equation}
\end{enumerate}
\end{lemma}

\begin{proof}
By \eqref{eq:f-pushforward-W1}, one has
\(
|C^{\prime}|
\subseteq
f(\widetilde W_1)
\subseteq
f(J).
\)
Since \(C^{\prime}\) is a spanning curve, Lemma~\ref{lem:spanning-implies-surjective} shows that \(f\colon J\to X\) is surjective. Therefore \(g\le g(\Gamma)\).

\begin{enumerate}[label=\textup{(\roman*)}]
\item
Applying \(f_*\) to \eqref{eq:Wp-pontryagin-product-W1} and using
Lemma~\ref{lem:pushforward-pontryagin-product}\textup{(i)} repeatedly gives
\begin{equation}\label{eq:fpush-Wp-Cprime-pontryagin}
p! C^{\prime[p]}
=
\underbrace{C^{\prime}\star\cdots\star C^{\prime}}_{p\ \mathrm{factors}}.
\end{equation}
{Define} {the addition map} {\(\mu_p\colon X^p\to X\) by}
\[
(y_1,\ldots,y_p)\longmapsto y_1+\cdots+y_p.
\] By the definition of the Pontryagin product,
\[
p! C^{\prime[p]}
=
(\mu_p)_*(C^{\prime}\times\cdots\times C^{\prime}).
\]
Since \(C^{\prime}\times\cdots\times C^{\prime}\) is an effective tropical \(p\)-cycle and
\(\mu_p\) is locally integral affine,
\[
\dim\mu_p\bigl(|C^{\prime}\times\cdots\times C^{\prime}|\bigr)\le p.
\]
{Choose edge-endpoint pairs \((\tau_1,v_1),\ldots,(\tau_p,v_p)\) such that the following tangent directions are \(\R\)-linearly independent:}
\[
\boldsymbol{u}_{\tau_1,v_1},\ldots,\boldsymbol{u}_{\tau_p,v_p}.
\]
{Such pairs exist by Definition~\ref{def:spanning-curve}, since \(C^{\prime}\) is a spanning curve on \(X\).} Then
\[
\tau_1\times\cdots\times\tau_p
\subseteq
|C^{\prime}|^p
=
|C^{\prime}\times\cdots\times C^{\prime}|.
\]
On
\(\relint(\tau_1)\times\cdots\times\relint(\tau_p)\),
the differential of \(\mu_p\) has rank \(p\), so
\(
\dim\mu_p(\tau_1\times\cdots\times\tau_p)=p.
\)
Therefore
\[
\dim\mu_p\bigl(|C^{\prime}\times\cdots\times C^{\prime}|\bigr)=p.
\]
Since \(C^{\prime}\times\cdots\times C^{\prime}\) is effective,
Definition~\ref{def:proper-pushforward} shows that
\[
(\mu_p)_*(C^{\prime}\times\cdots\times C^{\prime})
\]
is a nonzero effective tropical \(p\)-cycle. Hence
\(C^{\prime[p]}\) is a nonzero effective tropical \(p\)-cycle.

\item
Apply \(\cyc_X\) to \eqref{eq:fpush-Wp-Cprime-pontryagin} and use
Proposition~\ref{prop:cycle-class-pontryagin-product} repeatedly.

\item
Assume that \(p=g-1\). By \textup{(i)},
\(C^{\prime[g-1]}\) is an effective tropical cycle of codimension one on \(X\).
By \cite[Sec.~4.3 and Thm.~4.5]{GS21}, there exists a tropical Cartier divisor
\(D_{C^{\prime}}\) such that
\([D_{C^{\prime}}]=C^{\prime[g-1]}\).
Since \(C^{\prime[g-1]}\) is effective,
Remark~\ref{rem:min-convention-theta} shows that \(D_{C^{\prime}}\) is effective.
Hence Lemma~\ref{lem:effective-divisor-semipositive} shows that
\(Q_{D_{C^{\prime}}}\) is symmetric positive semidefinite.
Since \(f\colon J\to X\) is surjective,
Lemma~\ref{lem:homomorphism-basic}\textup{(i)} implies that
\(f^\sharp\colon\Lambda_X\to H^0(\Gamma,\Omega_\Gamma)\) is injective.
By \eqref{eq:T-Cprime-via-f} {and \eqref{eq:dual-hom-pairing-compat}}, for every nonzero \(v\in(\Lambda_X)_{\R}\),
\[
[v,T_{C^{\prime}}(v)]_X
=
[f^\sharp(v),\kappa_{\Theta_\Gamma}^{-1}(f^\sharp(v))]_J
=
Q_{\Theta_\Gamma}(\kappa_{\Theta_\Gamma}^{-1}(f^\sharp(v)),
\kappa_{\Theta_\Gamma}^{-1}(f^\sharp(v)))
>0.
\]
Thus \(T_{C^{\prime}}\) is an isomorphism over \(\R\). By \eqref{eq:T-Cprime-via-f}, it follows that
\begin{equation}\label{eq:det-Mprime-nonzero}
\det(M^{\prime})\neq0.
\end{equation}
Since \(H_{g-1,g-1}(X)\) is torsion-free, \textup{(ii)} and
\eqref{eq:pontryagin-product-g-minus-one-minors} give
\[
\cyc_X([D_{C^{\prime}}])
=
\sum_{i,j=1}^{g}
\det(M^{\prime}_{\widehat{i},\widehat{j}})
\lambda_{\widehat{i}}'\otimes\eta_{\widehat{j}}^*.
\]
By Proposition~\ref{prop:cyc-compat}, one has
\[
\cyc_X([D_{C^{\prime}}])
=
c_1(\MO_X(D_{C^{\prime}}))
\cap
\cyc_X([X]).
\]
Using \eqref{eq:fundamental-cycle-class} and \eqref{eq:cap-product-integral-tori}, this becomes \eqref{eq:c1-DC-adjM-main}, and hence
\eqref{eq:EDC-adjM}.
Together with \eqref{eq:det-Mprime-nonzero}, this gives
\[
\det(E^{D_{C^{\prime}}})
=
\varepsilon^g\det(M^{\prime})^{g-1}
\neq0.
\]
Thus \(Q_{D_{C^{\prime}}}\) is nondegenerate. Since it is positive semidefinite, it is positive definite, so \(D_{C^{\prime}}\) is ample.

\item
By \eqref{eq:EDC-adjM},
\[
\varepsilon\det(E^{D_{C^{\prime}}})
=
\varepsilon^{g+1}\det(\operatorname{adj}(M^{\prime}))
=
\varepsilon^{g+1}\det(M^{\prime})^{g-1}
=
m_{C^{\prime}}^{g-1}.
\]
Hence Lemma~\ref{lem:top-self-intersection-det}\textup{(i)} gives
\eqref{eq:DC-top-self-intersection}.
By \eqref{eq:c1-DC-adjM-main} and
\eqref{eq:cup-product-integral-tori},
\[
\begin{aligned}
c_1(\MO_X(D_{C^{\prime}}))^{g-1}
&=(g-1)!\sum_{i,j=1}^{g}\det\bigl(E^{D_{C^{\prime}}}_{\widehat{i},\widehat{j}}\bigr)
(\lambda_1^{\prime})^*\wedge\cdots\wedge\widehat{(\lambda_i^{\prime})^*}\wedge\cdots\wedge(\lambda_g^{\prime})^*\\
&\qquad\otimes
\eta_1\wedge\cdots\wedge\widehat{\eta_j}\wedge\cdots\wedge\eta_g.
\end{aligned}
\]
By Proposition~\ref{prop:cyc-compat}, one has
\[
\cyc_X\!\left(([D_{C^{\prime}}])^{g-1}\right)
=
\bigl(c_1(\MO_X(D_{C^{\prime}}))\bigr)^{g-1}
\cap
\cyc_X([X]).
\]
{By the two preceding equalities, \eqref{eq:cap-product-integral-tori}, and {\eqref{eq:fundamental-cycle-class}},}
\[
\cyc_X\!\left(([D_{C^{\prime}}])^{g-1}\right)
=
(g-1)! \sum_{i,j=1}^{g}
\varepsilon(-1)^{i+j}
\det\bigl(E^{D_{C^{\prime}}}_{\widehat{i},\widehat{j}}\bigr)
\lambda_i^{\prime}\otimes\eta_j^*.
\]
{By \eqref{eq:EDC-adjM},}
the identity
\(
\operatorname{adj}\bigl(\operatorname{adj}(M^{\prime})^{\mathsf T}\bigr)
=
\det(M^{\prime})^{g-2}(M^{\prime})^{\mathsf T}
\)
gives
\[
\varepsilon(-1)^{i+j}
\det\bigl(E^{D_{C^{\prime}}}_{\widehat{i},\widehat{j}}\bigr)
=
m_{C^{\prime}}^{g-2}M^{\prime}_{i,j}.
\]
This proves
\eqref{eq:DC-power-cycle-class}.\qedhere
\end{enumerate}
\end{proof}

\subsection{Proof of the tropical Matsusaka--Ran criterion and the main theorem}
\label{subsec:proof-tropical-matsusaka-ran}
The key step in the proof of the tropical Matsusaka--Ran criterion and Theorem~\ref{thm:main-equivalent-criteria} is the following proposition.

\begin{prop}
\label{prop:equivalence-conditions-two-criteria}
Let \(C\) be an effective tropical \(1\)-cycle on \(X\) such that \(|C|\) is connected. Then the following are equivalent.
\begin{enumerate}[label=\textup{(\roman*)}, itemsep=2pt]
\item The tropical \(1\)-cycle \(C\) is a spanning curve on \(X\), and \((\Theta\cdot C)_X=g\).
\item The polarization \(\Theta\) is principal, and \((g-1)! C\equiv_{\hom}([\Theta])^{g-1}\).
\end{enumerate}
\end{prop}

\begin{proof}
Assume \textup{(ii)}. By Lemma~\ref{lem:C-must-be-spanning-in-proof}, \(C\) is a spanning curve on \(X\). Since \(\Theta\) is a principal polarization, Lemma~\ref{lem:top-self-intersection-det}\textup{(iii)} gives \((\Theta^g)_X=g!\). Since \((g-1)! C\equiv_{\hom}([\Theta])^{g-1}\), one has
\begin{equation*}
(g-1)!(\Theta\cdot C)_X
=
\bigl(\Theta\cdot([\Theta])^{g-1}\bigr)_X
=
(\Theta^g)_X
=
g!.
\end{equation*}
Thus \((\Theta\cdot C)_X=g\), and \textup{(i)} holds.

Conversely, assume \textup{(i)}. By Lemma~\ref{lem:cycle-from-smooth-curve}, there exist a smooth tropical curve \(\Gamma\) and a morphism of rational polyhedral spaces \(\chi\colon\Gamma\to X\) such that \(\chi_*[\Gamma]=C\). We use the notation of Subsection~\ref{subsec:preparations-two-criteria} for this choice of \(\Gamma\) and \(\chi\). Since \(C^{\prime}\) is a translate of \(C\), it is also a spanning curve on \(X\). Moreover, translation acts trivially on tropical homology, so
\[
C^{\prime}\equiv_{\hom}C.
\]
Hence Proposition~\ref{prop:cyc-compat} gives
\[
(\Theta\cdot C^{\prime})_X=(\Theta\cdot C)_X=g.
\]
Let \(D_{C^{\prime}}\) be the ample tropical Cartier divisor constructed in Lemma~\ref{lem:C-p-effective-cycle}, so that \([D_{C^{\prime}}]=C^{\prime[g-1]}\),
\[
(D_{C^{\prime}}^g)_X=g!m_{C^{\prime}}^{g-1},
\quad \text{and} \quad
([D_{C^{\prime}}])^{g-1}
\equiv_{\hom}
(g-1)! m_{C^{\prime}}^{g-2}C^{\prime}.
\]
Using \((\Theta\cdot C^{\prime})_X=g\) and Proposition~\ref{prop:cyc-compat}, this gives
\[
(\Theta\cdot D_{C^{\prime}}^{g-1})_X
=
\left(\Theta\cdot ([D_{C^{\prime}}])^{g-1}\right)_X
=
(g-1)!m_{C^{\prime}}^{g-2}(\Theta\cdot C^{\prime})_X
=
g!m_{C^{\prime}}^{g-2}.
\]
Since \(\Theta\) and \(D_{C^{\prime}}\) are ample, Proposition~\ref{prop:tropical-KT} gives
\(
(\Theta\cdot D_{C^{\prime}}^{g-1})_X>0.
\)
Moreover, \((D_{C^{\prime}}^g)_X>0\). Hence \(m_{C^{\prime}}>0\).

By Lemma~\ref{lem:top-self-intersection-det}\textup{(ii)} and \eqref{eq:DC-top-self-intersection}, one has
\[
(\Theta^g)_X\ge g!,
\qquad
(D_{C^{\prime}}^g)_X=g!m_{C^{\prime}}^{g-1}.
\]
Therefore Proposition~\ref{prop:tropical-KT} gives
\[
g!m_{C^{\prime}}^{g-2}
=
(\Theta\cdot D_{C^{\prime}}^{g-1})_X
\ge
(\Theta^g)_X^{1/g}(D_{C^{\prime}}^g)_X^{(g-1)/g}
\ge
g!m_{C^{\prime}}^{(g-1)^2/g}.
\]
Hence \(m_{C^{\prime}}\le1\). Since \(m_{C^{\prime}}\in\Z_{>0}\), one has
\(
m_{C^{\prime}}=1.
\)
Therefore
\[
(D_{C^{\prime}}^{g-1}\cdot\Theta)_X=g!.
\]
By Corollary~\ref{cor:principal-intersection-equality}, \(\Theta\) is a principal polarization and
\begin{equation}\label{eq:DC-homologous-Theta}
[D_{C^{\prime}}]\equiv_{\hom}[\Theta].
\end{equation}
By \eqref{eq:DC-power-cycle-class}, \eqref{eq:DC-homologous-Theta}, and \(m_{C^{\prime}}=1\), one has
\[
(g-1)! C^{\prime}
\equiv_{\hom}
([D_{C^{\prime}}])^{g-1}
\equiv_{\hom}
([\Theta])^{g-1}.
\]
Since \(C^{\prime}\equiv_{\hom}C\), it follows that
\[
(g-1)! C\equiv_{\hom}([\Theta])^{g-1}.
\]
Thus \textup{(ii)} holds.
\end{proof}

We first state the tropical Matsusaka--Ran criterion.

\begin{thm}[Tropical Matsusaka--Ran criterion]\label{thm:tropical-matsusaka-ran-criterion}
Let \(C\) be an effective tropical \(1\)-cycle on \(X\) such that \(|C|\) is connected. Assume that \(C\) is a spanning curve on \(X\) and
\(
(\Theta\cdot C)_X=g.
\)
Then there exist a smooth tropical curve \(\Gamma\) of genus \(g\) and a morphism of rational polyhedral spaces \(\chi\colon\Gamma\to X\) such that \(\chi_*[\Gamma]=C\) and the homomorphism {\(f\colon\Jac(\Gamma)\to X\)} induced by \(\chi\) is an isomorphism {of principally polarized tropical abelian varieties}{:}
\[
\bigl(\Jac(\Gamma),\Theta_\Gamma\bigr)\cong(X,\Theta).
\]
\end{thm}

\begin{proof}
By Proposition~\ref{prop:equivalence-conditions-two-criteria}, the polarization \(\Theta\) is principal and
\[
(g-1)!C\equiv_{\hom}([\Theta])^{g-1}.
\]
The assertion now follows from Theorem~\ref{thm:tropical-matsusaka-criterion}. 
\end{proof}

We now state the main theorem of this paper.

\begin{thm}\label{thm:main-equivalence}
Let \(C\) be an effective tropical \(1\)-cycle on \(X\) such that \(|C|\) is connected. Then the following are equivalent.

\begin{enumerate}[label=\textup{(\roman*)}]
\item The tropical \(1\)-cycle \(C\) is a spanning curve on \(X\), and \((\Theta\cdot C)_X=g\).

\item The polarization \(\Theta\) is principal, and \((g-1)!C\equiv_{\hom}([\Theta])^{g-1}\).
 
\item There exist a smooth tropical curve \(\Gamma\) of genus \(g\) and a morphism of rational polyhedral spaces \(\chi\colon\Gamma\to X\) such that \(\chi_*[\Gamma]=C\) and the homomorphism {\(f\colon\Jac(\Gamma)\to X\)} induced by \(\chi\) is an isomorphism {of principally polarized tropical abelian varieties}{:}
\[
\bigl(\Jac(\Gamma),\Theta_\Gamma\bigr)\cong(X,\Theta),
\]
where \(\Theta_\Gamma\) denotes the canonical principal polarization on \(\Jac(\Gamma)\).
\end{enumerate}
\end{thm}

\begin{proof}
The equivalence of \textup{(i)} and \textup{(ii)} is Proposition~\ref{prop:equivalence-conditions-two-criteria}, and the implication \textup{(i)}\(\Longrightarrow\)\textup{(iii)} is Theorem~\ref{thm:tropical-matsusaka-ran-criterion}.

Assume \textup{(iii)}. Fix \(q\in\Gamma\), and let \(\phi_q\colon\Gamma\to\Jac(\Gamma)\) be the tropical Abel--Jacobi map. Since \(f\) is induced by \(\chi\), one has
\begin{equation}\label{eq:f-phi-q-chi}
f\circ\phi_q=t_{-\chi(q)}\circ\chi,
\end{equation}
where \(t_{-\chi(q)}\colon X\to X\) denotes translation by \(-\chi(q)\).
Since \(f\) is an isomorphism of principally polarized tropical abelian varieties, one has
\[
f_*\bigl(([\Theta_\Gamma])^{g-1}\bigr)
\equiv_{\hom}
([\Theta])^{g-1}.
\]
Applying \(f_*\) to Theorem~\ref{Thm:Poincare} and using \eqref{eq:f-pushforward-W1}, one obtains
\[
(g-1)!(t_{-\chi(q)})_*C
=
(g-1)!f_*[\widetilde W_1]
\equiv_{\hom}
([\Theta])^{g-1}.
\]
Since translation acts trivially on tropical homology, it follows that
\[
(g-1)!C\equiv_{\hom}([\Theta])^{g-1}.
\]
Thus \textup{(ii)} holds.
\end{proof}

\noindent\textbf{Use of AI.} The author used ChatGPT, including GPT-4o and GPT-5-series models, to search for references cited in this paper. The initial draft was completed in February 2026 with assistance from ChatGPT. ChatGPT was also used to check the mathematical correctness of the initial draft. Later revisions used ChatGPT, including GPT-6 Astra and GPT-5.6 Sol, for manuscript organization, revision, and proofreading. In particular, Section~\ref{sec:orthogonal-decompositions} was substantially revised with AI assistance. The author takes responsibility for the final manuscript.

\end{document}